\documentclass[a4paper,11pt]{amsart}
\usepackage{amsmath,amsthm,amssymb,mathtools}
\usepackage[mathscr]{eucal}
\usepackage{cite}
\usepackage{upgreek}

\usepackage[bookmarks,bookmarksnumbered]{hyperref}

\usepackage{enumerate}
\usepackage{mathrsfs}
\usepackage{blindtext}
\usepackage{scrextend}
\usepackage{enumitem}
\usepackage{bbm}
\addtokomafont{labelinglabel}{\sffamily}
\usepackage{color}
\usepackage[at]{easylist}

\numberwithin{equation}{section}

\theoremstyle{plain}
\newtheorem{main}{Theorem}
\newtheorem{mcor}[main]{Corollary}
\newtheorem{mprop}[main]{Proposition}

\newtheorem{theorem}{Theorem}[section]
\newtheorem{claim}[theorem]{Claim}
\newtheorem{lemma}[theorem]{Lemma}
\newtheorem{proposition}[theorem]{Proposition}
\newtheorem{corollary}[theorem]{Corollary}
\theoremstyle{definition}
\newtheorem{definition}[theorem]{Definition}
\newtheorem*{definition*}{Definition}

\newtheorem{remark}[theorem]{Remark}

\newtheorem{question}[theorem]{Question}

\newcommand{\cA}{\mathcal{A}}

\newcommand{\C}{\mathbb{C}}\newcommand{\cC}{\mathcal{C}}

\newcommand{\bF}{\mathbb{F}}

\newcommand{\bM}{\mathbb{M}}\newcommand{\cM}{\mathcal{M}}
\newcommand{\N}{\mathbb{N}}

\newcommand{\Z}{\mathbb{Z}}\newcommand{\cZ}{\mathcal{Z}}

\newcommand{\ot}{\otimes}

\newcommand{\Prob}{\operatorname{Prob}}

\newcommand{\bary}{\operatorname{bar}}

\newcommand{\dint}{\int^\oplus}

\newcommand{\eps}{\varepsilon}
\newcommand{\vphi}{\varphi}

\newcommand{\norm}[1]{\left\|#1\right\|}

\begin{document}

\title[Trace spaces of full free product $C^*$-algebras]
{Trace spaces of full free product $C^*$-algebras}

\author[A. Ioana]{Adrian Ioana}
\address{Department of Mathematics, University of California San Diego, 9500 Gilman Drive, La Jolla, CA 92093, USA}
\email{aioana@ucsd.edu}

\author[P. Spaas]{Pieter Spaas}
\address{Department of Mathematical Sciences, University of Copenhagen, Universitetsparken 5, DK-2100 Copenhagen \O, Denmark}
\email{pisp@math.ku.dk}

\author[I. Vigdorovich]{Itamar Vigdorovich}
\address{Department of Mathematics, University of California San Diego, 9500 Gilman Drive, La Jolla, CA 92093, USA}
\email{itamar.vigi@gmail.com}

\thanks{A.I. was supported by NSF grants FRG-DMS-1854074 and DMS-2153805.
P.S. was supported by MSCA Fellowship No. 101111079 from the European Union.
I.V. was supported by NSF postdoctoral fellowship 
DMS-2402368}

\begin{abstract}

We study the space of traces  associated with
arbitrary full free products of  unital, separable
$C^*$-algebras. We show that, unless certain basic obstructions
(which we fully characterize) occur, the space of traces always results
in the same object: the Poulsen simplex, that is, the unique infinite-dimensional metrizable Choquet
simplex whose extreme points are dense. Moreover, we show that whenever such a trace space is the Poulsen simplex, the extreme points are dense in the Wasserstein topology. 
Concretely for the case of groups,
we find that, unless 
the trivial character is isolated in the space of characters, 
the space of traces of any free product of non-trivial countable
groups is the Poulsen simplex.
Our main technical
contribution is a new perturbation result for pairs of von Neumann
subalgebras $(M_{1},M_{2})$ of a tracial von Neumann algebra $M$,
providing necessary conditions under which $M_{1}$ and a small unitary
perturbation of $M_{2}$ generate a II$_{1}$ factor.
\end{abstract}

\maketitle


\section{Introduction}
The main object of study in this paper is the space of tracial states,  or {\it traces}, of a $C^*$-algebra. The space of traces $\mathrm{T}(A)$ of a unital and separable $C^*$-algebra $A$ is a metrizable Choquet simplex \cite[Theorem~3.1.18]{Sa71} (see also \cite[Theorem 1.1]{BR24}), i.e., a compact convex set where every point is the barycenter of a unique Borel probabilty measure supported on its extreme points. 
Conversely, any metrizable Choquet simplex can be realized as the trace simplex of a unital, separable (and moreover simple and AF) $C^*$-algebra, see \cite[Theorem~5.1]{Go77} and  \cite[Theorem~3.10]{Bl80}.

The trace simplex has long been recognized as a useful invariant, playing a prominent role in the Elliott classification program for nuclear simple $C^*$-algebras, see, e.g., \cite{Ro02,Wi18,Wh23,CGSTW23}. 
When $A=C^*G$ is the full group $C^*$-algebra of a countable discrete group $G$, traces on $A$ are in one-to-one correspondence with traces on $G$, i.e., positive definite, conjugation-invariant functions $\varphi:G\rightarrow\mathbb C$ with $\varphi(e)=1$. Recently, traces on groups have found 
exciting applications to dynamics, representation theory, operator algebras, and group stability, see, e.g., \cite{Be07, BKKO14, Pe14, HS18, BH21, BBHP22, LV23}.

\subsection*{The trace simplex of free products}
It is well known that the category of unital $C^*$-algebras admits coproducts, implemented by the unital full free product construction, see, e.g., \cite{Av82,VDN92}.  For the rest of this paper, the free product $A=A_1*A_2$ of two unital $C^*$-algebras $A_1$ and $A_2$ will refer to their {\it unital full} free product.
The main objective of this paper is to understand $\mathrm{T}(A)$ in terms of $\text{T}(A_1)$ and $\text{T}(A_2)$.  
We note that another important invariant, the $K$-theory, has been much studied for free product $C^*$-algebras $A=A_1*A_2$ and shown to be completely determined by the $K$-theories of $A_1$ and $A_2$ (see, e.g., \cite{Cu82,Ge97,Th03,FG20}). 
In contrast with the situation of general $C^*$-algebras where any prescribed Choquet simplex can arise as the trace simplex, we find that, under mild assumptions on $\text{T}(A_1)$ and $\text{T}(A_2)$, the space of traces of $A$ is always the same object: the Poulsen simplex. Strikingly, this implies that, unlike $K$-theory, the trace space of a free product $A=A_1*A_2$ does not typically depend on the trace spaces of $A_1$ and $A_2$.

In our first main result we characterize exactly when the trace space of a 
free product $C^*$-algebra is the Poulsen simplex. 
In \cite{Po61}, Poulsen was the first to construct a metrizable Choquet simplex 
whose extreme points are dense. It was shown in \cite{LOS78} that  there exists a unique non-trivial (i.e., neither empty nor a singleton) Poulsen simplex up to affine homeomorphisms, now called \textit{the Poulsen simplex}. Moreover, the Poulsen simplex is universal in the sense that it admits every metrizable Choquet simplex as a closed face.   Hereafter, any non-trivial metrizable Poulsen simplex will be referred to as the Poulsen simplex. 

Before stating our first main result, we also recall that the GNS construction associates, to any trace $\varphi$ on a unital $C^*$-algebra $A$, a von Neumann algebra $M$ with a faithful normal trace $\tau$, together with a $*$-homomorphism $\pi:A\rightarrow M$ such that $\varphi=\tau\circ\pi$ and $\pi(A)''=M$. 
It is well-known that  $\vphi$ is an extreme point of $\mathrm{T}(A)$ 
if and only if $M$ is a factor. 
We say that $\vphi$ is $1$-dimensional if $M$ is $1$-dimensional, and similarly for finite or infinite dimensional.

\begin{main}\label{Poulsen} For $i=1,2$, let $A_i$ be a unital, separable $C^*$-algebra such that $\emph{T}(A_i)$ is non-empty and does not consist of a single $1$-dimensional trace. Let $A=A_1*A_2$ be their 
free product. Then the following are equivalent:
\begin{enumerate}
    \item $\emph{T}(A)$ is a Poulsen simplex.
    \item If $A_1$ (resp. $A_2$) has an isolated extreme 1-dimensional trace, then $A_2$ (resp. $A_1$) does not have an isolated extreme finite dimensional trace.
    \item $A$ does not have an isolated extreme finite dimensional trace.
\end{enumerate}
Moreover, if these equivalent conditions hold, then
the set of extreme, infinite dimensional traces is dense in $\emph{T}(A)$.
\end{main}

The assumption in Theorem~\ref{Poulsen} that $\mathrm{T}(A_i)$ does not consist of a single $1$-dimensional trace is necessary. If, for example, $A_1$ has a unique trace which is $1$-dimensional, then  the restriction map $\mathrm{T}(A)\ni \varphi\mapsto\varphi_{|A_2}\in\mathrm{T}(A_2)$ is an affine homeomorphism. But, as explained above, $\mathrm{T}(A_2)$ can be any metrizable Choquet simplex.
This assumption also implies that $\text{T}(A)$ is not a singleton, and thus if $\text{T}(A)$ is a Poulsen simplex, then it is the (non-trivial) Poulsen simplex.  Theorem \ref{Poulsen} therefore establishes the following lack of rigidity:  a large family of $C^*$-algebras, consisting of most free product $C^*$-algebras, all have the same trace simplex. 

In \cite{MR19}, Musat and R{\o}rdam studied the trace simplex of the 
free product of matrix algebras $\mathbb M_n(\mathbb C)*\mathbb M_n(\mathbb C)$. They showed that $\text{T}(\mathbb M_n(\mathbb C)*\mathbb M_n(\mathbb C))$  parametrizes the so-called  factorizable quantum channels $\mathbb M_n(\mathbb C)\rightarrow\mathbb M_n(\mathbb C)$. Additionally, they proved that the Poulsen simplex is a face of $\text{T}(\mathbb M_n(\mathbb C)*\mathbb M_n(\mathbb C))$ and asked whether $\text{T}(\mathbb M_n(\mathbb C)*\mathbb M_n(\mathbb C))$ is the Poulsen simplex, for $n\geq 3$.

Motivated by this question, Orovitz, Slutsky, and the third author recently showed in \cite{OSV23} that $\mathrm{T}(\bM_n(\C)\ast \bM_n(\C))$, for $n\geq 4$,  and  $\mathrm{T}(C^*\bF_m)$, for $2\leq m\leq\infty$, are the Poulsen simplex. However, their method does not extend to general free products, as it relies heavily on the structure of the matrix algebras (for which moreover the condition $n\geq 4$ is needed) and the free groups, respectively.

Theorem~\ref{Poulsen} significantly generalizes these results.
Indeed, condition {\it (2)} from Theorem~\ref{Poulsen} is 
satisfied if one of the algebras admits no 
finite dimensional trace. 

This condition is also
satisfied if both algebras are finite dimensional, but have no $1$-dimensional direct summands, allowing us to classify the 
 free products of finite dimensional $C^*$-algebras whose trace simplex is the Poulsen simplex:

\begin{mcor}\label{fin_dim_algebras}
    Let $A=A_1*A_2$ be the 
    free product of finite dimensional $C^*$-algebras $A_1$ and $A_2$. Then $\mathrm{T}(A)$ is the Poulsen simplex if and only if $A_1$ and $A_2$ have no 1-dimensional direct summands. 
   
\end{mcor}
 In particular, Corollary \ref{fin_dim_algebras} shows that $\text{T}(\mathbb M_n(\mathbb C)*\mathbb M_n(\mathbb C))$ is the Poulsen simplex, for every $n\geq 2$,
 thereby completely answering the question from \cite{MR19} discussed above.

Theorem~\ref{Poulsen} also yields a complete characterization of when the trace simplex of 
a free product of countable  groups is  the Poulsen simplex.

Given a countable discrete group $G$, recall that the space of traces on $G$, which we denote by $\text{Tr}(G)$, is naturally identified with $\text{T}(C^*G)$.

In this context, extreme traces are often called ``characters". Following \cite[Definition 7.1]{OSV23}, we say that the group $G$ has \textit{property (ET)} if the trivial character $1_G$ is isolated in the space of all characters of $G$.

\begin{mcor}\label{group C*-algebras}
Let $G=G_1*G_2$ be the free product of non-trivial countable discrete groups  $G_1$ and $G_2$. Then the following are equivalent:
\begin{enumerate}
    \item The space of traces on $G$ is the Poulsen simplex.
    \item $G_i$ does not have property \emph{(ET)}, for some $i\in\{1,2\}.$
    \item $G$ does not have property \emph{(ET)}.
\end{enumerate}
\end{mcor}

This result significantly generalizes \cite[Theorem 1.1]{OSV23}.

Indeed, since property (ET) passes to quotient groups by \cite[Lemma 7.2]{OSV23} and countable infinite abelain groups do not have property (ET), Corollary \ref{group C*-algebras} implies that $\text{Tr}(G_1*G_2)$ is the Poulsen simplex, for any non-trivial countable groups $G_1$ and $G_2$ such that $G_1$ admits an infinite abelian quotient.

Note that countable groups with Kazhdan's property (T) also satisfy property (ET).
It is shown in \cite{LSV23} that the trace simplex of Kazhdan groups is, in fact, as far as possible from being a Poulsen simplex—it is a Bauer simplex, meaning that the set of extreme points is closed.

\subsection*{Approximation in the Wasserstein topology}
To put our results above into a better perspective,
let us point out that the proofs show much more than that the trace space of the $C^*$-algebras in question is a Poulsen simplex. Recall that the trace space of a unital separable $C^*$-algebra admits, in addition to the weak$^*$-topology, a second natural topology, the so-called {\it Wasserstein topology} introduced by Biane and Voiculescu in \cite{BV01}.

\begin{definition}(see \cite[Section 1.2]{BV01})
Let $A$ be a unital, separable $C^*$-algebra, and $(a_n)_{n=1}^\infty$  a norm dense sequence in the unit ball of $A$.
For $i=1,2$, let $\varphi_i\in\text{T}(A)$, $(M_i,\tau_i)$ be a tracial von Neumann algebra and $\pi_i:A\rightarrow M_i$ be a unital $*$-homomorphism such that $\varphi_i=\tau_i\circ\pi_i$ and $\pi_i(A)''=M_i$. We define the {\it Wasserstein distance} $d_W(\varphi_1,\varphi_2)$ as the infimum of the quantity $\sum_{n=1}^\infty 2^{-n}\|\pi_1(a_n)-\pi_2(a_n)\|_{2,\tau}$, over all tracial von Neumann algebras $(M,\tau)$ such that $M_i\subset M$ and $\tau_{|M_i}=\tau_i$, for every $i=1,2$. 
The {\it Wasserstein topology} on $\text{T}(A)$ is the topology induced by $d_W$.
\end{definition}

It is easy to see that $d_W$ is a well-defined metric (e.g., by adapting \cite[Theorem 1.3]{BV01}), that the Wasserstein topology is independent of the choice of the sequence $(a_n)_{n=1}^\infty$, and that the Wasserstein topology is stronger than the weak$^*$-topology on $\text{T}(A)$. 

In general, the Wasserstein topology is strictly stronger than the weak$^*$-topology. By \cite[Proposition 1.7]{GJNS21} this holds if $A=C([-1,1])^{*m}$ is the 
free product of $m$ copies of $C([-1,1])$, for $m\geq 2$. In fact, $\text{T}(A)$ is not even separable in the Wasserstein topology  \cite[Theorem 1.8]{GJNS21}.
Nevertheless, our density results apply even for this topology.

\begin{mcor}\label{mcor:Wasserstein}
    Let $A$ be any $C^*$-algebra for which Theorem \ref{Poulsen}, Corollary \ref{fin_dim_algebras}, or Corollary \ref{group C*-algebras} shows that $\text{T}(A)$ is a Poulsen simplex. Then ${\partial_e\text{T}(A)}$  is dense in $\text{T}(A)$ in the Wasserstein topology.
\end{mcor}

Prior to our work, this statement was only known for $A=C([-1,1])^{*m}$, for every $m\geq 2$, by a result of Dabrowski  \cite[Corollary 5]{Da10}. 
The statement of Corollary~\ref{mcor:Wasserstein} is new for all other $C^*$-algebras that we treat, including notably $A=\mathbb M_n(\mathbb C)*\mathbb M_n(\mathbb C)$, for any $n\geq 2$.

\subsection*{Perturbations of von Neumann subalgebras}
We next turn to our main technical result, Theorem \ref{perturbation}, from which we will deduce Theorem \ref{Poulsen}.  To motivate the former, we first outline our strategy for proving Theorem \ref{Poulsen}.

Let  $A=A_1*A_2$  be the 
free product of unital, separable $C^*$-algebras $A_1$ and $A_2$, and let $\varphi\in\text{T}(A)$. 
The GNS construction provides a tracial von Neumann algebra $(M,\tau)$ and a $*$-homomorphism $\pi:A\rightarrow M$  such that $\pi(A)''=M$ and $\varphi=\tau\circ\pi$. Recall that $\varphi$ is an extreme trace if and only if $M$ is a factor.
As in \cite{OSV23}, in order to conclude that $\text{T}(A)$ is a Poulsen simplex, it suffices to approximate $\varphi$ by traces whose GNS von Neumann algebras are factors. To this end, we construct $*$-homomorphisms $\widetilde\pi:A\rightarrow \widetilde M$, for some tracial von Neumann algebra $(\widetilde M,\widetilde\tau)$ containing $M$ and satisfying $\widetilde\tau_{|M}=\tau$, which are pointwise close to $\pi$ in $\|\cdot\|_2$ and such that $\widetilde\pi(A)''$ is a factor. Once this is achieved, it follows that $\varphi$ is close in the Wasserstein distance (and thus in the weak$^*$-topology) to the extreme trace $\widetilde\tau\circ\widetilde\pi$.

In order to construct $\widetilde\pi$, we exploit the free product structure of $A$ by perturbing the image of $A_2$, while leaving $A_1$ unchanged. Specifically, we consider small perturbations $\widetilde\pi$ of $\pi$ of the form $\widetilde\pi(a_1)=\pi(a_1)$, for every $a_1\in A_1$, and $\widetilde\pi(a_2)=u\pi(a_2)u^*$, for every $a_2\in A_2$, for some unitary $u\in\widetilde M$ with $\|u-1\|_2\approx 0$.
Denoting $M_1=\pi(A_1)''$ and $M_2=\pi(A_2)''$, the factoriality of $\widetilde\pi(A)''$ then amounts to the von Neumann algebra $M_1\vee uM_2u^*$ generated by $M_1$ and $uM_2u^*$ being a factor. 

However, finding such a unitary $u$ is not always possible: 
if there are minimal central projections $p_1\in M_1$ and $p_2\in M_2$ with $\tau(p_1)+\tau(p_2)>1$, then $p_1\wedge up_2u^*\in M_1\vee uM_2u^*$ is a non-zero minimal central projection for any unitary $u$, see also Subsection~\ref{ssec:nodiff}. In other words, if $M_1$ and $M_2$ have ``large'' $1$-dimensional direct summands, then  $M_1\vee uM_2u^*$ is never a factor.
More generally, we show in Proposition~\ref{findim} that the same conclusion can be reached if one of these direct summands is only assumed to be finite dimensional, but not necessarily $1$-dimensional.

Theorem \ref{perturbation} below shows that this is essentially the only obstruction to finding a unitary $u$ with $\|u-1\|_2\approx 0$ such that $M_1\vee uM_2u^*$ is a factor. 
In order to capture the presence of a ``large'' finite dimensional direct summand, we define for a tracial von Neumann algebra $(M,\tau)$ the following:
$$\mathrm{e}(M)\coloneq \max\,\{\,\tau(p)\mid\,\text{$p\in M$ a minimal projection}\}.$$
By convention, we define $\text{e}(M)=0$, if $M$ is diffuse.

\begin{main}\label{perturbation}
Let $(M,\tau)$ be a tracial von Neumann algebra and $M_1,M_2\subset M$ be von Neumann subalgebras. Assume that $M_1\neq \mathbb C1,M_2\neq \mathbb C1$, $\dim(M_1)+\dim(M_2)\geq 5$, and $\mathrm{e}(M_1)+\mathrm{e}(M_2)\leq 1$.

Then there exists a II$_1$ factor $(\widetilde M,\widetilde\tau)$  such that $M\subset\widetilde M$, $\widetilde\tau_{|M}=\tau$ and for every $\eps>0$, there exists $v_\varepsilon\in\mathcal U(\widetilde M)$ satisfying $\|v_\varepsilon-1\|_2<\varepsilon$ and $M_1\vee v_\eps M_2 v_\eps^*$ is a II$_1$ factor.
\end{main}

If $M_1$ or $M_2$ is diffuse, then Theorem \ref{perturbation} follows easily by applying  \cite{IPP05}, see Corollary~\ref{one diffuse algebra}. 
However, in order to prove Theorem \ref{Poulsen}, we need the general case of Theorem \ref{perturbation}.

The main novelty of Theorem \ref{perturbation} lies in treating the most difficult case 
when $M_1$ and $M_2$ are finite dimensional. In fact, most of the proof of Theorem \ref{perturbation} is devoted to the case when $M_1$ and $M_2$ are finite dimensional and abelian.  
In this case, the assumptions from Theorem \ref{perturbation} are inspired by Dykema's work \cite{Dy93}. Specifically, \cite[Theorem 2.3]{Dy93} implies that if  $M_1$ and $M_2$ are finite dimensional and abelian, then $M_1*M_2$ is a II$_1$ factor if and only if $(\star$) $M_1\neq \mathbb C1,M_2\neq \mathbb C1$, $\dim(M_1)+\dim(M_2)\geq 5$, and $\mathrm{e}(M_1)+\mathrm{e}(M_2)\leq 1$.
Identifying $M_1*M_2$ with $M_1\vee uM_2u^*$, where $u$ is a trace zero unitary which is freely independent from $M$, \cite[Theorem 2.3]{Dy93} implies that if $(\star)$ holds, then $M_1$ and \textit{some} unitary perturbation  of $M_2$ generate a II$_1$ factor.   
Theorem \ref{perturbation} 
shows that if $(\star)$ holds, then  $M_1$ and an \textit{arbitrarily small} unitary perturbation of $M_2$ generate a II$_1$ factor.

 We continue by illustrating the effectiveness of Theorem \ref{perturbation} with a consequence for the $C^*$-algebra $A_{m,n}:=\mathbb C^m*\mathbb C^n\cong C^*({\mathbb Z}/{m\mathbb Z}*{\mathbb Z}/{n\mathbb Z})$, where $m,n\geq 2$ satisfy $m+n\geq 5$.
Traces on $A_{m,n}$ have attracted considerable attention as they parameterize synchronous quantum correlations (see e.g. \cite{KPS18}). Note that  $\text{T}(A_{m,n})$ is not a Poulsen simplex, as follows from \cite[Theorem 1.4]{OSV23} (or  Corollary \ref{fin_dim_algebras}).  Nevertheless,  Theorem \ref{perturbation}  allows to characterize the traces on $A_{m,n}$  which can be approximated by extreme traces:

\begin{mcor}\label{A_{m,n}}
Let $m,n\geq 2$ with $m+n\geq 2$. Write    
 $\mathbb C^m=\bigoplus_{i=1}^m\mathbb Ce_i$ and $\mathbb C^n=\bigoplus_{j=1}^n\mathbb Cf_j$, for projections $(e_i)_{i=1}^m$ and $(f_j)_{j=1}^n$.  Then the following conditions are equivalent for every $\varphi\in\emph{T}(A_{m,n})$:
\begin{enumerate}
  \item $\varphi\in\overline{\partial_{\emph{e}}\emph{T}(A_{m,n})}^{\emph{Wasserstein}}$.
    \item $\varphi\in\overline{\partial_{\emph{e}}\emph{T}(A_{m,n})}^{\emph{weak}^*}$.
   
     \item $\varphi(e_i)+\varphi(f_j)\leq 1$, for every $1\leq i\leq m$ and $1\leq j\leq n$.
\end{enumerate}
\end{mcor}

An immediate interesting consequence of Corollary \ref{A_{m,n}} is that the closure of the extreme traces of $A_{m,n}$ is not a face of $\text{T}(A_{m,n})$.
Turning to the proof of Corollary, \ref{A_{m,n}}, it is clear that (1) $\Rightarrow$ (2). It is also easy to show that (3) holds if $\varphi\in\partial_{\text{e}}\text{T}(A_{m,n})$ (see the first paragraph of Subsection~\ref{ssec:nodiff}). Thus, it also holds if  $\varphi\in\overline{\partial_{\text{e}}\text{T}(A_{m,n})}^{\text{weak}^*}$.  
Finally, the implication (3) $\Rightarrow$ (2) is an immediate consequence of Theorem \ref{perturbation} in the case when $M_1\cong\mathbb C^m$ and $M_2\cong\mathbb C^n$.

Next, we discuss in further detail the assumptions from Theorem \ref{perturbation} and propose an open question suggested by its statement.
First, we note that the assumption that $\dim(M_1)+\dim(M_2)\geq 5$ is necessary. Indeed, if  $\dim(M_1)+\dim(M_2)<5$, then since $M_1\neq \mathbb C1$ and $M_2\neq \mathbb C1$, it follows that $\dim(M_1)=\dim(M_2)=2$. In this case, $M_1\vee uM_2u^*$ is generated by two projections, hence is a type I algebra and therefore not a II$_1$ factor, for any unitary $u$. Second, the assumption that $\text{e}(M_1)+\text{e}(M_2)\leq 1$ is also necessary if $M_1$ and $M_2$ are abelian. 
Indeed, if $\text{e}(M_1)+\text{e}(M_2)>1$, then we can find minimal (and necessarily central) projections $p_1\in M_1,p_2\in M_2$ with $\tau(p_1)+\tau(p_2)>1$. 

The II$_1$ factor $\widetilde M$ from Theorem \ref{perturbation} is obtained from $M$ by applying iteratively various amalgamated free product constructions. It remains open whether $\widetilde M$ can be taken of a specific form:

\begin{question}\label{open} Assume the setting of Theorem \ref{perturbation}.

\begin{itemize}
    \item[(a)] If $M$ is a II$_1$ factor, does the conclusion of Theorem \ref{perturbation} hold for $\widetilde M=M$?
    \item[(b)] For general $M$, does the conclusion of Theorem \ref{perturbation} hold for $\widetilde M=M*\text{L}(\mathbb Z)$ and some $v_\varepsilon\in\mathcal U(\widetilde M)$ 
    such that $v_\varepsilon\in\text{L}(\mathbb Z)$ and $\|v_\varepsilon-1\|_2\leq\varepsilon$.
\end{itemize}
\end{question}

\subsection*{The Poulsen simplex as a closed face}
A  natural question arising from Theorem~\ref{Poulsen} is what happens when the equivalent conditions of Theorem~\ref{Poulsen} fail, i.e., when $\mathrm{T}(A)$ is not a Poulsen simplex.  
We show that although $\mathrm{T}(A)$ might not be the Poulsen simplex, it always admits a closed face affinely homeomorphic to the Poulsen simplex, unless it is affinely homeomorphic to $\mathrm{T}(\C^2\ast\C^2)$. We refer to Remark~\ref{trace space of C^2*C^2} for a description of the trace simplex of $\C^2\ast\C^2\cong C^*(\Z/2\Z\ast\Z/2\Z)$.

\begin{main}\label{Poulsenface} For $i=1,2$, let $A_i$ be a unital, separable $C^*$-algebra such that $\emph{T}(A_i)$ is non-empty and does not consist of a single $1$-dimensional trace. Let $A=A_1*A_2$ be 
free product. Then exactly one of the following holds:
\begin{enumerate}
    \item $\emph{T}(A)$ admits a closed face which is affinely homeomorphic to the Poulsen simplex, or
    \item $A_i$ admits exactly two extreme traces, both of which are $1$-dimensional, for every $i\in\{1,2\}$. Moreover, in this case, $\emph{T}(A)$ is affinely homeomorphic to $\emph{T}(\mathbb C^2*\mathbb C^2)$.
\end{enumerate}
\end{main}

Note that if $\text{T}(A)$ admits the Poulsen simplex as a closed face, then it admits any metrizable Choquet simplex as a closed face. This is in contrast to many $C^*$-algebras of interest,  such as reduced $C^*$-algebras of countable groups with trivial amenable radical \cite{BKKO14}, that possess a unique trace. Note also that if $A$ is a unital separable $C^*$-algebra such that $\text{T}(A)$ admits the Poulsen simplex as a closed face, then the same is true for any unital separable $C^*$-algebra $B$ which admits $A$ as a quotient (see \cite{MR19}, and also Section~\ref{sec:face}).

The following corollary is an immediate consequence of Theorem \ref{Poulsenface}:

\begin{mcor}\label{Poulsenface_cor}

If $A=A_1*A_2$ is a 
free product of two finite dimensional $C^*$-algebras, then $\emph{T}(A)$ admits the Poulsen simplex as a closed face if and only if $A_1\neq\mathbb C\neq A_2$ and $A_1\ast A_2\not\cong \C^2\ast\C^2$. 

If $G=G_1\ast G_2$ is a free product of two non-trivial countable discrete groups, then the space of traces of $G$ admits the Poulsen simplex as a closed face if and only if $G\not\cong \Z/2\Z\ast \Z/2\Z$. 
\end{mcor}

 Corollary \ref{Poulsenface_cor} implies that the trace simplex of $\text{PSL}_2(\mathbb Z)\cong\mathbb Z/2\mathbb Z*\mathbb Z/3\mathbb Z$ admits the Poulsen simplex as a closed face. The same is therefore true for  $\text{SL}_2(\mathbb Z)$ since it has $\mathrm{PSL}_2(\Z)$ as a quotient.

Next, we discuss \cite[Question~1.11]{OSV23}, which asks whether the closed convex hull of the infinite dimensional extreme traces always yields a Poulsen simplex. 
We answer this question in the negative, and show in Proposition~\ref{noface} that this is in fact never the case. 
Moreover, we show that the closed convex hull of the infinite dimensional extreme traces is in fact never a Choquet simplex, and in particular not a closed face. Nevertheless, despite this negative answer, the extreme points of this closed convex set are dense. Additionally, 
in Lemma \ref{lem:diff},
we identify a related natural maximal \textit{non-closed} face of $\mathrm{T}(A)$ in which the extreme points are dense.

\subsection*{Infinite free products}

Another question suggested by Theorem \ref{Poulsen} is what happens for  free products of more than two $C^*$-algebras. Let $N\in\{2,3,\cdots\}\cup\{\infty\}$ and $\{A_i\}_{i=1}^N$ be unital separable $C^*$-algebra such that $\text{T}(A_i)$ is non-empty
and does not consist of a single $1$-dimensional trace for every $i$. Denote $A=*_{i=1}^N A_i$. 
If $N\in\{2,3,\cdots\}$, then Theorem \ref{Poulsen} readily implies that the following conditions are equivalent: 
\begin{enumerate}
 \item $\text{T}(A)$ is the Poulsen simplex.
    \item If $A_j$ has an isolated extreme 1-dimensional trace,  for every $j\in\{1,2,\cdots,N\}\setminus\{i\}$, for some $i\in\{1,2,\cdots,N\}$, then $A_i$ does not have an isolated extreme finite dimensional trace.
    \item $A$ does not have an isolated extreme finite dimensional trace.
\end{enumerate}

On the other hand, we show 
that if $N=\infty$, then $\text{T}(A)$ is always the Poulsen simplex:

\begin{mprop}\label{infinite-free-product}
For $i\geq 1$, let $A_i$  be a unital separable $C^*$-algebra such that $\emph{T}(A_i)$ is non-empty
and does not consist of a single $1$-dimensional trace. Let $A=*_{i=1}^\infty A_i$. Then $\emph{T}(A)$ is the Poulsen simplex.
\end{mprop}

This result implies that if $G=*_{i=1}^\infty G_i$ is a free product of infinitely many non-trivial countable groups, then the trace space of $G$ is the Poulsen simplex. This in particular applies when $G_i$ are non-trivial finite groups, contrasting the fact that the trace space a free product of finitely many non-trivial finite groups is never Poulsen, as established by \cite[Theorem 1.4]{OSV23}.

Let $B$ be a unital separable $C^*$-algebra such that $\text{T}(B)$ is non-empty
and does not consist of a single $1$-dimensional trace. Let $A=*_1^\infty B$ be the 
free product of infinitely many copies of $B$. 
It was shown in \cite[Theorems 5.3 and 5.5]{DDM14} that a certain compact convex subset $\text{TQSS}(B)\subset\text{T}(A)$ (consisting of the tracial quantum symmetric states on $A$) is a closed face of $\text{T}(A)$, which is the Poulsen simplex. Hence, $\text{T}(A)$ admits the Poulsen simplex as a closed face. 
Strengthening this fact, Proposition \ref{infinite-free-product} shows that $\text{T}(A)$ is in fact the Poulsen simplex.

\subsection*{An application to quantum information theory}
There is a strong and  interesting connection between traces on free products of $C^*$-algebras and quantum information theory. We describe below one such instance and explain the implications of our results. 

In quantum communication, a quantum channel is modeled by a unital completely positive trace preserving (UCPT) map between matrix algebras $T:\mathbb{M}_n(\C)\to \mathbb{M}_n(\C)$. Such a  UCPT map $T$ factorizes through a tracial von Neumann algebra $(M,\tau)$ if there exist unital $*$-homomorphisms $\alpha,\beta:\mathbb{M}_n(\C) \to M$, such that  
 $T=\beta^*\circ \alpha$, where $\beta^*:M\to \mathbb{M}_n(\C)$ is the adjoint map.   This condition is equivalent to  $\text{tr}_n(T(x)y)=\tau(\alpha(x)\beta(y))$, for every $x,y\in \mathbb M_n(\mathbb C)$, where $\text{tr}_n:\mathbb M_n(\mathbb C)\to \C$ is the normalized trace. The factorization is called {\it surjective} if $\alpha(\mathbb M_n(\mathbb C))\vee\beta(\mathbb M_n(\mathbb C))=M$.   
A UCPT map $T$ is called {\it factorizable} if it admits a factorization through a tracial von Neumann algebra. We note that if $n = 2$, then  all UCPT maps are factorizable \cite{Kum85}. On the other hand,  non-factorizable UCPT maps exist for all $n\geq 3$  \cite{HM11}.

It was shown in \cite{MR19} that factorizable UCPT maps $\mathbb M_n(\mathbb C)\rightarrow\mathbb M_n(\mathbb C)$ are parameterized by traces on the 
free product $C^*$-algebra $\mathbb{M}_n(\C)*\mathbb{M}_n(\C)$. This is exhibited by an explicit  surjective continuous affine map from the trace space $\mathrm{T}(\mathbb{M}_n(\C)*\mathbb{M}_n(\C))$ onto the space of factorizable UCPT maps.
Thus, in a similar fashion to \cite[Corollary 1.7]{OSV23}, Corollary \ref{fin_dim_algebras} 
has the following consequence: 
\begin{mcor}\label{cor:factorizable channels}
 Let $n\geq 2$.   Then any factorizable UCPT map $T:\mathbb M_n(\mathbb C)\rightarrow\mathbb M_n(\mathbb C)$ can be approximated arbitrarily well by  UCPT maps factorizing surjectively through a II$_1$ factor.
\end{mcor}
Corollary \ref{cor:factorizable channels} differs from \cite[Corollary 1.7]{OSV23} by showing that the approximating maps can be chosen to factorize surjectively through a II$_1$ factor rather than a (possibly finitely dimensional) tracial factor. Furthermore, Corollary \ref{cor:factorizable channels}  also covers the cases $n=2$ and $n=3$ not covered in \cite{OSV23}. On the finite dimensional side, we point out that in the space of all factorizable UCPT maps, the density of UCPT maps which factor through finite dimensional algebras is by \cite[Theorem~3.7]{HM14} equivalent to a positive answer to the  Connes embedding problem, a negative answer to which was announced in 
\cite{JNVWY20}.

\subsection*{Organization of the paper} Besides the Introduction, this paper consists of six other sections. In Section~\ref{sec:prelim}, we gather some definitions and notations, as well as a few elementary lemmas to be used throughout the paper. In Section~\ref{sec:obstruction}, we establish an obstruction to the trace simplex of a 
free product $C^*$-algebra being a Poulsen simplex, and we prove Proposition~\ref{infinite-free-product}. In Sections~\ref{sec:perturbation} and \ref{sec:gen_position}, we prove Theorem~\ref{perturbation}, and in Section~\ref{sec:Poulsen}, we prove Theorem~\ref{Poulsen}. Finally, in Section~\ref{sec:face}, we study what happens when the equivalent conditions of Theorem~\ref{Poulsen} fail and we prove Theorem~\ref{Poulsenface}.

\subsection*{Acknowledgments} The authors would like to thank Mikael R{\o}rdam for asking us a question which led to Theorem~\ref{Poulsenface}, and for pointing out the result from \cite{EL92} in Remark~\ref{rk:tracialquotient}.  The second-named author would also like to thank Mikael R{\o}rdam for fruitful conversations related to the contents of this paper, and both Mikael R{\o}rdam and Magdalena Musat for useful discussions about the contents of \cite{MR19} and \cite{OSV23}. Additionally, the authors thank Raz Slutsky for his remarks on the manuscript and Dima Shlyakhtenko for pointing out the reference \cite{Da10}.

\section{Preliminaries}\label{sec:prelim}

\subsection{Tracial von Neumann algebras} 
A {\it tracial von Neumann algebra} is a pair $(M,\tau)$ consisting of a von Neumann algebra $M$ and a fixed faithful normal tracial state $\tau:M\rightarrow\mathbb C$. 
We start by recalling some terminology concerning tracial von Neumann algebras; see  \cite{AP} for more information.

Let $(M,\tau)$ be a tracial von Neumann algebra. Any unital embedding $M_0\subset M$, where $(M_0,\tau_0)$ is a tracial von Neumann algebra, will be assumed trace preserving, i.e., such that $\tau_0=\tau_{|M_0}$. We denote by $\mathrm{E}_{M_0}:M\rightarrow M_0$ the {\it conditional expectation onto} $M_0$. For $x\in M$,  $\|x\|$  and $\|x\|_2\coloneq \sqrt{\tau(x^*x)}$ denote the operator norm and $2$-norm of $x$, respectively.  We denote by $\mathrm{L}^2(M)$ the Hilbert space obtained by completing $M$ with respect to the $2$-norm and by $\mathcal U(M)$ the {\it unitary group} of $M$. A projection $p\in M$ is called \textit{minimal} if $pMp\cong \C 1$. We call $M$ \textit{diffuse} if it has no minimal projections.

The von Neumann algebra generated by a self-adjoint set $S\subset M$ with $1\in S$ is equal to its double commutant $S''\coloneq (S')'$, where for $S\subset \mathbb B(\mathcal H)$ we write $S'\coloneq\{x\in \mathbb B(\mathcal H)\mid xy=yx \text{ for all } y\in S\}$. 
For von Neumann subalgebras $M_1,M_2\subset M$, which we will always assume unital, we denote by $M_1\vee M_2\coloneq (M_1\cup M_2)''$ the von Neumann algebra generated by $M_1$ and $M_2$.

We denote by $\mathcal Z(M)\coloneq M'\cap M$ the {\it center} of $M$.  Given a projection $p\in M$, we denote by $z_M(p)$ the smallest projection $z\in\mathcal Z(M)$ with $z\geq p$ and call it the {\it central support} of $p$. 

If $p,q\in M$ are projections, then $\tau(p\wedge q)=\tau(p)+\tau(q)-\tau(p\vee q)\geq \tau(p)+\tau(q)-1$. Here $p\wedge q$ denotes the largest lower bound, i.e., the largest projection $r$ satisfying $r\leq p$ and $r\leq q$, and $p\vee q$ denotes the smallest upper bound, i.e., the smallest projection $r$ satisfying $p\leq r$ and $q\leq r$.
We say that $p$ and $q$ are in {\it general position} if $\tau(p\wedge q)=\max\{\tau(p)+\tau(q)-1,0\}$.

We continue with two elementary lemmas:

\begin{lemma}\label{complement}
Let $(M,\tau)$ be a tracial von Neumann algebra, $p\in M\setminus\{0,1\}$ be a projection and $q_1,\ldots,q_n\in M\setminus\{1\}$ be pairwise orthogonal projections, for some $n\geq 1$. Assume 
that $p$ and $q_j$ are in general position, for every $1\leq j\leq n$. Then $p-\sum_{j=1}^np\wedge q_j\neq 0$.
\end{lemma}

\begin{proof} 
Let $S=\{1\leq j\leq n\mid p\wedge q_j\neq 0\}$. If $S=\emptyset$, the conclusion is immediate as $p\neq 0$. If $|S|=1$, let $1\leq k\leq n$ such that $S=\{k\}$. Since $q_{k}\neq 1$, we get that $\tau(p\wedge q_k)=\tau(p)+\tau(q_k)-1<\tau(p)$ and hence $p-\sum_{1\leq j\leq n}p\wedge q_j=p-p\wedge q_k\neq 0$. If $|S|\geq 2$, then as $p\neq 1$, $(|S|-1)(\tau(p)-1)<0$ and thus
 $$\tau(\sum_{1\leq j\leq n}p\wedge q_j)=\sum_{j\in S} (\tau(p)+\tau(q_j)-1)=\tau(p)+ (\sum_{j\in S}\tau(q_j)-1)+(|S|-1)(\tau(p)-1)<\tau(p).$$
This altogether implies that  $p-\sum_{1\leq j\leq n}p\wedge q_j\neq 0$.
\end{proof}

\begin{lemma}
[\!\!\cite{Dy93}]
\label{generation}
Let $M$ be a von Neumann algebra and $M_1, M_2\subset M$ be von Neumann subalgebras with $M=M_1\vee M_2$. Let $p\in \mathcal Z(M_1)$ be a projection and put $N=(\mathbb Cp\oplus M_1(1-p))\vee M_2$.
Then
\begin{enumerate}
\item $pMp=pNp\vee M_1p$. 
\item $z_M(p)=z_N(p)$.
\end{enumerate}
\end{lemma}

This result is proved in \cite[Theorem 1.2]{Dy93} when $M_1$ and $M_2$ are freely independent, but the same arguments work in the more general context here. We recall the proof for the reader's convenience. 
\begin{proof}
Since $M$ is generated by $N$ and $M_1p$, and $p(N\vee M_1p)p\subset pNp\vee M_1p$, part (1) follows.
Since $z_N(p)\geq p$, $z_N(p)$ commutes with $M_1p$. Since $z_N(p)$ also commutes with $N$, we deduce that $z_N(p)\in\mathcal Z(M)$ and thus $z_N(p)\geq z_M(p)$. Since we also have that $z_N(p)\leq z_M(p)$, we conclude that $z_N(p)=z_M(p)$, which proves part (2).
\end{proof}

\subsection{Minimal projections}\label{sec:minimal}
Recall that for a tracial von Neumann algebra $(M,\tau)$ we denote 
$$\mathrm{e}(M)=\max\{\tau(p)\mid\;\text{$p\in M$ a minimal projection}\}.$$
For instance, $\mathrm{e}(M)=0$ if $M$ is diffuse, and $\mathrm{e}(M)=\max\{\tau(p_i)\mid 1\leq i\leq n\}$ if $M=\mathbb Cp_1\oplus\cdots\oplus\mathbb Cp_n$ is finite dimensional and abelian.
Moreover, we see that $\mathrm{e}(M)=\mathrm{e}(A)$, if $A\subset M$ is a MASA (maximal abelian von Neumann subalgebra). The following lemma shows that whenever $\mathrm{e}(M)\neq 0$, we can in fact find a finite dimensional abelian von Neumann subalgebra $A\subset M$ with $\mathrm{e}(M)=\mathrm{e}(A)$.

\begin{lemma}\label{masa}
Let $(M,\tau)$ be a non-diffuse tracial von Neumann algebra with $M\neq \mathbb C1$. Then there exists a finite dimensional abelian von Neumann subalgebra $A\subset M$ such that $\mathrm{e}(A)=\mathrm{e}(M)$.  
Moreover, if $M$ is not isomorphic to either $\mathbb C^2$ or $\mathbb M_2(\mathbb C)$, then we can take $A$ with $\dim(A)\geq 3$.
\end{lemma}

\begin{proof}
  Let $\{p_i\}_{i=1}^N$, for $N\in\mathbb N\cup\{\infty\}$, be a maximal family of non-zero pairwise orthogonal minimal projections of $M$. Moreover, we can arrange that $\{\tau(p_i)\}_{i=1}^N$ is a decreasing sequence.
  Then $\mathrm{e}(M)=\tau(p_1)$, $p\coloneq 1-\sum_{i=1}^Np_i\in\mathcal Z(M)$ and $Mp$ is diffuse. 
  If $N\in\mathbb N$, put $K=N$ and $r=0$; if $N=\infty$, choose
 $K\in\mathbb N$ such that  $K\geq 3$ and $r=\sum_{i=K+1}^\infty p_i$ satisfies $\tau(r)\leq\tau(p_1)$. Let $L\in\mathbb N$ with  $L\geq 3$ and $L\tau(p_1)\geq\tau(p)$. Since $Mp$ is diffuse, there are projections $q_1,\ldots,q_L\in Mp$ such that $\sum_{j=1}^Lq_j=p$ and $\tau(q_j)\leq\tau(p_1)$, for every $1\leq j\leq L$. Then $A=(\bigoplus_{i=1}^K\mathbb Cp_i)\oplus\mathbb Cr\oplus(\bigoplus_{j=1}^L\mathbb Cq_j)$ satisfies $A\neq \mathbb C1$ and $\mathrm{e}(A)=\mathrm{e}(M)$. 

To prove the moreover assertion, assume that  $\dim(A)\leq 2$. Since $L\geq 3$, and $K\geq 3$ if $N=\infty$, we deduce that in the above construction, we necessarily have $p=0$ and $N\in\mathbb N$. Thus, $M$ is finite dimensional and $A=\oplus_{i=1}^N\mathbb Cp_i$ is a MASA. Since $N=\dim(A)\leq 2$, this forces $M$ to be isomorphic to either $\mathbb C^2$ or $\mathbb M_2(\mathbb C)$.
\end{proof}

We also record the following easy observation.

\begin{lemma}\label{e wrt central proj}
   Let $(\mathcal M,\tau)$ be a tracial von Neumann algebra and $M\subset\mathcal M$ be a von Neumann subalgebra. Assume that $\{z_i\}_{i=1}^K\subset M'\cap \mathcal M$, for some $K\in\mathbb N\cup\{\infty\},$ are projections such that $\sum_{i=1}^Kz_i=1$. Then $\mathrm{e}(M)\leq\sum_{i=1}^K\mathrm{e}(Mz_i)$.
\end{lemma} 
\begin{proof}
If $\mathrm{e}(M)=0$, there is nothing to prove. Hence, we can assume $\mathrm{e}(M)\neq 0$. Let $p\in M$ be a minimal projection with $\tau(p)=\mathrm{e}(M)$. As $pz_i\in Mz_i$ is a minimal projection, we deduce that $\tau(pz_i)\leq\mathrm{e}(Mz_i)$, for every $i$. Therefore, $\mathrm{e}(M)=\tau(p)=\sum_{i=1}^K\tau(pz_i)\leq\sum_{i=1}^K\mathrm{e}(Mz_i)$.
\end{proof}

\subsection{Full free products of \texorpdfstring{$C^*$-algebras}{C*-algebras}}

In this subsection we briefly recall the definition of the (unital full) free product of two unital $C^*$-algebras, see also, e.g., \cite{VDN92}. This is the co-product in the category of unital $C^*$-algebras:

\begin{definition}
    Let $A_1$ and $A_2$ be two unital $C^*$-algebras. Their \textit{unital full free product} $A=A_1\ast A_2$ is the unique (up to isomorphism) unital $C^*$-algebra $A$, together with unital $*$-homomorphisms $\psi_i:A_i\to A$ such that $\psi_1(A_1)$ and $\psi_2(A_2)$ generate $A$, and the following condition holds: for any unital $C^*$-algebra $B$ and unital $*$-homomorphisms $\pi_i:A_i\to B$, there exists a unique unital $*$-homomorphism $\pi:A\to B$ such that $\pi\circ\psi_i = \pi_i$, for every $i=1,2$.
\end{definition}

\begin{remark}
    \begin{enumerate}
        \item Many authors also use the notation $A_1\ast_\C A_2$ for the unital full free product of $A_1$ and $A_2$. We use the notation $A_1\ast A_2$ and terminology ``free product'' throughout the paper, and we will always mean the unital full free product as defined above.
        \item It is easy to see that the maps $\psi_i:A_i\to A$ are necessarily injective. Therefore, we will usually drop the notation $\psi_i$ and assume that $A_1$ and $A_2$ are unital $C^*$-subalgebras of $A$.
        \item For an alternative description, one can also show that $A_1\ast A_2$ is the enveloping $C^*$-algebra of the $*$-algebra free product of $A_1$ and $A_2$. In particular, this implies that words of the form $b_1c_1b_2c_2\cdots b_mc_m$, where $b_1, \ldots, b_m\in A_1$ and $c_1,\ldots, c_m\in A_2$, are dense in $A_1\ast A_2$.
        \item Given two discrete groups $G_1$ and $G_2$, it is easy to show that $C^*(G_1\ast G_2)\cong C^*(G_1)\ast C^*(G_2)$.
    \end{enumerate}
\end{remark}

\subsection{The trace space of a \texorpdfstring{$C^*$-algebra}{C*-algebra}}
Let $A$ be a unital, separable $C^*$-algebra.
We denote by $\mathrm{T}(A)$ the compact convex set of all traces on $A$ endowed with the weak$^*$-topology. We denote by $\partial_{\mathrm{e}}\mathrm{T}(A)$ the set of extreme points of $\mathrm{T}(A)$.
Then $\mathrm{T}(A)$ is a metrizable Choquet simplex, i.e., every $\varphi\in\mathrm{T}(A)$ is the barycenter of  a unique Borel probability measure supported on $\partial_{\mathrm{e}}\mathrm{T}(A)$.

A {\it tracial representation} of $A$ is a triple $(M,\tau, \pi)$ where $(M,\tau)$ is a tracial von Neumann algebra and $\pi:A\to M$ is a $*$-homomorphism with $\pi(A)''=M$. In such case $\tau\circ\pi\in\mathrm{T}(A)$. Two tracial representations $(M_1,\tau_1,\pi_1)$ and $(M_2,\tau_2,\pi_2)$ give rise to the same trace, namely $\tau_1\circ\pi_2 =\tau_2 \circ \pi_2$, if and only if they are  {\it quasi-equivalent}, namely, if there exists a $*$-isomorphism $\Psi:M_1 \rightarrow M_2$ such that $\tau_1=\tau_2\circ\Psi$ and $\pi_2 = \Psi\circ \pi_1$.

Conversely, given $\varphi\in\mathrm{T}(A)$, there exists a unique (up to quasi-equivalence) tracial representation $(M,\tau,\pi)$ of $A$ such that $\varphi=\tau\circ\pi$. More precisely, $\pi$ is the GNS representation associated to $\varphi$.
We say that $\varphi$ is {\it finite dimensional} if $\dim(M)<\infty$.
We denote by $\mathrm{T}_{\mathrm{fin}}(A)$ the set of all finite dimensional traces $\varphi\in\mathrm{T}(A)$ and let $\mathrm{T}_{\infty}(A)=\mathrm{T}(A)\setminus\mathrm{T}_{\mathrm{fin}}(A)$.
Note that $\varphi\in\partial_{\mathrm{e}}\mathrm{T}(A)$ if and only if $M$ is a factor (i.e.,  $\pi(A)''$ is either a II$_1$ factor or isomorphic to $\mathbb M_n(\mathbb C)$, for some $n\in\mathbb N$).
We say that $\varphi\in \partial_{\mathrm{e}}\mathrm{T}(A)$ is {\it $n$-dimensional}, for some $n\in\mathbb N$, if $M\cong\mathbb M_n(\mathbb C)$. Finally, we say that $\varphi$ is {\it von Neumann amenable} if $M$ is an amenable von Neumann algebra.

 For future reference, we record the following useful known lemma concerning convergence of traces in the weak$^*$-topology versus the Wasserstein topology.

\begin{lemma}[\!\!\cite{GJNS21}]\label{convergence}
   Let $A$ be a unital, separable $C^*$-algebra, $\varphi\in\mathrm{T}(A)$ and $(\varphi_n)_{n\in\mathbb N}\subset\mathrm{T}(A)$ be a sequence. Consider the following two conditions:
   \begin{enumerate}
       \item $\varphi_n\rightarrow\varphi$.
       \item there exist a II$_1$ factor $(M,\tau)$, $*$-homomorphisms $\pi:A\rightarrow M$ and $\pi_n:A\rightarrow M$ such that $\tau\circ\pi=\varphi$, $\tau\circ\pi_n=\varphi_n$ for every $n\in\mathbb N$, and $\|\pi_n(a)-\pi(a)\|_2\rightarrow 0$ for every  $a\in A$.
   \end{enumerate}
Then (2) $\Rightarrow$ (1). Moreover, if $\varphi$ is von Neumann amenable, then (1) $\Rightarrow$ (2). 
\end{lemma}

The main content of this lemma is the moreover assertion. This assertion can be deduced from \cite[Proposition 5.26]{GJNS21} when $A$ is a finitely generated $C^*$-algebra. For the reader's convenience we provide a short self-contained proof.
Note that the moreover assertion in particular applies if $\varphi\in\mathrm{T}_{\mathrm{fin}}(A)$. Note also that condition (2) is equivalent to $\varphi_n\rightarrow\varphi$ in the Wasserstein topology.

\begin{proof}

Let $\pi,\pi_n:A\rightarrow M$ be tracial representations such that $\tau\circ\pi=\varphi$ and $\tau\circ\pi_n=\varphi_n$, for every $n\in\mathbb N$. Then $|\varphi_n(a)-\varphi(a)|=|\tau(\pi_n(a)-\pi(a))|\leq \|\pi_n(a)-\pi(a)\|_2$, for every $a\in A$ and $n\in\mathbb N$. This shows that (2) implies (1).

For the moreover assertion, assume that $\varphi$ is von Neumann amenable and that (1) holds. Let $\pi:A\rightarrow N$ and $\rho_n:A\rightarrow N_n$ be tracial representations associated to $\varphi$ and $\varphi_n$, for every $n\in\mathbb N$.
Define $M=N*(*_{n\in\mathbb N}N_n)*\mathrm{L}(\mathbb F_2)$ and let $\tau$ be its canonical trace. Then $M$ is a II$_1$ factor \cite[Corollary 5.3.8]{AP}. Viewing $N$ and $N_n$ as subalgebras of $M$, we moreover have $\varphi=\tau\circ\pi$ and $\varphi_n=\tau\circ\rho_n$, for every $n\in\mathbb N$.   

We proceed with the following claim:

\textbf{Claim.}
    \textit{There exist $(u_n)_{n\in\mathbb N}\subset\mathcal U(M)$ such that
   $\|u_n\rho_n(a)u_n^*-\pi(a)\|_2\rightarrow 0$, for every $a\in A.$}

Given the claim, it is clear that the maps $\pi_n:A\rightarrow M$ given by $\pi_n(a)=u_n\rho_n(a)u_n^*$ satisfy condition (2). Hence proving the claim will finish the proof of the lemma.

\textit{Proof of the claim.} Let $(a_k)\subset (A)_1$ be a $\|\cdot\|$-dense sequence. For every $n\in\mathbb N$, let $$\varepsilon_n\coloneq \inf\{\;\sum_{k=1}^\infty2^{-k}\|u\rho_n(a_k)u^*-\pi(a_k)\|_2\mid u\in\mathcal U(M)\,\}.$$ 
  The claim is equivalent to the statement that $\varepsilon_n\rightarrow 0$. Assume by contradiction that this is false. Then we can find $\delta>0
$ and a subsequence $(\varepsilon_{n_k})\subset(\varepsilon_n)$  such that $\varepsilon_{n_k}\geq \delta$, for every $k\in\mathbb N$.

Let $\omega$ be a free ultrafilter on $\mathbb N$ and $M^{\omega}$ be the ultrapower von Neumann algebra together with its canonical trace given by $\tau^{\omega}((x_k)_{k\in\omega})=\lim_{k\rightarrow\omega}\tau(x_k)$.
Define  $*$-homomorphisms $\widetilde\pi,\rho:A\rightarrow M^{\omega}$ by setting $\widetilde\pi(a)=(\pi(a))_{k\in\omega}$ and $\rho(a)=(\rho_{n_k}(a))_{k\in\omega}$, for every $a\in A$. Then for every $a\in A$, we have $$\tau^{\omega}(\rho(a))=\lim_{k\rightarrow\omega}\tau(\rho_{n_k}(a))=\lim_{k\rightarrow\omega}\varphi_{n_k}(a)=\varphi(a)=\tau(\pi(a))=\tau^{\omega}(\widetilde\pi(a)).$$ Thus, there is a trace-preserving $*$-isomorphism $\theta:\widetilde\pi(A)''\rightarrow\rho(A)''$ such that $\theta(\widetilde\pi(a))=\rho(a)$, for every $a\in A$. Since $\varphi$ is von Neumann amenable, $\widetilde\pi(A)''$ is amenable. By Connes' theorem \cite{Co75}, $\widetilde\pi(A)''$ is approximately finite dimensional. Since $M$ is a II$_1$ factor, a well-known fact (see, e.g., \cite[Lemma 5.23]{GJNS21}) implies
the existence of $v=(v_k)_{k\in\omega}\in \mathcal U(M^\omega)$, where $v_k\in \mathcal U(M)$ for every $k\in\mathbb N$, such that $\widetilde\pi(a)=v\rho(a)v^*$, for every $a\in A$. In other words, $\lim_{k\rightarrow\omega}\|\pi(a)-v_k\rho_{n_k}(a)v_k^*\|_2=0$, for every $a\in A$. This implies that $\lim_{k\rightarrow\omega}\varepsilon_{n_k}=0$, contradicting that  $n_k\geq \delta$, for every $k\in\mathbb N$. This finishes the proof of the claim and thus the lemma.
\end{proof}

\subsection{The tracial quotient of a \texorpdfstring{$C^*$-algebra}{C*-algebra}}\label{sec:tracialquotient}

Before continuing, we note that studying the trace space of a   $C^*$-algebra amounts to studying its \textit{tracial quotient}:

\begin{definition}
    Let $A$ be a unital, separable $C^*$-algebra with $\mathrm{T}(A)\neq\emptyset$. We define $A_{\mathrm{tr}}$ to be the quotient 
    \[
    A_{\mathrm{tr}} \coloneq A/I_{\mathrm{tr}},
    \]
    where $I_{\mathrm{tr}} = \{a\in A\mid \vphi(a^*a)=0 \text{ for all }\vphi\in\mathrm{T}(A)\}$. We call $A_{\mathrm{tr}}$ the \textit{tracial quotient of $A$}.
\end{definition}

\begin{remark}
    By construction, $\mathrm{T}(A)$ is affinely homeomorphic to $ \mathrm{T}(A_{\mathrm{tr}})$. Furthermore, we could equivalently define $A_{\mathrm{tr}}$ as the largest quotient $C^*$-algebra of $A$ that admits a faithful tracial state.

    We also note that the condition from Theorem~\ref{Poulsen} that $\mathrm{T}(A_i)$ does not consist of a single $1$-dimensional trace can be reformulated as $A_{i,\mathrm{tr}}\not\cong \C$ (i.e., $A_i$ is not ``tracially equivalent" to $\mathbb C$). 
\end{remark}

We record the following easy observation.

\begin{lemma}\label{lem:freeprodtracial}
    Let $A=A_1\ast A_2$ be the 
    free product of two unital, separable $C^*$-algebras $A_1$ and $A_2$. Then $\mathrm{T}(A)$ (or equivalently $\mathrm{T}(A_{\mathrm{tr}})$) is affinely homeomorphic to $\mathrm{T}(A_{1,\mathrm{tr}}\ast A_{2,\mathrm{tr}})$.
\end{lemma}
\begin{proof}
For $i=1,2$, we have by construction a unital, surjective $*$-homomorphism $\pi_i: A_i\to A_{i,\mathrm{tr}}$ such that any tracial representation $\rho_i:A_i\to M_i$ factors through $\pi_i$. Consider the unital, surjective $^*$-homomorphism $\pi:A\to A_{1,\mathrm{tr}}\ast A_{2,\mathrm{tr}}$ determined by $\pi_{|A_i} = \pi_i$, for $i=1,2$. 

Let $\rho:A\to M$ be a tracial representation. Since $\rho_{|A_i}$ factors through $\pi_i$, we can find unital $*$-homomorphisms $\delta_i: A_{i,\mathrm{tr}}\to M$ such that $\rho_{|A_i} = \delta_i\circ \pi_i$, for $i=1,2$. Defining the unital $*$-homomorphism $\delta:A_{1,\mathrm{tr}}\ast A_{2,\mathrm{tr}}\to M$ by $\delta_{|A_{i,\mathrm{tr}}} = \delta_i$, for $i=1,2$, we get that $\rho = \delta\circ \pi$.

This implies that the continuous affine injective map $\mathrm{T}(A_{1,\mathrm{tr}}\ast A_{2,\mathrm{tr}})\ni \vphi\mapsto \vphi\circ \pi\in \mathrm{T}(A)$ is surjective, and therefore an affine homeomorphism. This finishes the proof of the lemma.
\end{proof}

\begin{remark}\label{rk:tracialquotient}
    It is not clear to the authors whether for a 
    free 
    product $C^*$-algebra $A=A_1\ast A_2$, it is true in general that $(\star)$ $A_{\mathrm{tr}}\cong A_{1,\mathrm{tr}}\ast A_{2,\mathrm{tr}}$. We note that by the proof of Lemma~\ref{lem:freeprodtracial}, every tracial representation of $A$ factors through the quotient $A_{1,\mathrm{tr}}\ast A_{2,\mathrm{tr}}$, and hence we have a natural surjective $*$-homomorphism $A_{1,\mathrm{tr}}\ast A_{2,\mathrm{tr}} \to A_{\mathrm{tr}}$. Since $A_{\mathrm{tr}}$ always has a faithful tracial state by construction, this implies that $(\star)$ is equivalent to the statement that the 
    free 
    product of two unital, separable $C^*$-algebras with a faithful tracial state, has a faithful tracial state. When $A_1$ and $A_2$ are residually finite dimensional (RFD), their 
    free product is also RFD by \cite[Theorem~3.2]{EL92}, and thus has a faithful trace. In particular, we conclude that $A_{\mathrm{tr}}\cong A_{1,\mathrm{tr}}\ast A_{2,\mathrm{tr}}$, whenever $A_{1,\mathrm{tr}}$ and $A_{2,\mathrm{tr}}$ are RFD. We leave it open whether $(\star)$ holds in general. 
\end{remark}

\section{An obstruction to the trace simplex being Poulsen}\label{sec:obstruction}

In this section, we prove one implication of Theorem~\ref{Poulsen}, by establishing an obstruction to the trace simplex of a 
free product $C^*$-algebra being a Poulsen simplex, see Corollary~\ref{cor:obstruction} below. 
This obstruction relies on following general lemma, which shows that isolated extreme points in compact convex sets satisfy an a priori stronger condition on convergence of barycenter measures.

\begin{lemma}
Let $C$ be a compact and convex set in a locally convex topological
vector space. Let $x\in\partial_{\mathrm{e}}C$ be an extreme point and suppose
that it is an isolated point of $\partial_{\mathrm{e}}C$. Consider a sequence of Borel probability
measures $\mu_{n}\in\Prob(C)$ supported on $\partial_{\mathrm{e}}C$. Then $\bary\mu_{n}$ converges to $x$
if and only if $\lim_n \mu_{n}(\{x\}) = 1$. 
\end{lemma}
\begin{proof}
Firstly, it is clear that if $\mu_{n}(\{x\})\to1$, then $\bary\mu_{n}\to x$. For
the converse, let $\mu$ be any accumulation point of the sequence $\mu_n$ in the weak-$^*$ compact space $\Prob(C)$. As the barycenter map is continuous, we have $ \bary \mu_n  \to \bary \mu $.  But by assumption $\bary \mu_n\to x$, and so $\bary\mu =x$. As $x$ is an extreme point,  $\mu$ must be the Dirac measure concentrated on $x$. In particular, $\mu$ is supported on $\partial_{\mathrm{e}} C$, as  
is each $\mu_n$ by assumption. Finally, as $\{x\}$ is an open set of $\partial_{\mathrm{e}} C$, weak-$^*$ convergence implies
\[
    \liminf_n \mu_n(\{x\}) \geq  \mu(\{x\})=1
\]
thus showing that $\mu_n(\{x\})$ converges to $1$.
\end{proof}

\begin{corollary}\label{strongly_isolated}
Let $A$ be a unital, separable $C^*$-algebra and let $\varphi\in\partial_{\mathrm{e}}\mathrm{T}(A)$ be an extreme
trace which is an isolated point in $\partial_{\mathrm{e}}\mathrm{T}(A)$. Let $(\varphi_{n})_{n\in\mathbb N}\subset\mathrm{T}(A)$ be a sequence 
such that $\varphi_n\rightarrow\varphi$. Then $\mu_{n}(\left\{ \varphi\right\} )\to1$ where for every $n$,
$\mu_{n}$ is the unique Borel probability measure on $\partial_{\mathrm{e}}\mathrm{T}(A)$
whose barycenter is $\varphi_{n}$. 
\end{corollary}

\subsection{An obstruction to diffuseness}\label{ssec:nodiff} Let $M_1,M_2$ be von Neumann subalgebras of a tracial von Neumann algebra $(M,\tau)$. Assume that there exist  projections $p\in \mathcal Z(M_1),q\in \mathcal Z(M_2)$ such that $M_1p=\mathbb Cp,M_2q=\mathbb Cq$ and
$\tau(p)+\tau(q)>1$.  Then $M_1(p\wedge q)=(M_1p)(p\wedge q)=\mathbb C(p\wedge q)$ and similarly $M_2(p\wedge q)=\mathbb C(p\wedge q)$. 
These facts implies that $p\wedge q\in \mathcal Z(M_1\vee M_2)$ and $(M_1\vee M_2)(p\wedge q)=\mathbb C(p\wedge q)$.
Since $\tau(p)+\tau(q)>1$, it follows that $M_1\vee M_2$ has a $1$-dimensional direct summand. 

We next generalize this fact by assuming that $M_2q$ is finite (but not necessarily one) dimensional.

\begin{proposition}\label{findim}
    Let $(M,\tau)$ be a tracial von Neumann algebra and $M_1,M_2\subset M$ be von Neumann subalgebras. Assume that there exist projections $p\in\mathcal Z(M_1),q\in\mathcal Z(M_2)$ and $k\in\mathbb N$ such that $M_1p=\mathbb Cp$, $M_2q\cong\mathbb M_k(\mathbb C)$ and $\tau(p)+\frac{\tau(q)}{k^2}>1$.

    Then there exists a non-zero  projection $p'\in\mathcal Z(M_1\vee M_2)$ such that $p'\leq p\wedge q$, $Mp'\cong\mathbb M_k(\mathbb C)$, and
 $\tau(p')\geq k^2(\tau(p)+\frac{\tau(q)}{k^2}-1)$.

\end{proposition}

\begin{proof}
Let $q_i\in M_2q\cong\mathbb M_k(\mathbb C)$ be minimal projections and $v_{i,j}\in M_2q$ be partial isometries such that $q_1+\cdots+q_k=q$, $v_{i,j}v_{i,j}^*=q_i$, $v_{i,j}^*v_{i,j}=q_j$ and $v_{i,j}v_{j,l}=v_{i,l}$, for every $1\leq i,j,l\leq k$.

Following the formula right before \cite[Definition 3.3]{Dy93} we define $p'=\sum_{i=1}^kp_i$, where $$\text{$p_i=\bigwedge_{1\leq j\leq k}v_{i,j}(p\wedge q_j)v_{i,j}^*= (p\wedge q_i)\wedge\bigwedge_{1\leq j\leq k, j\neq i}v_{i,j}(p\wedge q_j)v_{i,j}^*$, for every $1\leq i\leq k$.}$$ 

Then $p_i=v_{i,j}p_jv_{i,j}^*$ and hence $p'v_{i,j}=p'q_iv_{i,j}=p_iv_{i,j}=v_{i,j}p_j=v_{i,j}q_jp'=v_{i,j}p'$, for every $1\leq i,j\leq k$. This implies that $p'$ commutes with $M_2$. Since $p'\leq p$, $p'$ also commutes with $M_1$. Since $p'\in M_1\vee M_2$, we derive that $p'\in \mathcal Z(M_1\vee M_2)$. Since $p_i\leq p\wedge q$, for every $1\leq i\leq k$, we also get that $p'\leq p\wedge q$.

To prove the lower bound for $\tau(p')$, recall that for any projections $r\in M$, $e,f\in rMr$ we have $$\tau(e\wedge f)=\tau(e)+\tau(f)-\tau(e\vee f)\geq \tau(e)+\tau(f)-\tau(r).$$  Let $1\leq i\leq k$. By using repeatedly the last inequality, we get that \begin{align*}\tau(p_i)&\geq \sum_{j=1}^k\tau(p\wedge q_j)-(k-1)\tau(q_i)\\&\geq \sum_{j=1}^k(\tau(p)+\tau(q_j)-1)-(k-1)\tau(q_i)\\&=k(\tau(p)+\frac{\tau(q)}{k^2}-1).\end{align*} Thus,  $\tau(p')=\sum_{i=1}^k\tau(p_i)\geq k^2(\tau(p)+\frac{\tau(q)}{k^2}-1)$, proving that $p'\neq 0$ and the desired lower bound.

 Since $M_1p'=\mathbb Cp'$, we have $Mp'=M_2p'$. As $p'\leq q$ is non-zero and $M_2q\cong\mathbb M_k(\mathbb C)$ is a factor, we get that $Mp'\cong\mathbb M_k(\mathbb C)$, which finishes the proof.
\end{proof}

\subsection{Finite dimensional extreme traces} In this subsection, we use the foregoing results to establish an obstruction to the trace simplex of a 
free product $C^*$-algebra being a Poulsen simplex. We also refer to Lemma \ref{lem:ifd} for a converse.

\begin{theorem}\label{freeprod}
Let $A=A_1*A_2$ be the 
free product of two unital, separable $C^*$-algebras  $A_1$ and $A_2$. Assume that $A_1$ admits an extreme $1$-dimensional trace $\varphi_1$ which is isolated in $\partial_{\mathrm{e}}\mathrm{T}(A_1)$, and $A_2$ admits an extreme $k$-dimensional trace $\varphi_2$ which is isolated in $\partial_{\mathrm{e}}\mathrm{T}(A_2)$, for some $k\in\mathbb N$.
    
Then $A$ admits an extreme $k$-dimensional trace $\varphi$ which is isolated in $\partial_{\mathrm{e}}\mathrm{T}(A)$ and satisfies $\vphi_{|A_1}=\vphi_1$ and $\vphi_{|A_2}=\vphi_2$.
\end{theorem}

\begin{proof}
Let $\rho_1:A_1\rightarrow\mathbb C$ and $\rho_2:A_2\rightarrow\mathbb M_k(\mathbb C)$ be the tracial representations associated to $\varphi_1$ and $\varphi_2$. Let $\rho:A\rightarrow\mathbb M_k(\mathbb C)$ be the tracial representation given by $\rho(a_1)=\rho_1(a_1)1$ and $\rho(a_2)=\rho_2(a_2)$, for every $a_1\in A_1$ and $a_2\in A_2.$ We will prove that $\varphi=\mathrm{tr}_k\circ\rho\in \partial_{\mathrm{e}}\mathrm{T}(A)$ is isolated, where $\mathrm{tr}_k:\mathbb M_k(\mathbb C)\rightarrow\mathbb C$ denotes the normalized trace, thereby proving the theorem.

Let $(\psi_n)_{n\in\N}\subset\partial_{\mathrm{e}}\mathrm{T}(A)$ be a sequence such that $\psi_n\rightarrow\varphi$. For $i\in\{1,2\}$, let $\mu_{i,n}$ be the Borel probability measure on $\partial_{\mathrm{e}}\mathrm{T}(A_i)$ whose barycenter is ${\psi_n}_{|A_i}$. Since $\varphi_i=\varphi_{|A_i}\in\partial_{\mathrm{e}}\mathrm{T}(A_i)$ is isolated, Corollary~\ref{strongly_isolated} gives that $\lambda_{i,n}\coloneq \mu_{i,n}(\{\varphi_i\})\rightarrow 1$, for every $i\in\{1,2\}$.

For $n\in\mathbb N$, let $\pi_n:A\rightarrow M_n$ be a tracial representation, where $(M_n,\tau_n)$ is a tracial factor, such that $\psi_n=\tau_n\circ\pi_n$ and $\pi_n(A)''=M_n$.
For $i\in\{1,2\}$, let $M_{i,n}=\pi_n(A_i)''$. 
Then there exist projections $p_{1,n}\in \mathcal Z(M_{1,n})$, $p_{2,n}\in \mathcal Z(M_{2,n})$ and a $*$-isomorphism $\Psi:\mathbb M_k(\mathbb C)\rightarrow M_{2,n}p_{2,n}$ such that $\tau_n(p_{1,n})=\lambda_{1,n}$, $\tau_n(p_{2,n})=\lambda_{2,n}$, \begin{equation}\label{pi_n}\text{$\pi_n(a_1)p_{1,n}=\rho_1(a_1)p_{1,n}$ and $\pi_n(a_2)p_{2,n}=\Psi(\rho_2(a_2))$, for every $a_1\in A_1$ and $a_2\in A_2$.}\end{equation}

By Proposition \ref{findim}, there exists a projection $p_n'\in \mathcal Z(M_n)$ such that $p_n'\leq p_{1,n}\wedge p_{2,n}$ and $\tau(p_n')\geq k^2(\tau(p_{1,n})+\frac{\tau(p_{2,n})}{k^2}-1)=k^2(\lambda_{1,n}+\frac{\lambda_{2,n}}{k^2}-1)$. Then $\tau(p_n')\rightarrow 1$. Since $M_n$ is a factor we conclude that $p_n'=1$, for $n$ large enough. Hence $p_{1,n}=p_{2,n}=1$, for $n$ large enough.

Let $a=b_1c_1\cdots b_mc_m$, where $b_1,\ldots,b_m\in A_1$ and $c_1,\ldots,c_m\in A_2$, for some $m\in\mathbb N$.   By using \eqref{pi_n}  we get that $\pi_n(a)=\rho_1(a_1)\Psi(\rho_2(a_2))=\Psi(\rho_1(a_1)\rho_2(a_2))=\Psi(\rho(a))$, where $a_1=b_1\cdots b_m\in A_1$ and $a_2=c_1\cdots c_m\in A_2$.
As $\mathbb M_k(\mathbb C)$ is a factor, and thus has a unique trace, we get that  $\tau_n(\Psi(y))=\mathrm{tr}_k(y)$, for every 
 $y\in\mathbb M_k(\mathbb C)$.
These facts imply that 
\begin{align*}   \tau_n(\pi_n(a))= \tau_n(\Psi(\rho(a)))=\mathrm{tr}_k(\rho(a))=\varphi(a).\end{align*}
Since the linear span of such $a\in A$ is norm dense in $A$, we conclude that $\psi_n=\tau_n\circ\pi_n=\varphi$, for $n$ large enough. This implies that $\varphi\in\partial_{\mathrm{e}}\mathrm{T}(A)$ is isolated.
\end{proof}

\begin{corollary}\label{cor:obstruction}
    Let $A$ be a $C^*$-algebra satisfying the assumptions of Theorem \ref{freeprod} and assume that $|\mathrm{T}(A)|\geq 2$. Then $\mathrm{T}(A)$ is not a Poulsen simplex.
\end{corollary}

\begin{proof}
If $\mathrm{T}(A)$ were a Poulsen simplex, then  the set of its extreme points, $\partial_{\mathrm{e}}\mathrm{T}(A)$, would be connected by \cite{LOS78}. However, $\partial_{\mathrm{e}}\mathrm{T}(A)$ has an isolated point by Theorem \ref{freeprod}. Since  $|\mathrm{T}(A)|\geq 2$, we also have that $|\partial_{\mathrm{e}}\mathrm{T}(A)|\geq 2$ and thus $\partial_{\mathrm{e}}\mathrm{T}(A)$ is not connected.
\end{proof}

\subsection{Infinite free products of \texorpdfstring{$C^*$-algebras}{C*-algebras}}
In this subsection, we prove Proposition \ref{infinite-free-product}, showing that the obstructions established in the previous subsections for finite free products disappear when considering infinite free products.

\begin{proof}[Proof of Proposition \ref{infinite-free-product}] We will prove that $\partial_{\mathrm{e}}\text{T}(A)\cap \text{T}_\infty(A)$ is dense in $\text{T}(A)$, and thus $\text{T}(A)$ is a Poulsen simplex.
To this end, let $\varphi\in\text{T}(A)$ and $\pi:A\rightarrow M$ the associated tracial representation. 

Let $i\geq 1$. We claim that there is a tracial representation $\rho_i:A_i\rightarrow N_i=\rho_i(A_i)''$ such that any minimal projection of $N_i$ has trace at most $\frac{1}{2}$. If $A_i$ admits a factorial tracial representation $\pi_i$ of dimension at least $2$, we can take $\rho_i=\pi_i$. Otherwise, all factorial tracial representations of  $A_i$ are $1$-dimensional. Since $\text{T}(A_i)$ does not consist of a single $1$-dimensional trace, there are two distinct $1$-dimensional tracial representations $\pi_i^1,\pi_i^2:A_i\rightarrow\mathbb C$. Endow $\mathbb C^2$ with the trace assigning $\frac{1}{2}$ to $(1,0)$ and let $\rho_i=\pi_i^1\oplus\pi_i^2:A\rightarrow\mathbb C^2$. Then any minimal projection in  $\rho_i(A_i)''=\mathbb C^2$ has trace $\frac{1}{2}$.

For $n\geq 1$, let $M_n=M*(*_{i>n}N_i)$ and define a $*$-homomorphism $\pi_n:A\rightarrow M_n$ by letting ${\pi_n}_{|A_i}=\pi_{|A_i}$ if $1\leq i\leq n$ and ${\pi_n}_{|A_i}=\rho_i$ if $i>n$. Denote $\varphi_n:=\tau\circ\pi_n\in\text{T}(A)$. Since $\pi_n(a)=\pi(a)$, for every  $a\in*_{i=1}^nA_i$, we get that $\|\pi_n(a)-\pi(a)\|_2\rightarrow 0$,  for every $a\in A$.
Thus, $\varphi_n\rightarrow\varphi$.

In order to finish the proof, it is enough to argue that $\varphi_n\in\partial_{\mathrm{e}}\text{T}(A)\cap \text{T}_\infty(A)$ or, equivalently, that $\pi_n(A)''$ is a II$_1$ factor, for every $n\geq 1$. Let $n\geq 1$ and put $P_n=\pi(*_{i=1}^nA_i)''$. By the construction of $\pi_n$ we have that $\pi_n(A)''=P_n*(*_{i=n+1}^\infty N_i)$.
Thus, letting $Q_n=P_n*(*_{i=n+5}^\infty N_i)$, we have $\pi_n(A)''=Q_n*(N_{n+1}*N_{n+2})*(N_{n+3}*N_{n+4})$. Since every minimal projection in $N_i$ has trace at most $\frac{1}{2}$, \cite[Theorems 1.1 and  2.3]{Dy93} imply that $N_i*N_{i+1}$ is diffuse, for every $i\geq 1$. Using  this fact, the last free product decomposition of $\pi_n(A)''$ and \cite{Po83} (or \cite[Theorem 1.1]{IPP05}) we get that $\mathcal Z(\pi_n(A)'')\subset (N_{n+1}*N_{n+2})\cap (N_{n+3}*N_{n+4})=\mathbb C1$. Hence, $\pi_n(A)''$ is a II$_1$ factor, as desired.
\end{proof}

\section{Pairs of approximately factorial von Neumann subalgebras}\label{sec:perturbation}

This  section is devoted to proving Theorem \ref{perturbation}. It is convenient to state the theorem using the following terminology.

\begin{definition}\label{condition (A)}
    Let $M_1,M_2$ be tracial von Neumann algebras. 
\begin{enumerate}
    \item[(a)] Let $(M,\tau)$ be a tracial von Neumann algebra which contains $M_1$ and $M_2$. We say that the triple $(M_1,M_2,M)$ is {\it approximately factorial} if for every $\varepsilon>0$, we can find a II$_1$ factor $M_\varepsilon$ and $v_\varepsilon\in\mathcal U(M_\varepsilon)$ such that $M\subset M_\varepsilon$, $\|v_\varepsilon-1\|_2<\varepsilon$, and $M_1\vee v_\varepsilon M_2v_\varepsilon^*$ is a II$_1$ factor.
    \item[(b)] We say that the pair $(M_1,M_2)$ is {\it approximately factorial} if the triple $(M_1,M_2,M)$ is approximately factorial for every tracial von Neumann algebra $(M,\tau)$ which contains $M_1$ and $M_2$.
\end{enumerate}
\end{definition}

With this terminology, Theorem \ref{perturbation} asserts that a pair $(M_1,M_2)$ is approximately factorial, provided that $M_1\neq \mathbb C1$, $M_2\neq \mathbb C1$, $\dim(M_1)+\dim(M_2)\geq 5$ and $\mathrm{e}(M_1)+\mathrm{e}(M_2)\leq 1$.

Before continuing, we make a few observations about Definition~\ref{condition (A)}.

\begin{remark}
   \begin{enumerate}
       \item Since every tracial von Neumann algebra $(\mathcal M,\tau)$ embeds into a II$_1$ factor (e.g., $\mathcal M*\mathrm{L}(\mathbb F_2)$), taking $M_\varepsilon$ to be a tracial von Neumann does not change Definition \ref{condition (A)}(1).
       \item We can take the II$_1$ factor $M_\varepsilon$ to be independent of $\varepsilon$ in Definition \ref{condition (A)}(1). Indeed, if Definition \ref{condition (A)}(1) holds, then it still holds if we replace $M_\varepsilon$ with $*_{n\in\mathbb N, M}M_{\frac{1}{n}}$ for every $\varepsilon>0$, see also the paragraph following the statement of Theorem~\ref{perturbation}.
        \item If $M\subset \widetilde M$ are tracial von Neumann algebras which contain $M_1$ and $M_2$, then $(M_1,M_2,M)$ is approximately factorial if and only if $(M_1,M_2,\widetilde M)$ is approximately factorial (for the `only if' part, consider the  push-out $\widetilde M_\eps =\widetilde{M} *_M M_\eps$).
        Consequently, $(M_1,M_2,M)$ is approximately factorial if and only if $(M_1,M_2,M_1\vee M_2)$ is approximately factorial.
         \item A triple  $(M_1,M_2,M)$ (respectively, a pair $(M_1,M_2)$) is approximately factorial if and only if $(M_2,M_1,M)$ (respectively, $(M_2,M_1)$) is approximately factorial.
        \item If $M_1=\mathbb C1$, then a triple $(M_1,M_2,M)$ (respectively, a pair $(M_1,M_2)$) is approximately factorial if and only if $M_2$ is a II$_1$ factor.
   \end{enumerate} 
\end{remark}

\subsection{Perturbations in amalgamated free products}
Definition \ref{condition (A)} allows for perturbations within a larger von Neumann algebra $M_\eps$. We will construct this ambient von Neumann algebra $M_\eps$ as an iterated amalgamated free product, utilizing results from \cite{IPP05} concerning  
relative commutants of subalgebras, to establish factoriality. The following lemma provides a criterion for when a particular type of perturbation within an amalgamated free product generates a II$_1$ factor.

\begin{lemma}\label{rotate}
Let $(M,\tau)$ be a tracial von Neumann algebra. Let $A_1\subset M_1\subset M$ and $A_2\subset M_2\subset M$ be von Neumann subalgebras. Define $\widetilde M=M*_{A_2}(A_2\overline{\otimes}\mathrm{L}(\mathbb Z))$, let $u\in \mathrm{L}(\mathbb Z)$ be a Haar unitary and choose $h=h^*\in \mathrm{L}(\mathbb Z)$ such that $u=\exp(ih)$. For $t>0$, set $u_t=\exp(ith)$. 
Assume that 
\begin{enumerate}
\item $M_1\vee A_2\nprec_M A_2$. 

\item $A_1\vee A_2$ or $M_2$ is a factor.
\item $A_2\subsetneq M_2$.
\end{enumerate}
    Then $M_1\vee u_tM_2u_t^*$ is a \emph{II}$_1$ factor for every $t>0$.
\end{lemma}

Here, for von Neumann subalgebras $P,Q\subset M$, we write $P\prec_MQ$ if a corner of $P$ embeds into $Q$ inside $M$ in the sense of Popa, see
 Theorem 2.1 and Corollary 2.3 in \cite{Po03}.

\begin{proof}
Let $t>0$ and denote $N_t=M_1\vee u_tM_2u_t^*$.  Since $u_t\in\mathrm{L}(\mathbb Z)$  commutes with $A_2$, we get that $u_tA_2u_t^*=A_2$. Thus, $N_t$ contains $M_1\vee A_2$. Since $M_1\vee A_2\nprec_M A_2$ by assumption (1), \cite[Theorem 1.1]{IPP05} implies that $(M_1\vee A_2)'\cap \widetilde M\subset M$. Hence, $N_t'\cap\widetilde M\subset M$ and so $\mathcal Z(N_t)\subset M$.

Assume by contradiction that $N_t$ is not a II$_1$ factor. Then there is a projection $p\in \mathcal Z(N_t)\setminus\{0,1\}$; put $y=p-\mathrm{E}_{A_2}(p)$. By assumption $(2)$, there exists a factor 
$P$ with $A_2\subset P\subset N_t$. This implies that $\mathrm{E}_{P}(p)\in\mathcal Z(P)=\C$ and therefore $\mathrm{E}_{A_2}(p)=\mathrm{E}_{A_2}(\mathrm{E}_{P}(p)) =\tau(p)$. It follows that $y=p-\tau(p)$. From this we deduce that $y$ commutes with $u_t M_2 u_t^*$ (in fact with $N_t$) and that $y$ is invertible. We will show that the existence of an element  $y\in M$ with these two properties, along with the condition $A_2\subsetneq M_2$, leads to a contradiction. 

Indeed, as $A_2\subsetneq M_2$, we can find a non-zero element $x\in M_2$ with $\mathrm{E}_{A_2}(x)=0$. Since $y$ commutes with $u_tM_2u_t^*$, we have
\begin{equation}\label{commutation}
u_txu_t^*y=yu_txu_t^*.    
\end{equation}
Put $v_t=u_t-\tau(u_t)$, and let $Q$ denote the orthogonal projection from $\mathrm{L}^2(\widetilde M)$ onto the $\|\cdot\|_2$-closure of the linear span of $\{z_1z_2z_3z_4\mid z_1,z_3\in A_2\overline{\otimes}\mathrm{L}(\mathbb Z), z_2,z_4\in M, \text{ and }\forall 1\leq i\leq 4: \mathrm{E}_{A_2}(z_i)=0\}$. It is immediate that $Q(u_txu_t^*y)=v_txv_t^*y$ and $Q(yu_txu_t^*)=0$, and \eqref{commutation} thus implies that $v_txv_t^*y=0$. Using that $\mathrm{E}_{A_2}(x)=\mathrm{E}_{A_2}(y)=\mathrm{E}_{A_2}(v_t)=0$, $\mathrm{E}_{A_2}(v_t^*v_t)=\mathrm{E}_{A_2}(v_tv_t^*)=\|v_t\|_2^2\cdot 1$, and $v_t\in A_2'\cap\widetilde M$, we deduce that
\begin{align*}
    0=\|v_txv_t^*y\|_2^2=\tau(y^*v_tx^*v_t^*v_txv_t^*y)=&\|v_t\|_2^2\tau(y^*v_tx^*xv_t^*y)\\=&\|v_t\|_2^2\tau(y^*v_t\mathrm{E}_{A_2}(x^*x)v_t^*y)
    \\=&\|v_t\|_2^2\tau(y^*\mathrm{E}_{A_2}(x^*x)v_tv_t^*y)\\=&\|v_t\|_2^4\tau(y^*\mathrm{E}_{A_2}(x^*x)y)
\end{align*}
Since $v_t\neq 0$, we derive that $y^*\mathrm{E}_{A_2}(x^*x)y=0$. Since $y$ is invertible, it follows that $\mathrm{E}_{A_2}(x^*x)=0$ and hence $x=0$, which is a contradiction. In conclusion, $N_t$ is a II$_1$ factor.
\end{proof}

We obtain the following immediate consequence of Lemma \ref{rotate}, establishing the special case of Theorem~\ref{perturbation} when one of the algebras is diffuse.

\begin{corollary}\label{one diffuse algebra}
Let $M_1$ and $M_2$ be tracial von Neumann algebras such that $M_1$ is diffuse and $M_2\neq \mathbb C1$. Then $(M_1,M_2)$ is approximately factorial.
\end{corollary}

\begin{proof}
  Let $(M,\tau)$ be a tracial von Neumann algebra which contains $M_1$ and $M_2$.
   Let $\widetilde M=M*\mathrm{L}(\mathbb Z)$, $u\in \mathrm{L}(\mathbb Z)$ be a Haar unitary, and $h=h^*\in \mathrm{L}(\mathbb Z)$ such that $u=\exp(ih)$. For $t>0$, let $u_t=\exp(ith)$. 
   Since $M_1$ is diffuse, we have that $M_1\nprec_M\mathbb C1$.
   Applying Lemma \ref{rotate} with $A_1=A_2=\mathbb C1$ immediately yields that  $M_1\vee u_tM_2u_t^*$ is a II$_1$ factor, for every $t>0$. Since $\|u_t-1\|_2<\varepsilon$, for any small enough $t>0$, the conclusion follows.
\end{proof}

Using Lemma \ref{rotate}, we show next that the property of being approximately factorial is inherited from abelian subalgebras.

\begin{proposition} \label{prop: property A inherited from subalgs}

Let $(M,\tau)$ be a tracial von Neumann algebra and consider von Neumann subalgebras $A_1\subset M_1\subset M$ and $A_2\subsetneq M_2\subset M$. Assume that  $A_2$ is type I (e.g., abelian). 
     If $(A_1,A_2,M)$ is approximately factorial, then $(M_1,M_2,M)$  is approximately factorial.
\end{proposition}

\begin{proof}
    Assume that the triple $(A_1,A_2,M)$ is approximately factorial. Then for any $\eps>0$, there exist a II$_1$ factor $M_\varepsilon$ and 
        $v_\varepsilon\in \mathcal U(M_\varepsilon)$ such that $M\subset M_\varepsilon$, $\|v_\varepsilon-1\|_2<\varepsilon$, and $A_1 \vee v_\varepsilon A_2v_\varepsilon^*$ is a II$_1$ factor.
    
    Fix $\varepsilon>0$ and define $\widetilde{M}_\varepsilon=M_{\varepsilon}*_{v_{\varepsilon} A_2v_{\varepsilon}^*}(v_{\varepsilon} A_2v_{\varepsilon}^*\overline{\otimes}\mathrm{L}(\mathbb Z))$. Let $u\in \mathrm{L}(\mathbb Z)$ be a Haar unitary and $h=h^*\in \mathrm{L}(\mathbb Z)$ with $u=\exp(ih)$. Let $t>0$ such that $u_t=\exp(ith)$ satisfies $\|u_t-1\|_2<\varepsilon-\|v_\varepsilon-1\|_2$.
    
    Since $A_1 \vee v_\varepsilon A_2v_\varepsilon^*$ is a II$_1$ factor, we get that $M_1 \vee v_\varepsilon A_2v_\varepsilon^*$ is of type II$_1$. Since  $A_2$ is of type I, we derive that  $M_1 \vee v_\varepsilon A_2v_\varepsilon^*\nprec_{ M_\varepsilon} v_\varepsilon A_2v_\varepsilon^*$. Since we also assume that $A_2\subsetneq M_2$, we derive from Lemma \ref{rotate} that $M_1 \vee u_tv_{\varepsilon}M_2v_{\varepsilon}^*u_t^*$ is a II$_1$ factor.
    Since $\|u_tv_\varepsilon-1\|_2\leq \|u_t-1\|_2+\|v_\varepsilon-1\|_2<\varepsilon$, we conclude that the triple $(M_1,M_2,M)$ is approximately factorial.
\end{proof}

Lastly, we  establish approximate factoriality in the presence of projections of trace $\frac{1}{2}$ and factoriality of one of the algebras.

\begin{proposition} \label{prop: property A for projection 0.5}
    Let $M_1$ and $M_2$ be tracial von Neumann algebras, each admitting a projection of trace $\frac{1}{2}$. 
    Assume that at least one of $M_1$ and $M_2$ is a factor.
    Then the pair $(M_1,M_2)$ is approximately factorial. 
\end{proposition}
\begin{proof} 
Assume without loss of generality that $M_2$ is a factor. 
Let $(M,\tau)$ be a tracial von Neumann algebra containing $M_1$ and $M_2$, and let $\varepsilon>0$.
Let $p_1\in M_1$ and $p_2\in M_2$ be projections with $\tau(p_1)=\tau(p_2)=\frac{1}{2}$, and denote $A_1=\C p_1\oplus \C (1-p_1)$ and $A_2 = \C p_2 \oplus \C (1-p_2)$.
    
By \cite[Theorem~1.9]{CK12} (see also Proposition \ref{twoprojections}), there exist a II$_1$ factor $M_\varepsilon$ and $v_\varepsilon\in \mathcal U(M_\varepsilon)$ such that $M\subset M_\varepsilon$, $\|v_\varepsilon-1\|_2<\varepsilon$, and $A_1\vee v_\varepsilon A_2v_\varepsilon^*$ is diffuse.

Define $\widetilde{M}_\varepsilon=M_{\varepsilon}*_{v_{\varepsilon} A_2v_{\varepsilon}^*}(v_{\varepsilon} A_2v_{\varepsilon}^*\overline{\otimes}\mathrm{L}(\mathbb Z))$. Let $u\in \mathrm{L}(\mathbb Z)$ be a Haar unitary and $h=h^*\in \mathrm{L}(\mathbb Z)$ such that $u=\exp(ih)$. Let $t>0$ such that $u_t=\exp(ith)$ satisfies $\|u_t-1\|_2<\varepsilon-\|v_\varepsilon-1\|_2$.

Since $A_1\vee v_\varepsilon A_2v_\varepsilon^*$ is diffuse and $A_2$ is finite dimensional, we get that $M_1\vee v_\varepsilon A_2v_\varepsilon^*\nprec_{M_\varepsilon}v_\varepsilon A_2v_\varepsilon^*$. 
Since $M_2$ is a factor, $A_2\subsetneq M_2$ and thus $v_\varepsilon A_2v_\varepsilon^*\subsetneq v_\varepsilon M_2v_\varepsilon^*$. 
Using that $M_2$ is a factor again, we may thus apply Lemma \ref{rotate} and derive that $M_1\vee u_tv_{\varepsilon}M_2v_{\varepsilon}^*u_t^*$ is a II$_1$ factor.
It is left to note that $\|u_tv_\varepsilon-1\|_2\leq \|u_t-1\|_2+\|v_\varepsilon-1\|_2<\varepsilon$, so we conclude that the triple $(M_1,M_2,M)$ is approximately factorial.
\end{proof}

\subsection{General position of von Neumann subalgebras}
We continue with a technical result,  Theorem~\ref{abelian-perturbation} below, pertaining to von Neumann algebras of a specific form, and which we will rely on in the proof of Theorem~\ref{perturbation}. We postpone its proof to the next section. In order to state the theorem, we first introduce the following notation.

\begin{definition}
Let $\mathcal C$ be the class of tracial von Neumann algebras $(M,\tau)$  of the form $M=A\oplus B$, where $A$ is finite dimensional abelian and $B$ is diffuse. In other words, there exist projections $p_1,\ldots,p_m,p\in\mathcal Z(M)$ such that $\sum_{i=1}^mp_i+p=1$, $M=\mathbb Cp_1\oplus\cdots\oplus\mathbb Cp_m\oplus Mp$ and $Mp$ is diffuse.
\end{definition}

For algebras of class $\mathcal{C}$, Theorem~\ref{abelian-perturbation} establishes a more general version of Theorem \ref{perturbation}. 
To motivate its statement, note that, given von Neumann algebras $\C1\neq M_1,M_2\in\cC$ with $\dim(M_1)+\dim(M_2)\geq 5$, Theorem~\ref{perturbation} would imply that we can find a unitary $v$ close to $1$ (in some II$_1$ factor containing $M_1$ and $M_2$) such that $M_1\vee vM_2v^*$ is a II$_1$ factor \textit{if the condition $\mathrm{e}(M_1)+\mathrm{e}(M_2)\leq 1$ is satisfied}. If this condition fails, then that is never possible, since we will always have $1$-dimensional direct summands of the form $\C(p\wedge q)$ for central projections $p\in\cZ(M_1)$ and $q\in\cZ(vM_2v^*)$ satisfying $\tau(p)+\tau(q)>1$ (cf. the first paragraph of subsection~\ref{ssec:nodiff}). 
Nevertheless, Theorem~\ref{abelian-perturbation} tells us that, even if $\mathrm{e}(M_1)+\mathrm{e}(M_2) > 1$, we can find a unitary $v$ close to $1$ such that those are the only $1$-dimensional direct summands of $M_1\vee vM_2v^*$, they have the minimal possible support projections, and the remaining direct summand is a II$_1$ factor. We formalize this in the next definition.

Recall that two projections $p$ and $q$ in a tracial von Neumann algebra $(M,\tau)$ are said to be in general position if $\tau(p\wedge q) = \max\{\tau(p) + \tau(q) - 1, 0\}$. 

\begin{definition}\label{condition B} Let $M_1,M_2\in\mathcal C$. Write $M_1=\mathbb Cp_1\oplus \cdots\oplus\mathbb Cp_m\oplus M_1p$, $M_2=\mathbb Cq_1\oplus\cdots\oplus\mathbb Cq_n\oplus M_2q$, where $p_1,\ldots,p_m,p\in \mathcal Z(M_1), q_1,\ldots,q_n,q\in \mathcal Z(M_2)$ are projections and $M_1p,M_2q$ are diffuse.  

\begin{itemize}
    \item[(a)]
Let $(M,\tau)$ be a tracial von Neumann algebra which contains $M_1$ and $M_2$. \\ We say that the triple $(M_1,M_2,M)$ satisfies {\it condition (B)}  if for every $\varepsilon>0$, we can find a II$_1$ factor $M_\varepsilon$ and $v_\varepsilon\in\mathcal U(M_\varepsilon)$ such that $M\subset M_\varepsilon$, $\|v_\varepsilon-1\|_2<\varepsilon$,
\begin{enumerate}
\item $p_i$ and $v_\varepsilon q_jv_\varepsilon^*$ are in general position  
for every $1\leq i\leq m,1\leq j\leq n$, and 
\item $M_1\vee v_\varepsilon M_2 v_\varepsilon^*=\big(\bigoplus_{\substack{1\leq i\leq m \\ 1\leq j\leq n}}\mathbb C(p_i\wedge v_\varepsilon q_jv_{\varepsilon}^*)\big)\oplus B$, where $B$ is a II$_1$ factor.
\end{enumerate}
\item [(b)]
We say that the pair $(M_1,M_2)$  satisfies {\it condition (B)} if the triple $(M_1,M_2,M)$ satisfies condition (B) for every tracial von Neumann algebra $(M,\tau)$ which contains $M_1$ and $M_2$.
\end{itemize}
\end{definition}

 \begin{theorem}\label{abelian-perturbation}
Let $M_1, M_2\in\cC$ be such that $M_1,M_2\neq \mathbb C1$ and $\dim(M_1)+\dim(M_2)\geq 5$. Then $(M_1,M_2)$ satisfies condition (B).
\end{theorem}
For the proof of Theorem \ref{abelian-perturbation}, we refer to Section \ref{sec:gen_position}. In the remaining part of this section, we use Theorem \ref{abelian-perturbation} to establish Theorem~\ref{perturbation}.

\subsection{Proof  of Theorem \ref{perturbation}}

Before turning to the proof of Theorem~\ref{perturbation} in full generality, 
we note that Theorem~\ref{abelian-perturbation} immediately implies Theorem~\ref{perturbation} when $M_1$ and $M_2$ are finite dimensional abelian:

\begin{corollary}\label{abelian_perturbation}
Let $M_1$ and $M_2$ be finite dimensional abelian tracial von Neumann algebras. Assume that $M_1\neq \mathbb C1,M_2\neq \mathbb C1$, $\dim(M_1)+\dim(M_2)\geq 5$ and $\mathrm{e}(M_1)+\mathrm{e}(M_2)\leq 1$.
Then the pair $(M_1,M_2)$ is approximately factorial.
\end{corollary}

\begin{proof} Write $M_1=\mathbb Cp_1\oplus\cdots\oplus\mathbb Cp_m$ and $M_2=\mathbb Cq_1\oplus\cdots\oplus\mathbb Cq_n$, where $p_1,\ldots,p_m,q_1,\ldots,q_n$ are non-zero projections. Fix $\varepsilon>0$ and let $M$ be a tracial von Neumann algebra containing $M_1$ and $M_2$. By Theorem~\ref{abelian-perturbation}, we can find a II$_1$ factor $M\subset M_\eps$ and $v_\varepsilon\in\mathcal U(M_\varepsilon)$ with $\norm{v_\eps - 1}_2<\eps$ such that (1) and (2) in Definition \eqref{condition B}(a) hold.
If $1\leq i\leq m$ and $1\leq j\leq n$, then $\tau(p_i)+\tau(q_j)\leq\mathrm{e}(M_1)+\mathrm{e}(M_2)\leq 1$, thus $p_i\wedge v_\varepsilon q_jv_\varepsilon^*=0$ by (1). By (2) we thus get that $M_1\vee v_\varepsilon M_2v_\varepsilon^*=B$ is a II$_1$ factor. 
\end{proof}

\begin{proof}[Proof of Theorem \ref{perturbation}]
Let $(M,\tau)$ be a tracial von Neumann algebra, and let $M_1,M_2\subset M$ be von Neumann subalgebras such that $M_1,M_2\neq \mathbb C1$, $\mathrm{dim}(M_1)+\mathrm{dim}(M_2)\geq 5$, and $\mathrm{e}(M_1)+\mathrm{e}(M_2)\leq 1$.
Our goal is to prove that the triple $(M_1,M_2,M)$ is approximately factorial. 

If $M_1$ or $M_2$ is diffuse, then Corollary \ref{one diffuse algebra} implies that $(M_1,M_2,M)$ is approximately factorial. Thus, we may assume that neither $M_1$ nor $M_2$ is diffuse. By Lemma \ref{masa}, we can thus find a finite dimensional abelian von Neumann subalgebra  $A_i\subset M_i$, for $i=1,2$, such that $A_i\neq \mathbb C1$, $\mathrm{e}(A_i)=\mathrm{e}(M_i)$. Moreover $\dim(A_i)\geq 3$ unless $M_i$ is isomorphic to $\mathbb C^2$ or $\mathbb M_2(\mathbb C)$.

In particular, $\mathrm{e}(A_1)+\mathrm{e}(A_2)=\mathrm{e}(M_1)+\mathrm{e}(M_2)\leq 1$, and since $A_1\neq \mathbb C1$ and $A_2\neq \mathbb C1$, we get $\mathrm{dim}(A_1)+\mathrm{dim}(A_2)\geq 4$. 
We finish the proof by considering the following two cases:

{\bf Case 1.} $\mathrm{dim}(A_1)+\mathrm{dim}(A_2)\geq 5$.

In this case, Corollary \ref{abelian_perturbation} implies that the triple $(A_1,A_2,M)$ is approximately factorial. 
If $M_1=A_1$ and $M_2=A_2$, then we get that $(M_1,M_2,M)$ is approximately factorial. Otherwise, we may assume, without loss of generality, that $A_2\subsetneq M_2$. In this case, we obtain that the triple $(M_1,M_2,M)$ is approximately factorial by applying Proposition \ref{prop: property A inherited from subalgs}.

{\bf Case 2.} $\mathrm{dim}(A_1)+\mathrm{dim}(A_2)=4$.

In this case, $\mathrm{dim}(A_1)=\mathrm{dim}(A_2)=2$. Thus, $M_1$ and $M_2$ are isomorphic to either $\mathbb C^2$ or $\mathbb M_2(\mathbb C)$. Since $\mathrm{dim}(M_1)+\mathrm{dim}(M_2)\geq 5$, at least one of $M_1$ or $M_2$ is not isomorphic to $\mathbb C^2$. Thus, without loss of generality, we may assume that $M_2 \cong \mathbb M_2(\mathbb C)$.

Since $\mathrm{dim}(A_1)=\mathrm{dim}(A_2)=2$, we have that $A_1=\mathbb Cp_1\oplus\mathbb C(1-p_1)$ and $A_2=\mathbb Cp_2\oplus\mathbb C(1-p_2)$, for some projections $p_1,p_2\in M$ with $\tau(p_1)\geq\frac{1}{2}$ and $\tau(p_2)\geq\frac{1}{2}$. Since $1\geq \mathrm{e}(A_1)+\mathrm{e}(A_2)=\tau(p_1)+\tau(p_2)$, we conclude that $\tau(p_1)=\tau(p_2)=\frac{1}{2}$.
 We may thus proceed by applying Proposition \ref{prop: property A for projection 0.5} and conclude that the triple $(M_1,M_2,M)$ is approximately factorial.
\end{proof}

\section{General position of von Neumann subalgebras}\label{sec:gen_position}
In Section \ref{sec:perturbation} we proved Theorem \ref{perturbation} while relying on  Theorem \ref{abelian-perturbation}. The purpose of this section is to prove Theorem \ref{abelian-perturbation}.

\subsection{Pairs of \texorpdfstring{$2$}{2}-dimensional algebras}\label{sec:2D}
Before turning to the proof of Theorem~\ref{abelian-perturbation}, we note that the assumption that $\dim(M_1)+\dim(M_2)\geq 5$ from Theorem \ref{abelian-perturbation} is necessary. Indeed, if $M_1,M_2\neq\mathbb C1$, having $\dim(M_1)+\dim(M_2)<5$ forces that $\dim(M_1)=\dim(M_2)=2$, i.e., $M_1 \cong \C^2 \cong M_2$. In this case, $M_1\vee vM_2v^*$ is a type I von Neumann algebra, and therefore has no type II$_1$ direct summand, for any unitary $v$ in any II$_1$ factor containing $M_1$ and $M_2$.

Nevertheless, the case $\dim(M_1)=\dim(M_2)=2$ will be the starting point in our approach to Theorem~\ref{abelian-perturbation}. To this end, we prove in this subsection a weaker version of Theorem \ref{abelian-perturbation} for this case, see Proposition \ref{twoprojections} below. The weaker property we establish is essentially obtained by relaxing the requirement from Definition~\ref{condition B}(a) that $B$ should be a II$_1$ factor, and only ask $B$ to be diffuse:

\begin{definition}\label{condition C} Let $M_1,M_2\in\mathcal C$. Write $M_1=\mathbb Cp_1\oplus \cdots\oplus\mathbb Cp_m\oplus M_1p$, $M_2=\mathbb Cq_1\oplus\cdots\oplus\mathbb Cq_n\oplus M_2q$, where $p_1,\ldots,p_m,p\in \mathcal Z(M_1), q_1,\ldots,q_n,q\in \mathcal Z(M_2)$ are projections and $M_1p,M_2q$ are diffuse.

\begin{itemize}
\item[(a)] 
Let $(M,\tau)$ be a tracial von Neumann algebra which contains $M_1$ and $M_2$. \\
We say that the triple $(M_1,M_2,M)$ satisfies {\it condition (C)}  if for every $\varepsilon>0$, we can find a II$_1$ factor $M_\varepsilon$ and $v_\varepsilon\in\mathcal U(M_\varepsilon)$ such that $M\subset M_\varepsilon$, $\|v_\varepsilon-1\|_2<\varepsilon$, 
\begin{enumerate}
\item $p_i$ and $v_\varepsilon q_jv_\varepsilon^*$ are in general position, 
for every $1\leq i\leq m,1\leq j\leq n$, 
\item $M_1\vee v_\varepsilon M_2 v_\varepsilon^*=\big(\bigoplus_{\substack{1\leq i\leq m \\ 1\leq j\leq n}}\mathbb C(p_i\wedge v_\varepsilon q_jv_{\varepsilon}^*)\big)\oplus B$, where $B$ is diffuse, and
\item $\mathrm{E}_{\cZ(B)}(p_i),\mathrm{E}_{\cZ(B)}(v_\varepsilon q_jv_\varepsilon^*)\in\mathbb Cr$, for every $1\leq i\leq m,1\leq j\leq n$, where we denote by $r=1-\sum_{\substack{1\leq i\leq m\\1\leq j\leq n}}p_i\wedge v_\varepsilon q_jv_\varepsilon^*$ the unit projection of $B$.
 
\end{enumerate}
\item [(b)]
We say that the pair $(M_1,M_2)$  satisfies {\it condition (C)} if the triple $(M_1,M_2,M)$ satisfies condition (C) for every tracial von Neumann algebra $(M,\tau)$ which contains $M_1$ and $M_2$.
\end{itemize}
\end{definition}

\begin{proposition}\label{twoprojections}
Let $M_1, M_2$ be tracial von Neumann algebras such that $\dim(M_1)=\dim(M_2)=2$. Then $(M_1,M_2)$ satisfies condition (C). Moreover, if $M$ is a II$_1$ factor which contains $M_1$ and $M_2$, then in Definition \ref{condition C}(a), we can take $M_\varepsilon=M$, for every $\varepsilon>0$.
\end{proposition}

Write $M_1=\mathbb Cp\oplus\mathbb C(1-p)$ and $M_2=\mathbb Cq\oplus\mathbb C(1-q)$, for projections $p,q$ and let $M$ be a tracial von Neumann algebra which contains $M_1$ and $M_2$.  The main assertion of Proposition \ref{twoprojections} follows from \cite[Theorem 1.9]{CK12}, if $\tau(p)=\tau(q)=\frac{1}{2}$, and  \cite[Theorem 1.1]{Ha17}, in general.
These results imply that Definition \ref{condition C}(a) holds for $M_\varepsilon=M*\mathrm{L}(\mathbb F_\infty)$ and $v_\varepsilon=u_{t_\varepsilon}$, for any $t_\varepsilon>0$ with $\|v_\varepsilon-1\|_2<\varepsilon$, where $(u_t)_{t\geq 0}\subset\mathcal U(\mathrm{L}(\mathbb F_\infty))$ is a free unitary Brownian motion that is free from $M$. 

Below, we give a self-contained and elementary proof of Proposition \ref{twoprojections} which additionally implies its moreover assertion. We point out however, that the moreover assertion is not needed for our main results, and thus relying on \cite[Theorem 1.9]{CK12} and \cite[Theorem 1.1]{Ha17} is sufficient. 
In preparation for the proof of Proposition \ref{twoprojections}, we introduce the following notation: 

For $t\in [0,1]$, put  $$e_t=\begin{pmatrix} 1-t &\sqrt{t(1-t)} \\ \sqrt{t(1-t)}& t\end{pmatrix}\;\;\;\text{and}\;\;\;\;u_t=\begin{pmatrix}\sqrt{1-t}&-\sqrt{t}\\\sqrt{t}&\sqrt{1-t} \end{pmatrix}.$$
Note that any projection of normalized trace $\frac{1}{2}$ in $\mathbb M_2(\mathbb C)$ is equal to $e_t$, for some $t\in [0,1]$, and that $u_t\in\mathbb M_2(\mathbb C)$ is a unitary satisfying $e_t=u_te_0u_t^*$, for every $t\in [0,1]$. It is easy to check that \begin{equation}\label{u_tu_s}\text{$\|u_t-u_s\|_2\leq \sqrt{2|t-s|}$, for every $t,s\in [0,1]$.}\end{equation}

The following elementary result will be needed in the proof of Proposition \ref{twoprojections}.

\begin{lemma}\label{rotate and generate}
Let $A$ be a separable diffuse abelian von Neumann algebra and $p,q\in A\overline{\otimes}\mathbb M_2(\mathbb C)$ be projections such that $\mathrm{E}_{A\overline{\otimes}1}(p)=\mathrm{E}_{A\overline{\otimes} 1}(q)=\frac{1}{2}$. Then for any  $\varepsilon>0$, there exists a unitary $v\in A\overline{\otimes}\mathbb M_2(\mathbb C)$ such that
$\|v-1\|_2<\varepsilon$ and 
 $\{p,vqv^*\}''=A\overline{\otimes}\mathbb M_2(\mathbb C)$.
\end{lemma}

\begin{proof}
Since $A$ is separable, diffuse and abelian, we can identify it with $\mathrm{L}^\infty([0,1],\lambda)$, where $\lambda$ is the Lebesgue measure on $[0,1]$.
Since $\mathrm{E}_{A\overline{\otimes}1}(p)=\frac{1}{2}$, we may assume that $p=1\otimes e_0$.
Since $\mathrm{E}_{A\overline{\otimes} 1}(q)=\frac{1}{2}$, we can find a measurable function $f:[0,1]\rightarrow [0,1]$ such that  $q(t)=e_{f(t)}$, for every $t\in [0,1]$.

Let $\varepsilon\in (0,1)$. 
We continue with the following claim:

\begin{claim}\label{approx}
There is a measurable function $g:[0,1]\rightarrow (0,1)$  such that $\{g\}''=A$ and $\|g-f\|_1<\frac{\varepsilon^2}{2}$. 
\end{claim}

{\it Proof of Claim \ref{approx}.}
First, we find $n\geq 1$ and pairwise distinct numbers $\lambda_0,\ldots,\lambda_{n-1}\in (0,1)$ such that $\|h-f\|_1<\frac{\varepsilon^2}{4}$, where $h=\sum_{k=0}^{n-1}\lambda_k{\bf 1}_{[\frac{k}{n},\frac{k+1}{n})}$.
Let $\delta\in (0,\frac{\varepsilon^2}{4})$ such that $\lambda_i+\delta<1$, for every $i$, and $|\lambda_i-\lambda_j|>\delta$, for every $i\neq j$.
Let $z:[0,1]\rightarrow [0,1]$ be the identity function $z(t)=t$. Define $g=\sum_{k=0}^{n-1}(\lambda_k+\delta z){\bf 1}_{[\frac{k}{n},\frac{k+1}{n})}$. Since $\|g-h\|_1\leq\delta<\frac{\varepsilon^2}{4}$, we get that $\|g-f\|_1<\frac{\varepsilon^2}{2}$. 
Then for every $0\leq k < n$, we have $(\lambda_k+\delta z){\bf 1}_{[\frac{k}{n},\frac{k+1}{n})}={\bf 1}_{[\lambda_k,\lambda_k+\delta]}(g)g\in\{g\}''$.
Since $(\lambda_k+\delta z){\bf 1}_{[\frac{k}{n},\frac{k+1}{n})}$ generates $\mathrm{L}^\infty([\frac{k}{n},\frac{k+1}{n}))$ as a von Neumann algebra, we get that $\mathrm{L}^\infty([\frac{k}{n},\frac{k+1}{n}))\subset \{g\}''$, for every $0\leq k\leq n-1$. This implies $\{g\}''=A$, and finishes the proof of the claim.\qed

Let $g$ be the function given by Claim \ref{approx}, and let $r\in A\overline{\otimes}\mathbb M_2(\mathbb C)$ be the projection given by $r_t=e_{g(t)}$, for all $t\in [0,1]$.  Let $v\in A\overline{\otimes}\mathbb M_2(\mathbb C)$ be the unitary given by $v_t=u_{g(t)}u_{f(t)}^*$, for all $t\in [0,1]$. Then $v_te_{f(t)}v_t^*=e_{g(t)}$, for all $t\in [0,1]$, and thus $vqv^*=r$. Using \eqref{u_tu_s} we get that $$\|v-1\|_2^2=\int_0^1\|v_t-1\|_2^2\;\mathrm{d}t=\int_0^1\|u_{g(t)}-u_{f(t)}\|_2^2\;\mathrm{d}t\leq 2\int_0^1|g(t)-f(t)|\;\mathrm{d}t<\varepsilon^2,$$
and therefore $\|v-1\|_2<\varepsilon$. 

Put $B=\{p,vqv^*\}''=\{p,r\}''$. Since $prp=(1-g)\otimes e_0$, $(1-p)r(1-p)=g\otimes e_1$ and $\{g\}''=A$, we get that $pBp=A\otimes e_0$ and $(1-p)B(1-p)=A\otimes e_1$. As $(pr(1-p))(t)=\begin{pmatrix} 0&\sqrt{g(t)(1-g(t))}\\ 0&0\end{pmatrix}$ and $g(t)\not\in\{0,1\}$, for all $t\in [0,1]$, the partial isometry in the polar decomposition of $pr(1-p)$ is equal to $1\otimes d$, where $d=\begin{pmatrix} 0&1\\0&0\end{pmatrix}$.  Thus, $1\otimes d\in B$. Since $A\otimes e_0,A\otimes e_1$ and $1\otimes d$ generate $A\overline{\otimes}\mathbb M_2(\mathbb C)$, we conclude that $B=A\overline{\otimes}\mathbb M_2(\mathbb C)$, which finishes the proof of the lemma.
\end{proof}

\begin{proof}[Proof of Proposition \ref{twoprojections}]
 Write $M_1=\mathbb Cp\oplus\mathbb C(1-p)$ and $M_2=\mathbb Cq\oplus\mathbb C(1-q)$, where $p,q\in M$ are projections and $\tau(p),\tau(q)\geq\frac{1}{2}$. We assume without loss of generality that $\tau(p)\geq\tau(q)$. 

Consider a II$_1$ factor $M$  containing $M_1$ and $M_2$.  Let $\varepsilon>0$ and put $N=\{p,q\}''$. Then there are projections $z_1,z_2\in \mathcal Z(N)$ such that  $z_1+z_2=1$, $Nz_1$ is abelian and $Nz_2$ is of type $I_2$. 
Moreover, writing  $Nz_2=C\overline{\otimes}\mathbb M_2(\mathbb C)$, where $C$ is an abelian von Neumann algebra, we claim that  
\begin{equation}\label{expectation}\mathrm{E}_{C\overline{\otimes}1}(pz_2)=\mathrm{E}_{C\overline{\otimes}1}(qz_2)=\frac{z_2}{2}.\end{equation}
Indeed, $pz_2$ and $qz_2$ are projections that generate $C\overline{\otimes} \bM_2(\C)\cong \dint_X \bM_2(\C) \,\mathrm{d}\mu(x)$, where we identify $C\cong L^\infty(X,\mu)$ for some standard probability space $(X,\mu)$. Disintegrating $pz_2 = \dint_X p_x \,\mathrm{d}\mu(x)$, we see that each projection $p_x\in\bM_2(\C)$ necessarily has trace $0$, $\frac{1}{2}$, or $1$. If the measure of $X_0\coloneq\{x\in X\mid p_x=0\}$ is strictly positive, then $\{pz_21_{X_0},qz_21_{X_0}\}'' = \{qz_21_{X_0}\}''\neq \mathrm{L}^\infty(X_0)\overline{\otimes}\bM_2(\C)$, contradicting the fact that $pz_2$ and $qz_2$ generate $\mathrm{L}^\infty(X)\overline{\otimes}\bM_2(\C)$. One can argue similarly for the points where $p_x=1$, and thus deduce that for almost every $x\in X$, $\tau(p_x)=\frac{1}{2}$. A similar argument holds for $qz_2$, yielding \eqref{expectation}.

Let $A_2\subset \mathbb M_2(\mathbb C)'\cap z_2Mz_2$ be a MASA containing $C$. Then $Nz_2\subset A_2\overline{\otimes}\mathbb M_2(\mathbb C)$ and since we have $\mathrm{E}_{A_2\overline{\otimes}1}\circ\mathrm{E}_{C\overline{\otimes}\mathbb M_2(\mathbb C)}=\mathrm{E}_{C\overline{\otimes}1}$, \eqref{expectation} implies that 
\begin{equation}\label{E_A2}
\mathrm{E}_{A_2\overline{\otimes}1}(pz_2)=\mathrm{E}_{A_2\overline{\otimes}1}(qz_2)=\frac{z_2}{2}. 
\end{equation}
Since $M$ is diffuse, $A_2$ is  diffuse. 
 Thus, by applying Lemma \ref{rotate and generate}, we can find a unitary $v_2\in z_2Mz_2$ such that \begin{equation}\label{v_2}\|v_2-z_2\|_2<\frac{\varepsilon}{2} 
 \;\;\;\text{and}\;\;\;
 \{pz_2,v_2(qz_2)v_2^*\}''=A_2\overline{\otimes}\mathbb M_2(\mathbb C).\end{equation}
 
Next, \eqref{expectation} implies that $\tau(pz_2)=\tau(qz_2)=\frac{\tau(z_2)}{2}$. Hence $\tau(pz_1)\geq \tau(qz_1)\geq\frac{\tau(z_1)}{2}$. Since $Nz_1$ is abelian, $r\wedge s=rs$, for any projections $r,s\in Nz_1$.
 Therefore, $\tau(pqz_1)\geq\tau(pz_1)+\tau(qz_1)-\tau(z_1)\geq 0$ and $\tau(p(1-q)z_1)\geq\tau(pz_1)-\tau(qz_1)\geq 0$.
 
Since $M$ is diffuse, there
is a projection $z_3\in M$ such that $z_3\leq pz_1$, $z_3$ commutes with $q$,
 \begin{equation}\label{z_3}\tau(qz_3)=\tau(pz_1)+\tau(qz_1)-\tau(z_1)\;\;\;\text{and}\;\;\;\tau((1-q)z_3)=\tau(pz_1)-\tau(qz_1).
 \end{equation}
  Since $z_3\leq p$, we have that $\tau(pz_3)+\tau(qz_3)-\tau(z_3)=\tau(qz_3)$ and $\tau(pz_3)-\tau(qz_3)=\tau((1-q)z_3)$. By combining this with \eqref{z_3}, we get that $z_4=z_1-z_3$ satisfies $\tau(pz_4)+\tau(qz_4)-\tau(z_4)=0$ and $\tau(pz_4)-\tau(qz_4)=0$. In other words, we have that 
  \begin{equation}
  \tau(pz_4)=\tau(qz_4)=\frac{\tau(z_4)}{2}.
  \end{equation}
Let $A_4=\mathrm{L}^\infty([0,1],\lambda)$, where $[0,1]$ is endowed with its Lebesgue measure $\lambda$. Since $M$ is a II$_1$ factor, we can find a unital trace preserving embedding $A_4\overline{\otimes}\mathbb M_2(\mathbb C)\subset z_4Mz_4$. 
Let $\tilde r\in A_4$ be a projection such that $\tau(\tilde r)=2\tau(pqz_4)$ and let $\tilde{p},\tilde{q}\in A_4\overline{\otimes}\mathbb M_2(\mathbb C)$ be given by 
 $\tilde{p}=1\otimes e_0$ and $\tilde{q}=\tilde r\otimes e_0+(1-\tilde r)\otimes e_1$. Then $\tau(\tilde{p})=\frac{\tau(z_4)}{2}=\tau(pz_4), \tau(\tilde{q})=\frac{\tau(z_4)}{2}=\tau(qz_4)$ and $\tau(\tilde{p}\tilde{q})=\tau(\tilde{r} \otimes e_0)=\tau(pqz_4)$.
 Using this and the fact that $M$ is a II$_1$ factor, we can find a unitary $u\in z_4Mz_4$ such that $pz_4=u\tilde{p}u^*$ and $qz_4=u\tilde{q}u^*$. Hence, after unitary conjugacy with $u$, we may assume that $pz_4=\tilde{p}$ and $qz_4=\tilde{q}$. 
 Then 
 \begin{equation}\label{E_A4}
 \mathrm{E}_{A_4\overline{\otimes}1}(pz_4)=\mathrm{E}_{A_4\overline{\otimes}1}(qz_4)=\frac{z_4}{2},
 \end{equation}
 and Lemma \ref{approx} gives a unitary $v_4\in z_4Mz_4$ such that
 \begin{equation}\label{v_4}\|v_4-z_4\|_2<\frac{\varepsilon}{2} \;\;\;\text{and}\;\;\;
 \{pz_4,v_4(qz_4)v_4^*\}''=A_4\overline{\otimes}\mathbb M_2(\mathbb C).\end{equation}

Since $z_3+z_4=z_1$, we have that $z_2+z_3+z_4=1$, and hence $v_\varepsilon\coloneq v_2+z_3+v_4\in M$ is a unitary. Furthermore, by using \eqref{v_2} and \eqref{v_4}, we get that $\|v_\varepsilon-1\|_2\leq \|v_2-z_2\|_2+\|v_4-z_4\|_2<\varepsilon$. 
We will prove that $v_\varepsilon$ satisfies (1)--(3) from Definition~\ref{condition C}(a).

First, let $p'\in\{p,1-p\}$ and $q'\in\{q,1-q\}$.  If two projections $r,s$ generate a diffuse von Neumann algebra, then $r'\wedge s'=0$, for every $r'\in \{r,1-r\}$ and $s'\in\{s,1-s\}$. Hence \eqref{v_2} and \eqref{v_4} imply that 
 $p'z_2\wedge v_2(q'z_2)v_2^*=0$ and $p'z_4\wedge v_4(q'z_4)v_4^*=0$. Thus, $p'\wedge v_\varepsilon q'v_\varepsilon^*=p'z_3\wedge q'z_3=p'q'z_3$. Since $z_3\leq p$, we derive that
 \begin{equation}\label{wedge1} (1-p)\wedge v_\varepsilon qv_\varepsilon^*=(1-p)\wedge v_\varepsilon(1-q)v_\varepsilon^*=0, \end{equation}
 \begin{equation}\label{wedge2} p\wedge v_\varepsilon qv_\varepsilon^*=qz_3\;\;\;\text{and}\;\;\;p\wedge v_\varepsilon(1-q)v_\varepsilon^*=(1-q)z_3.\end{equation} 
 Since  $\tau(pz_2)=\tau(qz_2)=\frac{\tau(z_2)}{2}$ by \eqref{expectation} and $z_1+z_2=1$, we get that $\tau(pz_1)+\tau(qz_1)-\tau(z_1)=\tau(p)+\tau(q)-1$ and $\tau(pz_1)-\tau(qz_1)=\tau(p)-\tau(q)$. Using \eqref{z_3} and \eqref{wedge2} we get that
 $$\tau(p\wedge v_\varepsilon qv_\varepsilon^*)=\tau(qz_3)=\tau(p)+\tau(q)-1\;\;\;\text{and}\;\;\;\tau(p\wedge v_\varepsilon(1-q)v_\varepsilon^*)=\tau((1-q)z_3)=\tau(p)-\tau(q).$$
 This shows that $p'$ and $v_\varepsilon q'v_\varepsilon^*$ are in general position, for every $p'\in \{p,1-p\}$ and $q'\in\{q,1-q\}$.
 
Second, let $N_\eps \coloneq M_1\vee v_\varepsilon M_2 v_\varepsilon^*=\{p,v_\varepsilon qv_\varepsilon^*\}''$. 
By \eqref{wedge2}, we deduce that $z_3=qz_3+(1-q)z_3\in N_\eps$.
Since $z_3\leq p$, by using \eqref{wedge1} and \eqref{wedge2} we get that $$N_\eps z_3=\mathbb Cqz_3\oplus \mathbb C(1-q)z_3=\bigoplus_{\substack{p'\in\{p,1-p\}\\ q'\in\{q,1-q\}}} \C (p'\wedge v_\varepsilon q'v_\varepsilon^*).$$
Since $N_\eps z_2=\{pz_2,v_2(qz_2)v_2^*\}''$ and $N_\eps z_4=\{pz_4,v_4(qz_4)v_4^*\}''$, by \eqref{v_2} and \eqref{v_4} we derive that $N_\eps z_2$ and $N_\eps z_4$ are diffuse. Thus, $N_\eps (z_2+z_4)$ is diffuse. Since $z_2+z_3+z_4=1$, this shows (2).

Finally, we check (3). By \eqref{v_2} and \eqref{v_4}, we have $B\coloneq N_\eps (z_2+z_4) = \big(A_2\overline{\otimes} \bM_2(\C)\big) \oplus \big(A_4\overline{\otimes} \bM_2(\C)\big)$, and the unit projection of $B$ is $r = z_2+z_4$. 
In particular, $\cZ(B) = (A_2\overline{\otimes} 1)\oplus (A_4\overline{\otimes} 1)$, and thus by \eqref{E_A2} and \eqref{E_A4} we get
\[
\text{E}_{\cZ(B)}(p) = \text{E}_{A_2\overline{\otimes} 1}(pz_2) + \text{E}_{A_4\overline{\otimes} 1}(pz_4) = \frac{z_2}{2} + \frac{z_4}{2}=\frac{r}{2}\in \C r,
\]
and similarly for $1-p$, $q$, and $1-q$. We conclude that all conditions from Definition~\ref{condition C}(a) hold, finishing the proof of the lemma.
\end{proof}

\subsection{Proof of Theorem \ref{abelian-perturbation}}
In \cite[Theorem 2.3]{Dy93}, Dykema showed that the free product of finite dimensional abelian von Neumann algebras $M_1=\mathbb Cp_1\oplus\cdots\oplus\mathbb Cp_m$ and $M_2=\mathbb Cq_1\oplus\cdots\oplus\mathbb Cq_n$ is given by $M_1*M_2=\big(\bigoplus_{\substack{1\leq i\leq m\\1\leq j\leq n}}\mathbb C(p_i\wedge q_j)\big)\oplus B$, for an interpolated free group factor $B$. The proof of Theorem \ref{abelian-perturbation} is inspired by the proof of \cite[Theorem 2.3]{Dy93}.
Paralleling \cite{Dy93}, we will prove Theorem \ref{abelian-perturbation} by induction, building up from the case $m=n=2$ treated in Proposition \ref{twoprojections}.

The following lemma is the main technical part of the proof. It is formulated in a general way, so that it can be used repeatedly in various steps of the induction process.

\begin{lemma}\label{induction2}
Let $M_1,M_2\in\mathcal C$. Let $p_0\in \mathcal Z(M_1)$ be a projection. 
Write $M_2=\mathbb Cq_1\oplus\cdots\oplus\mathbb Cq_n\oplus M_2q$, for projections $q_1,\ldots,q_n, q\in\mathcal Z(M_2)$ such that  $\sum_{j=1}^nq_j+q=1$ and $M_2q$ is diffuse.  
Assume that 
\begin{enumerate}\setcounter{enumi}{-1}
 \item $p_0\neq 0$, $q_j\neq 1$, for every $1\leq j\leq n$,
 \item the pair $(\mathbb Cp_0\oplus M_1(1-p_0),M_2)$  satisfies condition (C), and
 \item the pair $(M_1p_0,L)$ satisfies condition (B), for any $L\in\mathcal C$ with unit $p_0$ of the form $L=\mathbb Cr_1\oplus\cdots\oplus\mathbb Cr_n\oplus K$, with $r_j$ a projection of trace $\max(\tau(p_0)+\tau(q_j)-1,0)$, for all $1\leq j\leq n$, and $K$ diffuse.
\end{enumerate}

Then the pair $(M_1,M_2)$ satisfies condition (B).

\end{lemma}

\begin{proof} 
If $p_0=1$, there is nothing to prove, so we may assume that $p_0\neq 1$.
Let $(P,\tau)$ be any tracial von Neumann algebra containing $M_1$ and $M_2$. To prove the conclusion, we will show that the triple $(M_1,M_2,P)$ satisfies condition (B).

Let $\varepsilon>0$ and write $M_1=\mathbb Cp_1\oplus\cdots\oplus\mathbb Cp_m\oplus M_1p$, for projections $p_1,\ldots,p_m,p\in\mathcal Z(M_1)$ such that $\sum_{i=1}^mp_i+p=1$ and $M_1p$ is diffuse. Since $p_0\in \mathcal Z(M_1)$, we may assume that $p_0=\sum_{i=1}^kp_i+p'$, for some $0\leq k\leq m$ and a projection $p'\in\mathcal Z(M_1)p$. Let $p''=p-p'$.
Note that  we have $\mathbb Cp_0\oplus M_1(1-p_0)=\mathbb Cp_0\oplus\mathbb Cp_{k+1}\oplus\cdots\oplus\mathbb Cp_m\oplus M_1p''$ and $M_1p''$ is diffuse.

Since the pair  $(\mathbb Cp_0\oplus M_1(1-p_0),M_2)$  satisfies condition (C) by assumption (1), so does the triple $(\mathbb Cp_0\oplus M_1(1-p_0),M_2,P)$. 
Thus, we can find a II$_1$ factor $Q\supset P$ and $y\in \mathcal U(Q)$ with $\|y-1\|_2<\frac{\varepsilon}{2}$ such that, denoting $\widetilde q_j=yq_jy^*$, $\widetilde M_2=yM_2y^*$ and $N=(\mathbb Cp_0\oplus M_1(1-p_0))\vee \widetilde M_2$, we have that
\begin{enumerate}

\item[(a)] $p_i$ and $\widetilde q_j$ are in general position, for every $i\in\{0\}\cup\{k+1,\ldots, m\}$ and $1\leq j\leq n$, 
\item[(b)] $N=\bigoplus_{\substack{i\in\{0\}\cup\{k+1,\ldots,m\}\\1\leq j\leq n}}\mathbb C(p_i\wedge \widetilde q_j)\oplus A$,
 where $A$ is diffuse, and 
 \item[(c)] $\mathrm{E}_{\mathcal Z(A)}(p_i), \mathrm{E}_{\mathcal Z(A)}(\widetilde q_j)\in\mathbb Cr$, for every $i\in\{0\}\cup\{k+1,\ldots,m\}$ and $1\leq j\leq n$, where \\ $r=1-\sum_{\substack{i\in\{0\}\cup\{k+1,\ldots,m\}\\ 1\leq j\leq n}}p_i\wedge \widetilde q_j$ is the support projection of $A$.
\end{enumerate}
Since $p_0\not\in\{0,1\}$ and $q_j\neq 1$, for every $1\leq j\leq n$,  (a) and Lemma \ref{complement} together imply that $p_0-\sum_{j=1}^np_0\wedge\widetilde q_j\neq 0$. Thus, $p_0r=p_0-\sum_{j=1}^np_0\wedge\widetilde q_j\neq 0$ and hence $\tau(p_0r)\neq 0$. Since by (c) we have that $\mathrm{E}_{\mathcal Z(A)}(p_0)=\frac{\tau(p_0r)}{\tau(r)}r$, (b) implies that
\begin{equation}\label{centralsupport}
z_N(p_0)=1-\sum_{\substack{k+1\leq i\leq m\\1\leq j\leq n}}p_i\wedge\widetilde q_j.
\end{equation}
Further, note that (b) implies that \begin{equation}\label{N_0} \text{$p_0Np_0=\bigoplus_{1\leq j\leq n}\mathbb C(p_0\wedge \widetilde q_j)\oplus B$, where $B$ is diffuse.}\end{equation} 

Next, (a) gives that $\tau(p_0\wedge \widetilde q_j)=\max(\tau(p_0)+\tau(q_j)-1,0)$, for every $1\leq j\leq n$. Assumption (2) therefore implies that 
 the pair $(M_1p_0,p_0Np_0)$, and thus the triple $(M_1p_0,p_0Np_0,p_0Qp_0)$, satisfies condition (B). Since $M_1p_0= \mathbb Cp_1\oplus\cdots\oplus\mathbb Cp_k\oplus M_1p'$ with $M_1p'$ is diffuse, using the symmetry of condition (B) with respect to swapping indices, we obtain  II$_1$ factors $C$ and $R\supset p_0Qp_0$, as well as $w\in \mathcal U(R)$ such that $\|w-1\|_2<\frac{\varepsilon}{2}$ and, denoting $\widetilde p_i=wp_iw^*$, for $1\leq i\leq k$, 
 \begin{enumerate}
 \item[(d)] $\widetilde p_i$ and $p_0\wedge\widetilde q_j$ are in general position, for every $1\leq i\leq k$ and $1\leq j\leq n$.
 \item[(e)] $wM_1p_0w^*\vee p_0Np_0=\bigoplus_{\substack{1\leq i\leq k\\ 1\leq j\leq n}}\mathbb C(\widetilde p_i\wedge (p_0\wedge \widetilde q_j))\oplus C$.  
 \end{enumerate}
Let $1\leq i\leq k$ and $1\leq j\leq n$. Since $\widetilde p_i\leq p_0$, we get that $\widetilde p_i\wedge (p_0\wedge \widetilde q_j)=\widetilde p_i\wedge \widetilde q_j$. Using (a) and (d), we get that if $\widetilde p_i\wedge\widetilde q_j\neq 0$, then
$\tau(\widetilde p_i\wedge \widetilde q_j)=\tau(\widetilde p_i)+\tau(p_0\wedge \widetilde q_j)-\tau(p_0)=\tau(\widetilde p_j)+\tau(\widetilde q_j)-1$. Hence, $\widetilde p_i$ and $\widetilde q_j$ are in general position. For $k+1\leq i\leq m$, let $\widetilde p_i=p_i$.
Then (a), (d), and (e) imply that
 \begin{enumerate}
 \item[(f)] $\widetilde p_i$ and $\widetilde q_j$ are in general position, for every $1\leq i\leq m$ and $1\leq j\leq n$.
 \item[(g)] $wM_1p_0w^*\vee p_0Np_0=\bigoplus_{\substack{1\leq i\leq k\\ 1\leq j\leq n}}\mathbb C(\widetilde p_i\wedge\widetilde q_j)\oplus C$. 
 \end{enumerate}

Since $R$ and $Q$ are II$_1$ factors and $R\supset p_0Qp_0$, we can find a II$_1$ factor $S$ such that $S\supset Q$ and $p_0Sp_0=R$. Let $x\in\mathcal U(S)$ be given by $x=w+(1-p_0)$.
Denote $\widetilde M_1=xM_1x^*$ and $M=\widetilde M_1\vee \widetilde M_2$. 
Since $\widetilde M_1=w(M_1p_0)w^*\oplus M_1(1-p_0)$, by Lemma \ref{generation}(1) we get that $p_0Mp_0$ is generated by $p_0[(\mathbb Cp_0\oplus\widetilde M_1(1-p_0)\vee\widetilde M_2]p_0=p_0[(\mathbb Cp_0\oplus M_1(1-p_0)\vee\widetilde M_2]p_0=p_0Np_0$ and $\widetilde M_1p_0=w(M_1p_0)w^*$. Thus, (g) rewrites as
 \begin{enumerate}
 \item [(h)] $p_0Mp_0=\bigoplus_{\substack{1\leq i\leq k\\ 1\leq j\leq n}}\mathbb C(\widetilde p_i\wedge\widetilde q_j)\oplus C$. 
 \end{enumerate}
Also, note that $\widetilde p_i=xp_ix^*$, for every $1\leq i\leq m$.
 By combining Lemma \ref{generation}(2) and \eqref{centralsupport}, we get that $z_M(p_0)=z_N(p_0)=1-\sum_{\substack{k+1\leq i\leq m\\1\leq j\leq n}}\widetilde p_i\wedge \widetilde q_j$. Denote by $s=p_0-\sum_{\substack{1\leq i\leq k\\1\leq j\leq n}}\widetilde p_i\wedge\widetilde q_j$ the support projection of $C$. Since $\widetilde p_i\wedge \widetilde q_j\in\mathcal Z(M)$, for every $1\leq i\leq m,1\leq j\leq n$, we get that $$z_M(s)=z_M(p_0)-\sum_{\substack{1\leq i\leq k\\1\leq j\leq n}}\widetilde p_i\wedge\widetilde q_j=1-\sum_{\substack{1\leq i\leq m\\1\leq j\leq n}}\widetilde p_i\wedge \widetilde q_j.$$ 
 Thus, $M(1-z_M(s))=\bigoplus_{\substack{1\leq i\leq m\\1\leq j\leq n}}\mathbb C(\widetilde p_i\wedge\widetilde q_j)$. On the other hand, since
 $C=sMs$ is a II$_1$ factor, we get that $Mz_M(s)$ is also a II$_1$ factor. Finally, recall that $\widetilde M_1=xM_1x^*=\bigoplus_{1\leq i\leq m}\mathbb C\widetilde p_i\oplus x(M_1p)x^*$ and $\widetilde M_2=yM_2y^*=\bigoplus_{1\leq j\leq n}\mathbb C\widetilde q_j\oplus y(M_2q)y^*$.  Since $\|y^*x-1\|_2\leq \|x-1\|_2+\|y-1\|_2<\varepsilon$, it follows that the triple $(M_1,M_2,P)$ satisfies condition (B), which finishes the proof.
  \end{proof}

The following corollaries to Lemma~\ref{induction2} establish various cases of the induction process, and will provide us with all the necessary ingredients for the proof of Theorem~\ref{abelian-perturbation}.

\begin{corollary}\label{diffuse parts}
Let $M_1,M_2\in\mathcal C$ be tracial von Neumann algebras. Assume that there are central projections $p_0\in M_1\setminus\{0,1\}$ and $q_0\in M_2\setminus\{0,1\}$ such that 
\begin{enumerate}
\item $M_1p_0$ is diffuse and $M_1(1-p_0)=\mathbb C(1-p_0)$. 
\item $M_2q_0$ is either equal to $\mathbb Cq_0$ or diffuse, and $M_2(1-q_0)=\mathbb C(1-q_0)$. 
\end{enumerate}
Then the pair $(M_1,M_2)$ satisfies condition (B).
\end{corollary}

\begin{proof} 
Assume first that $M_2q_0=\mathbb Cq_0$, so $M_2=\mathbb Cq_0\oplus\mathbb C(1-q_0)$. Assume $L\in\mathcal C$ has unit $p_0$ and is of the form $L=\mathbb Cr_1\oplus\mathbb Cr_2\oplus K$, where $r_1,r_2$ are projections with $\tau(r_1)=\max(\tau(p_0)+\tau(q_0)-1,0)$ and $\tau(r_2)=\max(\tau(p_0)+\tau(1-q_0)-1,0)$ and $K$ is diffuse. Since $q_0\in M_2\setminus\{0,1\}$, we have that $\tau(r_1)<\tau(p_0)$ and $\tau(r_2)<\tau(p_0)$, hence $L\neq \mathbb Cp_0$.
Since  $M_1p_0$ is diffuse, Corollary~\ref{one diffuse algebra} implies that $(M_1p_0,L)$ satisfies condition (B).  
As the pair $(\mathbb Cp_0\oplus M_1(1-p_0),M_2)=(\mathbb Cp_0\oplus\mathbb C(1-p_0),\mathbb Cq_0\oplus\mathbb C(1-q_0))$ satisfies condition (C) by Proposition \ref{twoprojections}, Lemma \ref{induction2} implies that $(M_1,M_2)$ satisfies condition (B). 

Second, assume that $M_2q_0$ is diffuse. Then $(\mathbb Cp_0\oplus\mathbb C(1-p_0),M_2)$ satisfies condition (B) and hence condition (C) by the above.  Let $L\in\mathcal C$ with unit $p_0$ and of the form $L=\mathbb Cr\oplus K$, where $r$ is a projection with $\tau(r)=\max(\tau(p_0)+\tau(1-q_0)-1,0)$ and $K$ is diffuse. 
Since $q_0\in M_2\setminus\{0,1\}$, we have that $\tau(r)<\tau(p_0)$, hence $L\neq \mathbb Cp_0$.
Since  $M_1p_0$ is diffuse, Corollary~\ref{one diffuse algebra} again implies that $(M_1p_0,L)$ satisfies condition (B).  
 Lemma \ref{induction2} thus implies that $(M_1,M_2)$ satisfies condition (B).
\end{proof}

\begin{corollary}\label{2,2}
Let $M_1,M_2\in\mathcal C$. Assume that $M_1=\mathbb Cp_1\oplus\mathbb Cp_2\oplus M_1p$, for projections $p_1,p_2,p\in\mathcal Z(M_1)\setminus\{0\}$, with $M_1p$  diffuse. Assume that $\dim(M_2)=2$. Then $(M_1,M_2)$ satisfies condition (B).
\end{corollary}

\begin{proof}
Write $M_2=\mathbb Cq_1\oplus\mathbb Cq_2$, for projections $q_1,q_2$ such that $q_1+q_2=1$ and $\tau(q_1)\geq\tau(q_2)$. Put $p_0=p_1+p_2$. Since $M_1(1-p_0)=M_1p$ is diffuse and $p\neq 0$, Corollary \ref{diffuse parts}, implies that \begin{equation}\label{Cp_0}\text{($\mathbb Cp_0\oplus M_1(1-p_0),M_2)$ satisfies condition (B).}\end{equation}

Note that $\tau(p_1)+\tau(p_2)\in (0,1)=\cup_{n\geq 1}(\frac{n-1}{n},\frac{n}{n+1}]$.
We will prove that $(M_1,M_2)$ satisfies condition~(B) by induction on $n\geq 1$ such that $\tau(p_1)+\tau(p_2)\in (\frac{n-1}{n},\frac{n}{n+1}]$.

To prove the base case $n=1$, assume that $\tau(p_1)+\tau(p_2)\in (0,\frac{1}{2}]$. Let $L\in\mathcal C$ with support projection $p_0$ be of the form $L=\mathbb Cr_1\oplus\mathbb Cr_2\oplus K$ where $\tau(r_j)=\max(\tau(p_0)+\tau(q_j)-1,0)$, for every $1\leq j\leq 2$, and $K$ is diffuse. Since $\tau(p_0)=\tau(p_1)+\tau(p_2)\leq \frac{1}{2}$ and $\tau(q_2)\leq\frac{1}{2}$, we get that $r_2=0$, hence $L=\mathbb Cr_1\oplus K$.  Moreover, since $q_1\neq 1$, we have that $r_1\neq p_0$. By applying Corollary \ref{diffuse parts}, we get that $(M_1p_0,L)=(\mathbb Cp_1\oplus\mathbb Cp_2,\mathbb Cr_1\oplus K)$ satisfies condition (B). By combining this fact with \eqref{Cp_0}, Lemma \ref{induction2} implies that $(M_1,M_2)$ satisfies condition (B).

Assume that $n\geq 2$ is such that the inductive hypothesis holds for every $1\leq k\leq n-1$. Assume that $\tau(p_1)+\tau(p_2)\in (\frac{n-1}{n},\frac{n}{n+1}]$. Let $L\in\mathcal C$ with support projection $p_0$ be of the form $L=\mathbb Cr_1\oplus\mathbb Cr_2\oplus K$ where $\tau(r_j)=\max(\tau(p_0)+\tau(q_j)-1,0)$, for every $1\leq j\leq 2$, and $K$ is diffuse. If $r_2=0$, then the same argument as in the case $n=1$ implies that $(M_1p_0,L)$ satisfies condition (B). If $r_2\neq 0$, then since $\tau(r_2)\leq \tau(r_1)$ we also have $r_1\neq 0$. Moreover, since $\tau(p_0)=\tau(p_1)+\tau(p_2)\leq \frac{n}{n+1}$, we get that
\begin{align*}
\frac{\tau(r_1)}{\tau(p_0)}+\frac{\tau(r_2)}{\tau(p_0)}=\frac{\tau(p_0)+\tau(q_1)-1}{\tau(p_0)}+\frac{\tau(p_0)+\tau(q_2)-1}{\tau(p_0)}=2-\frac{1}{\tau(p_0)}\leq \frac{n-1}{n}.
\end{align*}
Since $\dim(M_1p_0)=2$,  the inductive hypothesis implies that $(M_1p_0,L)$ satisfies condition (B). 
By combining this fact with \eqref{Cp_0}, Lemma \ref{induction2} implies that $(M_1,M_2)$ satisfies condition (B). By induction, this proves the conclusion.
\end{proof}

\begin{corollary}\label{m,2}
Let $M_1,M_2\in\mathcal C$. Assume that $M_1=\mathbb Cp_1\oplus\cdots\oplus\mathbb Cp_m\oplus M_1p$, for projections $p_1,\ldots,p_m,p\in\mathcal Z(M_1)\setminus\{0\}$, where $m\geq 0$ and $M_1p$ is diffuse. Assume that $\dim(M_2)=2$. Then $(M_1,M_2)$ satisfies condition (B).
\end{corollary}

\begin{proof} Write $M_2=\mathbb Cq_1\oplus\mathbb Cq_2$, for projections $q_1,q_2$. 
If $m=0$, then $M_1$ is diffuse and the conclusion follows from Corollary \ref{one diffuse algebra}. 
If $m=1$, the conclusion is a consequence of Corollary \ref{diffuse parts}. 

We will prove the conclusion for $m\geq 2$ by induction on $m$. If $m=2$, the 
conclusion follows from Corollary \ref{2,2}.
If $m\geq 3$, put $p_0=p_1+p_2$. Since $\mathbb Cp_0\oplus M_1(1-p_0)=\mathbb Cp_0\oplus\mathbb Cp_2\oplus\cdots\oplus\mathbb Cp_m\oplus M_1p$,  $\mathbb (\mathbb Cp_0\oplus M_1(1-p_0),M_2)$ satisfies condition (B) by the inductive hypothesis. Let $L\in\mathcal C$ with support projection $p_0$ be of the form $L=\mathbb Cr_1\oplus\mathbb Cr_2\oplus K$ where $\tau(r_j)=\max(\tau(p_0)+\tau(q_j)-1,0)$, for every $1\leq j\leq 2$, and $K$ is diffuse. 
The proof of Lemma~\ref{complement} implies that $r_1+r_2\neq p_0$.
Since $\dim(M_1p_0)=2$ and the  conclusion holds if $m\in\{0,1,2\}$, we thus get by symmetry of condition (B) with respect to swapping indices that $(M_1p_0,L)$ satisfies condition (B). By Lemma \ref{induction2} we conclude that $(M_1,M_2)$ satisfies condition (B), which finishes the proof of the corollary.
\end{proof}

\begin{corollary}\label{m,n with diffuse}
    Let $M_1,M_2\in\mathcal C$. Assume that $M_1=\mathbb Cp_1\oplus\cdots\oplus \mathbb Cp_m\oplus M_1p$, for projections $p_1,\ldots,p_m,p\in\mathcal Z(M_1)\setminus\{0\}$, where $m\geq 0$ and $M_1p$ is diffuse. Assume $M_2\neq \C 1$. Then $(M_1,M_2)$ satisfies condition (B).
\end{corollary}
\begin{proof}
Write $M_2=\mathbb Cq_1\oplus\cdots\oplus\mathbb Cq_n\oplus M_2q$, for some $n\geq 0$, with $q_1, \ldots, q_n$ non-zero and $M_2q$ diffuse. We note that, if $n=0$, then $M_2=M_2q$ is diffuse, and the result follows from Corollary~\ref{one diffuse algebra}. Thus, we may assume $n\geq 1$. 

First, assume $n=1$, and thus necessarily $q\neq 0$. Corollary~\ref{m,2} implies that $(M_1, M_2(1-q) \oplus \C q) = (M_1, \C (1-q) \oplus \C q)$ satisfies condition (B). Let $L\in\cC$ with support projection $q$ be of the form $L = \C r_1\oplus\cdots\oplus \C r_m \oplus K$ with $r_i$ a projection of trace $\max(\tau(p_i) + \tau(q) - 1, 0)$, for all $1\leq i\leq m$, and $K$ diffuse. Since $\tau(r_i)<\tau(q)$, for all $1\leq i\leq m$, we have that $L\neq \mathbb Cq$.
Since $M_2q$ is diffuse, Corollary~\ref{one diffuse algebra} implies that $(L,M_2q)$ satisfies condition (B). Hence, Lemma~\ref{induction2} implies that $(M_1,M_2)$ satisfies condition (B). 

Next, if $n=2$ and $q=0$, then the conclusion follows from Corollary~\ref{m,2}.

We proceed by induction on $k$, where $k$ is the number of non-zero projections appearing in the description of $M_2$ (i.e., $k=n$ if $q=0$ and $k=n+1$ otherwise). The base case $k=2$ was proved above. Assume the conclusion holds for some $k\geq 2$, and consider the case $k+1$. Since $k+1\geq 3$, we necessarily have $n\geq 2$. Define $q_0 = q_1+q_2$. Then by the induction hypothesis, $(M_1, \C q_0\oplus M_2(1-q_0))$ satisfies condition (B). Moreover, since $M_2q_0$ is 2-dimensional, Corollary~\ref{m,2} implies that $(L,M_2q_0)$ satisfies condition (B) for every $L\in \cC$ with support projection $q_0$ of the form $L = \C r_1\oplus\cdots\oplus \C r_m \oplus K$ with $r_i$ a projection of trace $\max(\tau(p_i) + \tau(q_0) - 1, 0)$, for all $1\leq i\leq m$, and $K$ diffuse. Hence, Lemma~\ref{induction2} implies that $(M_1,M_2)$ satisfies condition (B). This finishes the inductive step and therefore the proof of the corollary.
\end{proof}

We now have all the ingredients required for the proof of Theorem~\ref{abelian-perturbation}:

\begin{proof}[Proof of Theorem \ref{abelian-perturbation}]
When one of $M_1$ or $M_2$ has a non-zero diffuse direct summand, the result follows from Corollary~\ref{m,n with diffuse}. Hence we can assume that $M_1=\mathbb Cp_1\oplus\cdots\oplus\mathbb Cp_m$ and $M_2=\mathbb Cq_1\oplus\cdots\oplus\mathbb Cq_n$, for non-zero projections $p_1,\ldots,p_m,q_1,\ldots,q_n$, where $m,n\geq 2$ and $m+n\geq 5$. 

We will prove the conclusion by induction on $m+n$. Assume without loss of generality that $m\geq 3$ and write $p_0=p_1+p_2$.  Let $L\in\mathcal C$ with support projection $p_0$ be of the form $L=\mathbb Cr_1\oplus\cdots\oplus\mathbb Cr_n \oplus K$ where $\tau(r_j)=\max(\tau(p_0)+\tau(q_j)-1,0)$, for every $1\leq j\leq n$, and $K$ is diffuse. 
The proof of Lemma~\ref{complement} implies that $r_1+\cdots+r_n\neq p_0$.
Since $\mathrm{dim}(M_1p_0)=2$, Corollary \ref{m,2} implies that $(M_1p_0,L)$ satisfies condition (B).

For the base case, $m+n=5$, hence $m=3$ and $n=2$. Thus $\dim(\mathbb Cp_0\oplus M_1(1-p_0))=\dim(M_2)=2$ and therefore $(\mathbb Cp_0\oplus M_1(1-p_0),M_2)$ satisfies condition (C) by Proposition \ref{twoprojections}. 
 Lemma \ref{induction2} thus implies that $(M_1,M_2)$ satisfies condition (B).
 
Now assume that the conclusion holds if $m+n=k$, for some $k\geq 5$.  If $m+n=k+1$, then the inductive hypothesis implies that $(\mathbb Cp_0\oplus M_1(1-p_0),M_2)$ satisfies condition (B). 
Hence Lemma \ref{induction2} again implies that $(M_1,M_2)$ satisfies condition (B). This proves the inductive step and finishes the proof.
\end{proof}

\section{The Poulsen simplex as the trace simplex of a free product}\label{sec:Poulsen}

In this section we will prove Theorem~\ref{Poulsen}, thereby characterising exactly when the trace simplex of a 
free product $C^*$-algebra is a Poulsen simplex. 

Let $A_1, A_2$ be unital, separable $C^*$-algebras and put $A=A_1*A_2$. 
Any trace $\varphi\in\mathrm{T}(A)$ gives rise to a triple $(M_1,M_2,M)$, where $\pi:A\rightarrow M$ is the tracial representation associated to $\varphi$ and $M_i=\pi(A_i)''\subset M$, for $i\in\{1,2\}$.

\begin{definition}
    We say that a trace $\varphi\in\mathrm{T}(A)$ is {\it approximately factorial} if the corresponding triple $(M_1,M_2,M)$ is approximately factorial, as defined in Definition~\ref{condition (A)}. 
We denote by $\mathrm{T}_{\mathrm{apf}}(A)$ the collection of approximately factorial traces. 
\end{definition}

We start with the following observation. Recall that $\mathrm{T}_\infty(A) = \mathrm{T}(A)\setminus \mathrm{T}_{\mathrm{fin}}(A)$ denotes the set of traces whose GNS von Neumann algebra is infinite dimensional.

\begin{proposition}\label{prop:property A traces can be approxiamted}
$\mathrm{T}_{\mathrm{apf}}(A)\subset  \overline{\partial_{\mathrm{e}}\mathrm{T}(A)\cap \mathrm{T}_{\infty}(A)} \subset \overline{\partial_{\mathrm{e}}\mathrm{T}(A)}$. 
\end{proposition}
\begin{proof}
    
Let $\varphi\in\mathrm{T}_{\mathrm{apf}}(A)$ and $\pi:A\rightarrow M$ be the tracial representation associated to $\varphi$. Denoting $M_1=\pi(A_1)''$ and $M_2=\pi(A_2)''$, we have by assumption that the triple $(M_1,M_2,M)$ is approximately factorial. 
As a result, for any $n\in\mathbb N$, we can find a II$_1$ factor $P_n$ and $v_n\in\mathcal U(P_n)$ such that $M\subset P_n$,  $\|v_n-1\|_2<\frac{1}{n}$, and $M_1\vee v_nM_2v_n^*$ is a II$_1$ factor. Define the $*$-homomorphism $\pi_n:A\rightarrow P_n$ by letting $\pi_n(a_1)=\pi(a_1)$, for every $a_1\in A_1$, and $\pi_n(a_2)=v_n\pi(a_2)v_n^*$, for every $a_2\in A_2$. Then clearly $\|\pi_n(a)-\pi(a)\|_2\rightarrow 0$, for every $a\in A$. Thus, if $\varphi_n\coloneq \tau\circ\pi_n\in\mathrm{T}(A)$, then $\varphi_n\rightarrow\varphi$. Since $\pi_n(A)''=M_1\vee v_nM_1v_n^*$ is a II$_1$ factor, we get that $\varphi_n\in \partial_{\mathrm{e}}\mathrm{T}(A)\cap \mathrm{T}_{\infty}(A)$, for every $n\in\mathbb N$. This finishes the proof. 
\end{proof}

In view of Proposition \ref{prop:property A traces can be approxiamted}, in order to show that $\mathrm{T}(A)$ is a Poulsen simplex, it suffices to show that $\mathrm{T}_{\mathrm{apf}}(A)$ is dense inside $\mathrm{T}(A)$.
For this, it will be useful to consider the class of traces whose corresponding triple satisfies the conditions of Theorem~\ref{perturbation}.

\begin{definition}
Let $A_1, A_2$ be unital, separable $C^*$-algebras and put $A=A_1*A_2$. We define $\mathrm{T}_{\mathrm{l}}(A)$  (${\mathrm{l}}$ standing for large) to be the set of traces $\varphi\in\mathrm{T}(A)$ whose corresponding triple $(M_1,M_2,M)$ satisfies $M_1,M_2\neq \mathbb C1$, $\mathrm{dim}(M_1)+\mathrm{dim}(M_2)\geq 5$, and $\mathrm{e}(M_1)+\mathrm{e}(M_2)\leq 1$.
 \end{definition}

We note that Theorem \ref{perturbation} implies that $\mathrm{T}_{\mathrm{l}}(A)\subset \mathrm{T}_{\mathrm{apf}}(A)$.

\subsection{Perturbation of traces}

The main result of this section is as follows:

\begin{lemma}\label{2}
Let $A=A_1\ast A_2$ be the free product of two unital, separable $C^*$-algebras. Assume that $\mathrm{T}(A_i)$ is non-empty and does not consist of a single $1$-dimensional trace, for $i=1,2$. Assume that 
if $A_1$ (resp. $A_2$) has an isolated extreme 1-dimensional trace, then $A_2$ (resp. $A_1$) does not have an isolated extreme finite dimensional trace.
Then $\mathrm{T}(A)\subset\overline{\mathrm{T}_{\mathrm l}(A)}.$
\end{lemma}

\begin{proof}

Let $\varphi\in\mathrm{T}(A)$ and $\pi:A\rightarrow M$ be the tracial representation associated to $\varphi$. Denote $M_1=\pi(A_1)''$ and $M_2=\pi(A_2)''$.  Let $\mathcal M$ be a II$_1$ factor containing $M$.

We treat three cases separately:

 {\bf Case 1.} \textit{$A_1$ has no isolated extreme finite dimensional traces.}

 For $n\in\N$, we will define tracial representations $\pi_n:A\rightarrow \widetilde{\mathcal M}$, where $\widetilde{\mathcal M}\supset \mathcal M$, by first defining its restrictions to $A_1$ and $A_2$.

 If $M_2\neq \mathbb C$, we define $\mathcal M_2=\mathcal M$ and $\pi_{n,2}:A_2\rightarrow\mathcal M_2$ by $\pi_{n,2}=\pi_{|A_2}$ for every $n$. Then we have $\pi_{n,2}(A_2)''=M_2\neq \mathbb C$.
 If $M_2=\mathbb C$, then $\varphi_{|A_2}\in\mathrm{T}(A_2)$ is a $1$-dimensional trace. Since $\mathrm{T}(A_2)\neq \{\varphi_{|A_2}\}$ by assumption, we can find $\psi\in\partial_{\mathrm{e}}\mathrm{T}(A_2)\setminus\{\varphi_{|A_2}\}$. 
Let $\rho:A_2\rightarrow\mathcal N$ be the tracial representation associated to $\psi$. Let $p_n\in \mathcal M$ be a projection with $\tau(p_n)=1-\frac{1}{n}$. We define $\mathcal M_2=\mathcal M\overline{\otimes}\mathcal N$ and a $*$-homomorphism $\pi_{n,2}:A_2\rightarrow \mathcal M_2$ by letting $\pi_{n,2}(a_2)=\pi(a_2)p_n\otimes 1+(1-p_n)\otimes \rho(a_2)$, for every $a_2\in A_2$. 
Since $\varphi_{|A_2},\psi\in\partial_{\mathrm{e}}\mathrm{T}(A_2)$ are distinct, $\pi_{n,2}(A_2)''=\mathbb Cp_n\otimes 1\oplus (1-p_n)\otimes \mathcal N$. In particular, $\pi_{n,2}(A_2)''\neq \mathbb C$, for every $n\in\N$.

 In either case, we have that $\|\pi_{n,2}(a_2)-\pi(a_2)\|_2\rightarrow 0$, for every $a_2\in A_2$.
Since $\pi_{n,2}(A_2)''\neq \mathbb C$, we have $\mathrm{e}(\pi_{n,2}(A_2)'')<1$. Put $\delta_n\coloneq 1-\mathrm{e}(\pi_{n,2}(A_2)'')>0$. Then we have that \begin{equation}\label{A_2}\text{$\dim(\pi_{n,2}(A_2)'')\geq 2$ and $\mathrm{e}(\pi_{n,2}(A_2)'')= 1-\delta_n$}.\end{equation}

Next, we define representations $\pi_{n,1}:A_1\rightarrow\mathcal M_1$, for $n\in\N$, for an appropriate von Neumann algebra $\cM_1$ to be determined later. To this end, let $\{z_k\}_{k=0}^K\subset \mathcal Z(M_1)$, for some $K\in\mathbb N\cup\{\infty\}$, be (possibly zero) projections such that $\sum_{k=0}^K z_k=1$, $M_1z_0$ is diffuse, and $M_1z_k\cong\mathbb M_{l_k}(\mathbb C)$, for some $l_k\in\mathbb N$, for every $1\leq k\leq K$. For every $n\in\N$, fix some $L_n\in\mathbb N$ such that $L_n\geq\max\{3,\frac{1}{\delta_n}\}$. 

Fix  $1\leq k\leq K$ and let $\psi_k\in\partial_{\mathrm{e}}\mathrm{T}(A_1)$ be given by $\psi_k(a_1)=\frac{\tau(\pi(a_1)z_k)}{\tau(z_k)}$, for every $a_1\in A_1$. Since $\psi_k$ is $l_k$-dimensional, it is not isolated in $\partial_{\mathrm{e}}\mathrm{T}(A_1)$ by our assumption. Let $\psi_{n,k}\in\partial_{\mathrm{e}}\mathrm{T}(A_1)\setminus\{\psi_k\}$ be a sequence such that $\psi_{n,k}\rightarrow\psi_k$, as $n\rightarrow\infty$. Moreover, we may assume that $\psi_{n,k}\neq \psi_{n',k}$, whenever $n\neq n'$.
Defining $\eta_{n,k}=\frac{1}{L_n}\sum_{i=0}^{L_n-1}\psi_{n+i,k}\in\mathrm{T}(A_1)$, we also have that $\eta_{n,k}\rightarrow\psi_k$, as $n\rightarrow\infty$. Since $\psi_k$ is finite dimensional, Lemma \ref{convergence} implies the existence of a tracial von Neumann algebra $\mathcal N_k$ and $*$-homomorphisms $\rho_{n,k}:A_1\rightarrow \mathcal N_k$ such that $M_1z_k\subset\mathcal N_k$, $\|\rho_{n,k}(a_1)-\pi(a_1)z_k\|_2\rightarrow 0$, for every $a_1\in A_1$, and $\tau\circ\rho_{n,k}=\eta_{n,k}$, for every $n\in\mathbb N$.

Note that $M_1= \bigoplus_{k=0}^K M_1z_k\subset M_1z_0\oplus(\bigoplus_{k=1}^K \mathcal N_k)$
and put $\mathcal M_1=\mathcal M*_{M_1}(M_1z_0\oplus(\bigoplus_{k=1}^K\mathcal N_k))$. For every $n\in\N$, we define a $*$-homomorphism $\pi_{n,1}:A_1\rightarrow\mathcal M_1$ by letting $\pi_{n,1}(a_1)=\pi(a_1)z_0+\sum_{k=1}^K\rho_{n,k}(a_1)$. Then by construction $\|\pi_{n,1}(a_1)-\pi(a_1)\|_2\rightarrow 0$, for every $a_1\in A_1$.
We claim that for every $n\in\N$,
\begin{equation}\label{A_1}\text{$\dim(\pi_{n,1}(A_1)'')\geq L_n\geq 3$ and $\mathrm{e}(\pi_{n,1}(A_1)'')\leq\frac{1}{L_n}\leq \delta_n$}.\end{equation}
Indeed, note that 
$\{z_k\}_{k=0}^K\subset \pi_{n,1}(A_1)'\cap\mathcal M_1$.
If $z_0\neq 0$, then $\pi_{n,1}(A_1)''z_0=\pi(A_1)''z_0=M_1z_0$ is diffuse. Thus, $\mathrm{dim}(\pi_{n,1}(A_1)''z_0)=\infty$ and $\mathrm{e}(\pi_{n,1}(A_1)''z_0)=0$. Otherwise, there is some $1\leq k\leq K$ such that $z_k\neq 0$. Then $\pi_{n,1}(A_1)''z_k=\rho_{n,k}(A_1)''$.
Since $\tau\circ\rho_{n,k}=\eta_{n,k}=\frac{1}{L_n}\sum_{i=0}^{L_n-1}\psi_{n+i,k}\in\mathrm{T}(A_1)$, and $\psi_{n+i,k}\in\partial_{\mathrm{e}}\mathrm{T}(A_1)$, $0\leq i\leq L_n-1$, are pairwise distinct, we get that $\dim(\rho_{n,k}(A_1)'')\geq L_n$ and $\mathrm{e}(\rho_{n,k}(A_1)'')\leq\frac{\tau(z_k)}{L_n}$. Consequently, we get that $\dim(\pi_{n,1}(A_1)''z_k)\geq L_n$ and $\mathrm{e}(\pi_{n,1}(A_1)''z_k)\leq\frac{\tau(z_k)}{L_n}$. Since these inequalities hold for every $1\leq k\leq K$ with $z_k\neq 0$, Lemma~\ref{e wrt central proj} implies that \eqref{A_1} holds.

Define $\widetilde{\mathcal M}=\mathcal M_1*_\mathcal M\mathcal M_2$ and $*$-homomorphisms $\pi_n:A\rightarrow\widetilde{\mathcal M}$, for $n\in\N$, by letting $\pi_n(a_1)=\pi_{n,1}(a_1)$ and $\pi_n(a_2)=\pi_{n,2}(a_2)$, for every $a_1\in A_1$ and $a_2\in A_2$. Then $\|\pi_n(a)-\pi(a)\|_2\rightarrow 0$, for every $a\in A$. Since $\pi_n(A_1)=\pi_{n,1}(A_1),\pi_n(A_2)=\pi_{n,2}(A_2)$,
\eqref{A_2} and \eqref{A_1} imply that $\dim(\pi_n(A_1)'')\geq 3,\dim(\pi_n(A_2))\geq 2$ and $\mathrm{e}(\pi_n(A_1)'')+\mathrm{e}(\pi_n(A_2)'')\leq 1$. Thus, $\varphi_n\coloneq \tau\circ\pi_n\in\mathrm{T}_{\mathrm{l}}(A)$. Since $\varphi_n\rightarrow\varphi$, we conclude that $\varphi\in\overline{\mathrm{T}_{\mathrm{l}}(A)}$, which finishes the proof of Case 1.

{\bf Case 2.} \textit{$A_1$ has an isolated extreme $1$-dimensional trace.}

In this case, by our assumption, $A_2$ has no isolated extreme finite dimensional traces. Hence the proof of Case 1 gives the conclusion by swapping the roles of $A_1$ and $A_2$.

 {\bf Case 3.} \textit{$A_1$ has an isolated extreme $k$-dimensional trace, for some $k\geq 2$, but no isolated extreme $1$-dimensional traces.} 
 
 In this case, by our assumption, $A_2$ has no isolated extreme $1$-dimensional traces.  In particular, neither $A_1$ nor $A_2$ has an isolated extreme $1$-dimensional trace. 

 For $1\leq j\leq 2$, let $\{z_{i,j}\}_{i=0}^{K_j}\subset \mathcal Z(M_j)$ be possibly zero projections such that  $\sum_{i=0}^{K_j}z_{i,j}=1$, $M_jz_{0,j}$ has no $1$-dimensional direct summands, and $M_jz_{i,j}=\mathbb Cz_{i,j}$,  for every  $1\leq i\leq K_j$. 
 By using that $A_j$ has no isolated extreme $1$-dimensional traces and reasoning similarly to the proof of Case 1 (with $L=3$),
 we can find $\mathcal M_j\supset \mathcal M$ and $*$-homomorphisms $\pi_{n,j}:A_j\rightarrow\mathcal M_j$ such that $\|\pi_{n,j}(a_j)-\pi(a_j)\|_2\rightarrow 0$, as $n\to\infty$, for every $a_j\in A_j$,
and for every $1\leq i\leq K_j$ with $z_{i,j}\neq 0$ we have that 
$z_{i,j}\in \pi_{n,j}(A_j)'\cap\mathcal M_j$, $\mathrm{dim}(\pi_{n,j}(A_j)''z_{i,j})\geq 3$ and $\mathrm{e}(\pi_{n,j}(A_j)''z_{i,j})\leq\frac{\tau(z_{i,j})}{3}$, for every $n\in\mathbb N$. Moreover, if $z_{0,j}\neq 0$, then since $\pi_{n,j}(A_j)''z_{0,j}=M_jz_{0,j}$ has no $1$-dimensional direct summands, we get that $\mathrm{dim}(\pi_{n,j}(A_j)''z_{0,j})\geq 4$ and $\mathrm{e}(\pi_{n,j}(A_j)''z_{0,j})\leq\frac{\tau(z_{0,j})}{2}$. 
Using that $\sum_{i=0}^{K_j}z_{i,j}=1$, Lemma~\ref{e wrt central proj} implies that $\dim(\pi_{n,j}(A_j)'')\geq 3$ and $\mathrm{e}(\pi_{n,j}(A_j)'')\leq\frac{1}{2}$, for every $n\in\mathbb  N$.

Define $\widetilde{\mathcal M}=\mathcal M_1*_\mathcal M\mathcal M_2$ and $*$-homomorphisms $\pi_n:A\rightarrow\widetilde{\mathcal M}$ by letting $\pi_n(a_1)=\pi_{n,1}(a_1)$ and $\pi_n(a_2)=\pi_{n,2}(a_2)$, for every $a_1\in A_1$ and $a_2\in A_2$. Then $\|\pi_n(a)-\pi(a)\|_2\rightarrow 0$, for every $a\in A$. Since $\pi_n(A_1)=\pi_{n,1}(A_1),\pi_n(A_2)=\pi_{n,2}(A_2)$,
we get that $\dim(\pi_n(A_1)'')\geq 3,\dim(\pi_n(A_2)'')\geq 3$ and $\mathrm{e}(\pi_n(A_1)'')+\mathrm{e}(\pi_n(A_2)'')\leq \frac{1}{2}+\frac{1}{2}=1$. Thus, $\varphi_n\coloneq \tau\circ\pi_n\in\mathrm{T}_{\mathrm{l}}(A)$. Since $\varphi_n\rightarrow\varphi$, we conclude that $\varphi\in\overline{\mathrm{T}_{\mathrm{l}}(A)}$, which finishes the proof of  Case 3, and therefore of the lemma.
\end{proof}

\subsection{Isolated extreme finite dimensional traces}

We continue by establishing the equivalence of (2) and (3) in Theorem~\ref{Poulsen}, by showing that the isolated extreme finite dimensional traces of $A=A_1\ast A_2$ correspond exactly to the traces whose restrictions to $A_1$ and $A_2$ satisfy the negation of statement (2) in Theorem~\ref{Poulsen}:

\begin{lemma}\label{lem:ifd}
    Let $A=A_1\ast A_2$ be the 
    free product of two unital, separable $C^*$-algebras $A_1$ and $A_2$ such that $\mathrm{T}(A_i)$ is non-empty and does not consist of a single $1$-dimensional trace, for $i=1,2$. Assume $\vphi\in\partial_{\mathrm{e}}\mathrm{T}(A)$. Then the following are equivalent.
    \begin{enumerate}
        \item $\vphi$ is finite dimensional and isolated in $\partial_{\mathrm{e}}\mathrm{T}(A)$.
        \item Up to swapping indices, there exist an isolated extreme $1$-dimensional trace $\vphi_1\in \partial_{\mathrm{e}}\mathrm{T}(A_1)$ and an isolated extreme finite dimensional trace $\vphi_2\in\partial_{\mathrm{e}}\mathrm{T}(A_2)$ such that $\vphi_{|A_1} = \vphi_1$ and $\vphi_{|A_2}=\vphi_2$.
    \end{enumerate}
\end{lemma}

To prove Lemma \ref{lem:ifd} we will need the following immediate consequence of Proposition \ref{twoprojections}:

\begin{lemma}\label{diffuse_piece}
    Let $(M,\tau)$ be a tracial von Neumann algebra and $M_1\neq\mathbb C1,M_2\neq\mathbb C1$ be  von Neumann subalgebras. Then there is $t\in (0,1]$ such that for every $\varepsilon>0$, we can find a II$_1$ factor $M_\varepsilon$ with $M\subset M_\varepsilon$, $v_\varepsilon\in\mathcal U(M_\varepsilon)$, and a projection $z_\varepsilon\in\mathcal Z(M_1\vee v_\varepsilon M_2v_\varepsilon^*)$  such that $\|v_\varepsilon-1\|_2<\varepsilon$, $\tau(z_\varepsilon)\geq t$, and $(M_1\vee v_\varepsilon M_2v_\varepsilon^*)z_\varepsilon$ is diffuse.
\end{lemma}

\begin{proof}
For $i=1,2$, let $p_i\in M_i$ be a projection with $\tau(p_i)\in [\frac{1}{2},1)$ and put $A_i=\mathbb Cp_i\oplus\mathbb C(1-p_i)\subset M_i$. Put $t=2-\tau(p_1)-\tau(p_2)\in (0,1]$ and let $\varepsilon>0$.
By applying Proposition \ref{twoprojections} to $A_1,A_2\subset M$, we can find a II$_1$ factor $M_\varepsilon$ with $M\subset M_\varepsilon$ and $v_\varepsilon\in\mathcal U(M_\varepsilon)$ with $\|v_\varepsilon-1\|_2<\varepsilon$ such that, denoting $q_\varepsilon=1-p_1\wedge v_\varepsilon p_2v_\varepsilon^*$, we have  $\tau(q_\varepsilon)=t$, $q_\varepsilon\in \mathcal Z(A_1\vee v_\varepsilon A_2v_\varepsilon^*)$ and $(A_1\vee v_\varepsilon A_2v_\varepsilon^*)q_\varepsilon$ is diffuse.

Then $q_\varepsilon(M_1\vee v_\varepsilon M_2v_\varepsilon)q_\varepsilon$ is also diffuse. Thus, if $z_\varepsilon\in \mathcal Z(M_1\vee v_\varepsilon M_2v_\varepsilon^*)$ is the central support of $q_\varepsilon$, then $(M_1\vee v_\varepsilon M_2v_\varepsilon^*)z_\varepsilon$ is diffuse and $\tau(z_\varepsilon)\geq\tau(q_\varepsilon) = t$.
\end{proof}

\begin{proof}[Proof of Lemma \ref{lem:ifd}]
$(2)\Rightarrow (1)$. This is exactly Theorem~\ref{freeprod}.

$(1)\Rightarrow (2)$. Assume $\vphi\in \partial_{\mathrm{e}}\mathrm{T}(A)$ is finite dimensional and isolated. Let $\pi:A\to M$ be the tracial representation associated to $\vphi$. Denote $M_1=\pi(A_1)''$ and $M_2=\pi(A_2)''$. 

If 
$\vphi\in \mathrm{T}_{\mathrm{l}}(A)$, then by Theorem~\ref{perturbation} and Proposition~\ref{prop:property A traces can be approxiamted}, it follows that $\vphi\in\overline{\partial_{\mathrm{e}}\mathrm{T}(A)\cap\mathrm{T}_\infty(A)}$, contradicting that $\vphi$ is finite dimensional and isolated.
Hence  $\vphi\not\in \mathrm{T}_{\mathrm{l}}(A)$. 
We will prove the conclusion by treating two  cases:

\textbf{Case 1.} $M_1 = \C 1$ or $M_2 = \C 1$. Assume without loss of generality that $M_1=\C 1$. Then $M_2=M$, hence $\vphi_{|A_2}$ is an extreme finite dimensional trace. We will argue by contradiction that both $\vphi_1\coloneq\vphi_{|A_1}\in\partial_{\mathrm{e}}\mathrm{T}(A_1)$ and $\vphi_2\coloneq\vphi_{|A_2}\in\partial_{\mathrm{e}}\mathrm{T}(A_2)$ are isolated, i.e., condition (2) holds. 

First, assume that $\vphi_2$ is not isolated in $\partial_{\mathrm{e}}\mathrm{T}(A_2)$. Then we can find a sequence of pairwise distinct traces $\vphi_{n,2}\in \partial_{\mathrm{e}}\mathrm{T}(A_2)$, for $n\in\mathbb N$, such that $\vphi_{n,2}\to \vphi_2$. Denote by $\rho_{n,2}$ the corresponding tracial representations of $A_2$ and define representations $\rho_n$ of $A$ by $\rho_n(a_1) = \vphi(a_1)1$ for $a_1\in A_1$ and $\rho_n(a_2) = \rho_{n,2}(a_2)$ for $a_2\in A_2$. Then  $\rho_n(A)'' = \rho_{n,2}(A_2)''$ are tracial factors, say with traces $\tau_n$, and the traces $\vphi_n\coloneq \tau_n\circ \rho_n\in\partial_{\mathrm{e}}\text{T}(A)$ are distinct and converge to $\vphi$. This contradicts the fact that $\vphi$ is isolated in $\partial_{\mathrm{e}}\mathrm{T}(A)$.

Second, assume that $\vphi_1$ is not isolated in $\partial_{\mathrm{e}}\text{T}(A_1)$. If $M_2 = \C 1$, we can repeat the above argument to conclude, hence we can assume that $M_2\neq \C 1$. In particular, $\dim(M_2)\geq 2$ and $\mathrm{e}(M_2)<1$. Since $\vphi_1$ is not isolated, we can repeat the argument from Case 1 in the proof of Lemma~\ref{2} to conclude that there exist traces $\vphi_n\in\mathrm{T}_{\mathrm{l}}(A)$ such that $\vphi_n\to\vphi$. By Theorem~\ref{perturbation} and Proposition~\ref{prop:property A traces can be approxiamted}, it follows once more that $\vphi\in \overline{\partial_{\mathrm{e}}\mathrm{T}(A)\cap\mathrm{T}_\infty(A)}$, contradicting that $\vphi$ is finite dimensional and isolated.

\textbf{Case 2.} $M_1\neq\C 1$ and $M_2\neq \C 1$. By Lemma \ref{diffuse_piece}, there is $t\in (0,1]$ such that for every $n\in\mathbb N$, we can find a II$_1$ factor $P_n$ with $M\subset P_n$, $v_n\in\mathcal U(P_n)$, and a projection $z_n\in\mathcal Z(M_1\vee v_nM_2v_n^*)$ such that  $\|v_n-1\|_2<\frac{1}{n}$, $(M_1\vee v_nM_2v_n^*)z_n$ is diffuse and $\tau(z_n)\geq t$. Define representations $\rho_n:A\to P_n$ by $\rho_n(a_1) = \pi(a_1)$ for $a_1\in A_1$ and $\rho_n(a_2) = v_n\pi(a_2)v_n^*$ for $a_2\in A_2$. 

Then defining $\vphi_n\coloneq \tau\circ\rho_n\in\text{T}(A)$ we have $\vphi_n\to\vphi$. Since $\vphi$ is isolated,  by Corollary~\ref{strongly_isolated} we get that $\mu_n(\{\vphi\})\to 1$, where $\mu_n$ is the unique probability measure on $\partial_{\mathrm{e}}\mathrm{T}(A)$ whose barycenter is $\vphi_n$. For every $n\in\mathbb N$, we can thus find a central projection $w_n\in\rho_n(A)''=M_1\vee v_nM_2v_n^*$  such that $\tau(w_n)=\mu_n(\{\vphi\})\rightarrow 1$ and $(M_1\vee v_nM_2v_n^*)w_n$ is isomorphic to $\pi(A)''=M$, and hence is finite dimensional.
Since $\tau(z_n)\geq t>0$, for every $n\in\mathbb N$, we can find $n\in\mathbb N$ such that $z_nw_n\neq 0$. Since $(M_1\vee v_nM_2v_n^*)z_nw_n$ would then be both diffuse and finite dimensional, this is a contradiction.

Since  Cases 1 and 2 either imply condition (2) or yield a contradiction, this finishes the proof of the lemma. 
\end{proof}

\subsection{Proof of Theorem~\ref{Poulsen} and its Corollaries} We now have all the ingredients to complete the proof of Theorem~\ref{Poulsen}:

\begin{proof}[Proof of Theorem~\ref{Poulsen}] Let $A=A_1 \ast A_2$ be the 
free product of two unital, separable $C^*$-algebras $A_1$ and $A_2$, satisfying that $\mathrm{T}(A_i)$ is non-empty and does not consist of a single $1$-dimensional trace, for every $1\leq i\leq 2$.

First of all, we note that Lemma~\ref{lem:ifd} implies the equivalence of statements (2) and (3). We proceed by proving the equivalence of statements (1) and (2), and the moreover statement. 

Assume that condition (2) in Theorem~\ref{Poulsen} holds. Note that Theorem~\ref{perturbation} and Proposition~\ref{prop:property A traces can be approxiamted} imply that $\mathrm{T}_{\mathrm{l}}(A)\subset \mathrm{T}_{\mathrm{apf}}(A)\subset \overline{\partial_{\mathrm{e}}\mathrm{T}(A)\cap\mathrm{T}_{\infty}(A)}$. From this and Lemma~\ref{2}, it thus follows that
\[
     \mathrm{T}(A)\subset \overline{\mathrm{T}_{\text l}(A)} \subset   \overline{\mathrm{T}_{\mathrm{apf}}(A)} \subset  \overline{\partial_{\mathrm{e}}\mathrm{T}(A)\cap\mathrm{T}_\infty(A)}\subset \overline{\partial_{\mathrm{e}}\mathrm{T}(A)}.
\]
Hence, the moreover statement of the theorem holds and $\mathrm{T}(A)\subset\overline{\partial_{\mathrm{e}}\mathrm{T}(A)}$, meaning that 
$\mathrm{T}(A)$ is a Poulsen simplex.

Conversely, assume that condition (2) in Theorem~\ref{Poulsen} fails. Without loss of generality, assume that $A_1$ admits an isolated extreme $1$-dimensional trace $\vphi$ and $A_2$ admits an isolated extreme $k$-dimensional trace for some $k\in\N$. By assumption, $\vphi$ is not the only trace on $A_1$. In particular, $|\mathrm{T}(A_1)|\geq 2$ and therefore also $|\mathrm{T}(A)|\geq 2$. Hence Corollary~\ref{cor:obstruction} tells us that $\mathrm{T}(A)$ is not a Poulsen simplex. This completes the proof of Theorem~\ref{Poulsen}.
\end{proof}

\begin{proof}[Proof of Corollary \ref{fin_dim_algebras}]
Assume that $A_i$ has no $1$-dimensional direct summand, for every $1\leq i\leq 2$. Since $A_i$ is finite dimensional, we get that $\mathrm{T}(A_i)$ is non-empty and contains no $1$-dimensional trace, for every $1\leq i\leq 2$. Theorem \ref{Poulsen} thus implies that $\mathrm{T}(A)$ is Poulsen.

Conversely, assume that $\text{T}(A)$ is Poulsen. Since $A_i$
is finite dimensional, it admits an isolated extreme finite dimensional trace, for every $1\leq i\leq 2$. By Corollary \ref{cor:obstruction} it follows that $A_i$ does not admit a $1$-dimensional trace and thus has no $1$-dimensional direct summand, for every $1\leq i\leq 2$.
\end{proof}

Finally, we turn to the proof of Corollary \ref{group C*-algebras}. In order to do so, we first establish the following result which appears to be of independent interest.
For a countable discrete group $G$, we denote by $\mathrm{Tr}(G)$ and $\mathrm{Ch}(G)$ the sets of traces and characters of $G$, respectively, and we identify $\mathrm{Tr}(G)=\mathrm{T}(C^*G)$ and $\mathrm{Ch}(G)=\partial_{\mathrm{e}}\mathrm{T}(C^*G)$, in the usual way.
Recall that $G$ is said to have \textit{property (ET)} if the trivial character $1_G$ is isolated in $\mathrm{Ch}(G)$.

\begin{lemma}\label{isolated fd characters}
    Let $G$ be a countable discrete group. Assume that there exists a finite dimensional  character $\varphi\in\mathrm{Ch}(G)$ which is isolated in $\mathrm{Ch}(G)$. Then $G$ has property \emph{(ET)}.
\end{lemma}
\begin{proof}
   To prove that $G$ has property (ET), let $(\psi_n)_{n\in\mathbb N}\subset\mathrm{Ch}(G)$ be a sequence such that $\psi_n\rightarrow 1_G$. Define $\varphi_n\coloneq \varphi\psi_n\in\mathrm{Tr}(G)$. 
   Let $\pi:G\rightarrow\mathcal U(k)$ be a homomorphism, for some $k\in\mathbb N$, such that $\varphi=\mathrm{tr}_k\circ\pi$ and $\pi(G)''=\mathbb M_k(\mathbb C)$, where $\text{tr}_k:\mathbb M_k(\mathbb C)\rightarrow\mathbb C$ denotes the normalized trace.
   For $n\in\mathbb N$, let $P_n$ be a tracial factor and $\rho_n:G\rightarrow\mathcal U(P_n)$ 
   be a homomorphism such that $\psi_n=\tau\circ\rho_n$ and $\rho_n(G)''=P_n$.
    Define $\pi_n:G\rightarrow\mathcal U(\mathbb M_k(\mathbb C)\overline{\otimes}P_n)$ by letting $\pi_n(g)=\pi(g)\otimes\rho_n(g)$, for every $g\in G$. Then $\tau\circ\pi_n=\varphi\psi_n=\varphi_n.$
   
   Let $\mu_n$ be the unique Borel probability measure on $\mathrm{Ch}(G)$ whose barycenter is $\varphi_n$. Since $\varphi_n\rightarrow\varphi$
and $\varphi$ is isolated in $\mathrm{Ch}(G)$, Corollary \ref{strongly_isolated} implies that $\lambda_n\coloneq \mu_n(\{\varphi\})\rightarrow 1$.
Since $\tau\circ\pi_n=\varphi_n$,  there is a projection $z_n\in\mathcal Z(\pi_n(G)'')$ such that $\tau(z_n)=\lambda_n$ and $\tau(\pi_n(g)z_n)=\lambda_n\varphi(g)$, for every $g\in G$.

If $k=1$, then $\pi:G\rightarrow\mathcal U(1)=\mathbb T$ is a homomorphism. Thus, $\varphi=\pi$ and $\pi_n(g)=\varphi(g)\rho_n(g)$. This implies that for every $n\in\mathbb N$ and $g\in G$ we have that $\lambda_n\varphi(g)=\tau(\pi_n(g)z_n)=\varphi(g)\tau(\rho_n(g)z_n)$. Hence, for every $n\in\mathbb N$, we have $\tau(\rho_n(g)z_n)=\lambda_n=\tau(z_n)$ and thus $\rho_n(g)z_n=z_n$, for every $g\in G$. This implies that $z_n$ lies in the center of $\rho_n(G)''=P_n$. Since $P_n$ is a factor and $\tau(z_n)\rightarrow 1$, we get that $z_n=1$ for large $n\in\mathbb N$. Hence, $\psi_n=\tau\circ\rho_n=1_G$, for large $n\in\mathbb N$. This shows that $1_G$ is isolated in $\text{Ch}(G)$ and so $G$ has property (ET) if $k=1$. For $k\geq 2$, we will need the following claim:

\begin{claim}\label{lower bound}
If $p_n\in\pi_n(G)'\cap (\mathbb M_k(\mathbb C)\overline{\otimes}P_n)$ is a non-zero projection, for $n\in\mathbb N$, then $\tau(p_n)\geq \frac{1}{k^2}$.
\end{claim}

\begin{proof}[Proof of Claim]
 Let $g\in G$ and note that $\pi(g)\otimes 1$ commutes with $1\overline{\otimes}P_n$ and that $1\otimes\rho_n(g)$ and
$\pi_n(g)=\pi(g)\otimes\rho_n(g)$ normalize 
$1\overline{\otimes}P_n$. Using these facts, we get that $$\mathrm{Ad}(1\otimes\rho_n(g))(\mathrm{E}_{1\overline{\otimes}P_n}(p_n))=\mathrm{Ad}(\pi_n(g))(\mathrm{E}_{1\overline{\otimes}P_n}(p_n))=\mathrm{E}_{1\overline{\otimes}P_n}(\mathrm{Ad}(\pi_n(g))(p_n))=\mathrm{E}_{1\overline{\otimes}P_n}(p_n).$$
In other words, $\mathrm{E}_{1\overline{\otimes}P_n}(p_n)$ commutes with $1\otimes \rho_n(g)$, for every $g\in G$. Since $\rho_n(G)''=P_n$ and $P_n$ is a factor, we conclude that $\mathrm{E}_{1\overline{\otimes}P_n}(p_n)=\tau(p_n)1$.
On the other hand, for every tracial von Neumann algebra $(P,\tau)$ we have that
\begin{equation}\label{PP86}
    \text{$\mathrm{E}_{1\overline{\otimes}P}(x)\geq\frac{1}{k^2}x$,\;\; for every $x\in \mathbb M_k(\mathbb C)\overline{\otimes}P$ with $x\geq 0$.}
\end{equation}
This inequality follows from \cite[Proposition 2.1]{PP86} when $P$ is a II$_1$ factor, since in this case $1\overline{\otimes}P\subset\mathbb M_k(\mathbb C)\overline{\otimes}P$ is an inclusion of II$_1$ factors of index $k^2$. For general $P$, let $Q$ be a II$_1$ factor which contains $P$ and consider the natural inclusion $\mathbb M_k(\mathbb C)\overline{\otimes}P\subset\mathbb M_k(\mathbb C)\overline{\otimes}Q$.  If $x\in \mathbb M_k(\mathbb C)\overline{\otimes}P$ and $x\geq 0$, then $\mathrm{E}_{1\overline{\otimes}P}(x)=\mathrm{E}_{1\overline{\otimes}Q}(x)\geq\frac{1}{k^2}x$, which proves \eqref{PP86}.

By applying \eqref{PP86}, we get that $\tau(p_n)1=\mathrm{E}_{1\overline{\otimes}P_n}(p_n)\geq\frac{1}{k^2}p_n$ and hence $p_n\leq k^2\tau(p_n)1$. Since $p_n$ is a non-zero projection, this implies that $k^2\tau(p_n)\geq 1$ and thus $\tau(p_n)\geq\frac{1}{k^2}$, as claimed.
\end{proof}

Next, since $\tau(z_n)\rightarrow 1$, we can find $N\in\mathbb N$ such that $\tau(z_n)>1-\frac{1}{k^2}$, for every $n\geq N$. Since $z_n\in\pi_n(G)'\cap (\mathbb M_k(\mathbb C)\overline{\otimes}P_n)$, Claim \ref{lower bound} implies that $z_n=1$, for every $n\geq N$. Thus, $\tau\circ\pi_n=\varphi$ and  
\begin{equation}\label{psi=1}\text{$\varphi\psi_n=\varphi$, for every $n\geq N$.}
\end{equation}
For $n\in\mathbb N$, put $H_n=\{g\in G\mid\psi_n(g)=1\}=\{g\in G\mid\rho_n(g)=1\}$.
Put $S=\{g\in G\mid \varphi(g)\neq 0\}$ and let $H=\langle S\rangle<G$ be the subgroup generated by $S$. 
By \ref{psi=1}, we get that $S\subset H_n$ and since $H_n$ is a group, we deduce that
\begin{equation}\label{containment}
    \text{$H\subset H_n$, for every $n\geq N$.}
\end{equation}

We claim that $[G:H]<\infty$. Otherwise, we can find a sequence $(g_i)_{i=1}^\infty\subset G$ such that $g_j^{-1}g_i\not\in  H$, for every $i\neq j$. But then for every $i\neq j$, we have $g_j^{-1}g_i\not\in S$ and thus $0=\varphi(g_j^{-1}g_i)=\mathrm{tr}_k(\pi(g_j^{-1}g_i))$. This would imply that the unitaries $(\pi(g_i))_{i=1}^\infty\subset\mathbb M_k(\mathbb C)$ are pairwise orthogonal with respect to the scalar product $\langle \xi,\eta\rangle=\text{tr}_k(\eta^*\xi)$, contradicting that $\mathbb M_k(\mathbb C)$ is finite dimensional.

Note that $H<G$ is a normal subgroup and let $\zeta:G\rightarrow G/H$ be the quotient homomorphism.
If $n\geq N$, then \eqref{containment} gives that $H\subset H_n$
and thus ${\psi_n}_{|H}=1_H$. This implies that there exists $\chi_n\in\mathrm{Ch}(G/H)$ such that $\psi_n=\chi_n\circ\zeta$.
Since $\psi_n\rightarrow 1_G$, we get that $\chi_n\rightarrow 1_{G/H}$. Since $G/H$ is a finite group, $\mathrm{Ch}(G/H)$ is a finite set and thus $1_{G/H}$ is isolated in $\mathrm{Ch}(G/H)$. 
Hence, we can find $N'\geq N$ such that $\chi_n=1_{G/H}$ and hence $\psi_n=1_G$, for every $n\geq N'$. This proves that $1_G$ is isolated in $\mathrm{Ch}(G)$, which means that $G$ has property (ET).
\end{proof}

\begin{proof}[Proof of Corollary \ref{group C*-algebras}]
    The implications (1) $\Rightarrow$ (3) and (3) $\Rightarrow$ (2) follow from \cite[Proposition 7.9]{OSV23} and \cite[Proposition 7.4]{OSV23}, respectively. 
    
    To justify the implication (2) $\Rightarrow$ (1), assume that $G_{i_0}$ does not have property (ET), for some $i_0\in\{1,2\}$. 
    Since $G_{i_0}$ does not have property (ET), Lemma \ref{isolated fd characters} implies that $\mathrm{Ch}(G_{i_0})$ contains no isolated finite dimensional characters. Equivalently, $\partial_{\mathrm{e}}\mathrm{T}(C^*G_{i_0})$ contains no isolated extreme finite dimensional traces.
    Additionally, since $G_i$ is a non-trivial group, for $i\in\{1,2\}$, $1_{G_i}$ and $\delta_e$ are distinct traces on $G_i$, and therefore
     $|\mathrm{T}(C^*G_i)|=|\mathrm{Tr}(G_i)|\geq 2$, for $i\in\{1,2\}$. Thus, by applying Theorem \ref{Poulsen} to $C^*G=C^*G_1*C^*G_2$, we conclude that $\mathrm{Tr}(G)=\mathrm{T}(C^*G)$ is a Poulsen simplex.
\end{proof}

\begin{remark}\label{dihedral}

    Denoting by $\widehat G$ the space of irreducible representations of a group $G$ equipped with the Fell topology, we recall the well-known fact that if $G$ has property (T), i.e., the trivial representation is isolated in $\widehat G$, then automatically every finite dimensional irreducible representation of $G$ is isolated in $\widehat G$ (see, for instance, \cite[Theorem~1.2.5]{BdlHV08}). Nevertheless, we note that the analogous statement for property (ET) fails, i.e., property (ET) does not imply that every finite dimensional character is isolated in $\mathrm{Ch}(G)$. 
    In other words, the converse to Lemma \ref{isolated fd characters} fails.

    Indeed, consider the group\footnote{$G$ is isomorphic to the infinite dihedral group $D_\infty =\Z/2\Z \ltimes \Z$} $G=\Z/2\Z *\Z/2\Z$. Then $G$ has property (ET) by \cite[Proposition~7.4]{OSV23}, but it admits a continuum of $2$-dimensional characters that are not isolated in $\mathrm{Ch}(G)$: since $C^*G\cong\mathbb C^2*\mathbb C^2$ is the universal $C^*$-algebra generated by two projections, say $p$ and $q$, we can define the ($2$-dimensional) representations $\pi_t:C^*G\to\bM_2(\C)$, for $t\in (0,1)$, determined by 
    \[
    \pi_t(p) = \begin{pmatrix} 1 & 0 \\ 0 & 0\end{pmatrix} \quad\text{and}\quad \pi_t(q) = \begin{pmatrix} 1-t &\sqrt{t(1-t)} \\ \sqrt{t(1-t)}& t\end{pmatrix}.
    \]
    Then  $\|\pi_s(x)-\pi_{t}(x)\|_2\rightarrow 0$, as $s\rightarrow t$, for every $x\in C^*G$. Thus, denoting $\vphi_t\coloneq\mathrm{tr}_2\circ\pi_t\in\text{Ch}(G)$, we have that $\vphi_s\rightarrow\vphi_t$, as $s\rightarrow t$.
    Hence, $\vphi_t$ is  not isolated in $\text{Ch}(G)$, for every $t\in (0,1)$.
\end{remark}

\begin{remark}\label{trace space of C^2*C^2}
In the notation of Remark \ref{dihedral}, for $\varepsilon_1,\varepsilon_2\in\{0,1\}$, let $\vphi_{\varepsilon_1,\varepsilon_2}:C^*G\rightarrow\mathbb C$ be the representation determined by $\vphi_{\varepsilon_1,\varepsilon_2}(p)=\varepsilon_1$, 
    $\vphi_{\varepsilon_1,\varepsilon_2}(q)=\varepsilon_2$, and view $\vphi_{\varepsilon_1,\varepsilon_2}\in\text{Ch}(G)$. 
Then we have $\text{Ch}(G)=
\{\vphi_t\mid t\in (0,1)\}\cup\{\vphi_{\varepsilon_1,\varepsilon_2}\mid\varepsilon_1,\varepsilon_2\in\{0,1\}\}.$ Moreover,  $\{\vphi_{\varepsilon_1,\varepsilon_2}\mid\varepsilon_1,\varepsilon_2\in\{0,1\}\}$ are isolated points in $\text{Ch}(G)$ and we have
$$\lim_{t\rightarrow 0}\vphi_t=\psi_0\coloneq\frac{1}{2}(\vphi_{0,0}+\vphi_{1,1})\;\;\;\;\text{and}\;\;\;\;
\lim_{t\rightarrow 1}\vphi_t=\psi_1\coloneq\frac{1}{2}(\vphi_{1,0}+\vphi_{0,1}).$$
  This implies that $\text{Tr}(G)$ is neither a Poulsen nor a Bauer simplex.  Moreover, since we have that $\overline{\text{Ch}(G)}=\text{Ch}(G)\cup\{\psi_0,\psi_1\}$, it follows that $\text{Tr}(G)$ does not contain the Poulsen simplex as a closed face. Indeed, if $X\subset\text{Ch}(G)$ is a non-empty set satisfying $\text{conv}(X)\subset\overline{X}$, then $X$ must be a singleton.
\end{remark}

\section{The Poulsen simplex as a closed face}\label{sec:face} 

Let $A=A_1\ast A_2$ be the 
free product of two unital, separable $C^*$-algebras $A_1$ and $A_2$. In this section, we investigate what happens when the conditions from Theorem~\ref{Poulsen} fail, i.e., when $\mathrm{T}(A)$ is not a Poulsen simplex. We start with proving Theorem~\ref{Poulsenface}, which identifies the Poulsen simplex as a closed face of $\mathrm{T}(A)$, unless it is affinely homeomorphic to $\mathrm{T}(\C^2\ast\C^2)$. We finish by answering in the negative \cite[Question~1.11]{OSV23}.

\subsection{The Poulsen simplex as a closed face of the trace simplex of a free product} In this subsection, we prove Theorem \ref{Poulsenface}.
The proof uses the following observation made in \cite{MR19}: 
a unital surjective $*$-homomorphism $\gamma:A\to B$ between two unital, separable $C^*$-algebras, induces an affine continuous injection $\mathrm{T}(B)\ni\varphi\mapsto\varphi\circ\gamma\in\mathrm{T}(A)$ whose image is a closed face. Thus, if $\text{T}(B)$ is the Poulsen simplex or, more generally, admits the Poulsen simplex as a closed face, then $\text{T}(A)$ also contains the Poulsen simplex as a closed face.
In particular, we get the following:

\begin{corollary}\label{cor:matrices}   Let $A=A_1\ast A_2$ be the 
free product of two unital, separable $C^*$-algebras $A_1$ and $A_2$. Assume that $A_1$ admits an irreducible $n$-dimensional representation and $A_2$ admits an irreducible $k$-dimensional representation, for some  $n,k\geq 2$. Then $\mathrm{T}(A)$ has a closed face which is affinely homeomorphic to the Poulsen simplex. \end{corollary}

\begin{proof} The assumptions imply the existence of a unital surjective $*$-homomorphism $A\to \bM_n(\C)\ast \bM_k(\C)$. Since $\mathrm{T}(\bM_n(\C)\ast\bM_k(\C))$ is the Poulsen simplex by Corollary~\ref{fin_dim_algebras}, the desired result follows from the  observation preceding the corollary.\end{proof}

In order to show that the trace simplexes of  
$\C^2\ast \C^n$, $n\geq 3$, and $\C^2\ast\bM_n(\C)$, $n\geq 2$, also admit the Poulsen simplex as a closed face,  we first prove that these $C^*$-algebras admit many quotients.

\begin{lemma}\label{quotients} Let $A$ be a unital $C^*$-algebra and $n\geq 2$ be an integer.
\begin{enumerate}
    \item \cite{Va98} 
    If $A$ is generated by $n-1$ self-adjoint elements, then there is a unital surjective $*$-homomorphism $\mathbb C^2*\mathbb C^n\rightarrow\mathbb M_n(\mathbb C)\otimes A$.
    \item If $A$ is generated by $n-1$ unitaries, then there is a unital surjective $*$-homomorphism $\mathbb C^2*\mathbb M_n(\mathbb C)\rightarrow\mathbb M_n(\mathbb C)\otimes A$. 
    \item If $A$ is generated by $2$ self-adjoint elements, then there is a unital surjective $*$-homomorphism $\mathbb C^2*\mathbb M_2(\mathbb C)\rightarrow \mathbb M_2(\mathbb C)\otimes A$.
\end{enumerate}
\end{lemma}

\begin{proof}
(1) This is \cite[Corollary~4.5]{Va98}. Indeed, \cite[Corollary~4.5]{Va98} shows that $\mathbb M_n(\mathbb C)\otimes A$ is generated by projections $p,q_1,\ldots,q_n$ such that $\sum_{i=1}^nq_i=1$. 
Define unital $*$-homomorphisms $\alpha:\mathbb C^2\rightarrow \mathbb M_n(\mathbb C)\otimes A$ and $\beta:\mathbb C^n\rightarrow\mathbb M_n(\mathbb C)\otimes A$ given by $\alpha(s_1,s_2)=s_1p+s_2(1-p)$ and $\beta(t_1,\ldots,t_n)=\sum_{i=1}^nt_iq_i$.  Since the images of $\alpha$ and $\beta$ generate $\mathbb M_n(\mathbb C)\otimes A$, it follows that the unital $*$-homomorphism $\gamma:\mathbb C^2*\mathbb C^n\rightarrow\mathbb M_n(\mathbb C)\otimes A$ determined by $\gamma_{|\mathbb C^2}=\alpha$ and $\gamma_{|\mathbb C^n}=\beta$ is surjective.

(2)  Suppose that $A$ is generated by unitaries $u_2, \ldots, u_n$. The first part of the proof is the same as the first part of the proof of \cite[Proposition~3.7(ii)]{MR19}. For $1\leq i,j\leq n$, denote by $e_{ij}\in\mathbb M_n(\mathbb C)$ the  elementary matrix with a $1$ in the $(i,j)$-position and $0$ elsewhere.
Set $f_{11} \coloneq e_{11}\ot 1_A$ and $f_{1j} \coloneq e_{1j}\ot u_j$, for $2\leq j\leq n$. Then $f_{1j}f_{1j}^* = e_{11}\ot 1_\cA$ and $f_{1j}^*f_{1j} = e_{jj}\ot 1_\cA$, for every $1\leq j\leq n$. Setting $f_{ij}\coloneq f_{1i}^*f_{1j}$ for $1\leq i,j\leq n$, we get that $\{f_{ij}\mid 1\leq i,j\leq n\}$ is a set of matrix units in $\bM_n(\C)\ot A$. Hence, there is a unital $*$-homomorphism $\beta:\bM_n(\C)\to \bM_n(\C)\ot A$ determined by $\beta(e_{ij})=f_{ij}$ for $1\leq i,j\leq n$. We note that, by construction, $e_{ii}\ot 1_A\in\beta(\mathbb M_n(\mathbb C))$, for every $1\leq i\leq n$.  

Next, let $p\in\mathbb M_n(\mathbb C)$ be the projection matrix whose every entry is equal to $\frac{1}{n}$, i.e., $(p)_{i,j}=\frac{1}{n}$, for every $1\leq i,j\leq n$.
Then the subalgebra $V$ of $\bM_n(\C)$ generated by $\{p,e_{ii}\mid 1\leq i\leq n\}$ equals $\bM_n(\C)$. Indeed, given any $1\leq i,j\leq n$, we see that $e_{ij} = n^2e_{ii}pe_{jj} \in V$, which implies that $V=\bM_n(\C)$. Define a unital $*$-homomorphism $\alpha: \C^2\to\bM_n(\mathbb C) \ot A$ by letting $\alpha(s_1,s_2)=s_1 (p\otimes 1_A)+s_2((1-p)\otimes 1_A)$.

Finally, denote by $\gamma:\C^2\ast\bM_n(\C)\to \bM_n(\C)\ot A$ the unital $*$-homomorphism determined by $\gamma_{|\C^2}=\alpha$ and $\gamma_{|\mathbb M_n(\mathbb C)}=\beta$. Then by construction, $p\ot 1_A$ and $e_{ii}\ot 1_A$ are in the image of $\gamma$, for every $1\leq i\leq n$, and thus by the above, $\bM_n(\C)\ot 1_A\subset \gamma(\C^2\ast \bM_n(\C))$. Since also $f_{1j} = e_{1j}\ot u_j\in \gamma(\C^2\ast \bM_n(\C))$, we deduce that $1\ot u_j$ is contained in $\gamma(\C^2\ast \bM_n(\C))$, for every $2\leq j\leq n$. We conclude that $\gamma$ is surjective, finishing the proof of (2).

(3) Suppose that $A$ is generated by $2$ self-adjoint elements. Then $A$ is generated by a single element $z\in A$ with $\|z\|\leq 1$. Define, using continuous functional calculus,  $p\in \mathbb M_2(\mathbb C)\otimes A=\mathbb M_2(A)$ by letting 
$$p=\begin{pmatrix}\frac{1+\sqrt{1-z^*z}}{2}&z^*\\ z&\frac{1-\sqrt{1-zz^*}}{2} \end{pmatrix}.$$
Since $z(z^*z)^n=(zz^*)^nz$, for every $n\geq 0$, it follows that $z\sqrt{1-z^*z}=\sqrt{1-zz^*}z$, which implies that $p$ is a projection. Define unital $*$-homomorphisms $\alpha:\mathbb C^2\rightarrow\mathbb M_2(\mathbb C)\otimes A$ and $\beta:\mathbb M_2(\mathbb C)\rightarrow\mathbb M_2(\mathbb C)\otimes A$ by letting $\alpha(s_1,s_2)=s_1p+s_2(1-p)$ and $\beta(x)=x\otimes 1_A$. Denote by $\gamma:\C^2\ast\bM_2(\C)\to \bM_2(\C)\ot A$ the unital $*$-homomorphism determined by $\gamma_{|\C^2}=\alpha$ and $\gamma_{|\mathbb M_2(\mathbb C)}=\beta$.
Since $z$ generates $A$, it follows that $p$ and $\mathbb M_2(\mathbb C)\otimes 1_A$  generate $\mathbb M_2(\mathbb C)\otimes A$. This implies that $\gamma$ is surjective.
\end{proof}

\begin{lemma}\label{lem:C2}The trace simplexes $\mathrm{T}(\C^2\ast\C^3)$ and $\mathrm{T}(\C^2\ast\bM_n(\C))$, for $n\geq 2$, have a closed face which is affinely homeomorphic to the Poulsen simplex.
\end{lemma}

\begin{proof} Let $A=C([0,1])*C([0,1])$. Since $A$ is generated by $2$ self-adjoint elements, parts (1) and (3) of Lemma \ref{quotients} provide unital surjective $*$-homomorphisms $\mathbb C^2*\mathbb C^3\rightarrow\mathbb M_3(\mathbb C)\otimes A$ and $\mathbb C^2*\mathbb M_2(\mathbb C)\rightarrow\mathbb M_2(\mathbb C)\otimes A$.
Since $\text{T}(A)$ is the Poulsen simplex by \cite[Theorem 1.5]{OSV23} (or Theorem~\ref{Poulsen}) and $\text{T}(\mathbb M_k(\mathbb C)\otimes A)\cong\text{T}(A)$, for $k\geq 2$, the conclusion for $\mathrm{T}(\C^2\ast\C^3)$ and $\mathrm{T}(\C^2\ast\bM_2(\C))$ follows. 

If $n\geq 3$, then since $C^*\mathbb F_{n-1}$ is generated by $n-1$ unitaries, part (2) of Lemma \ref{quotients} provides a unital surjective $*$-homomorphism $\mathbb C^2*\mathbb M_n(\mathbb C)\rightarrow\mathbb M_n(\mathbb C)\otimes C^*\mathbb F_{n-1}$. Since $\text{T}(C^*\mathbb F_{n-1})$ is the Poulsen simplex by \cite[Theorem 1.1]{OSV23} (or Theorem \ref{Poulsen}) and $\text{T}(\mathbb M_n(\mathbb C)\otimes C^*\mathbb F_{n-1})\cong\text{T}(C^*\mathbb F_{n-1})$, the conclusion also follows for $\text{T}(\mathbb C^2*\mathbb M_n(\mathbb C))$, for $n\geq 3$.
\end{proof}

\begin{lemma}\label{lem:infinite}
    Let $A=A_1\ast A_2$ be the full free product of two unital, separable $C^*$-algebras $A_1$ and $A_2$ such that $\mathrm{T}(A_i)$ is non-empty and does not consist of a single $1$-dimensional trace, for $i=1,2$. Assume that $\mathrm{T}(A_1)$ contains an infinite dimensional extreme trace. Then $\mathrm{T}(A)$ has a closed face which is affinely homeomorphic to the Poulsen simplex.
\end{lemma}
\begin{proof}
Fix $\vphi_1\in\partial_{\mathrm{e}}\mathrm{T}(A_1)\cap\mathrm{T}_\infty(A_1)$, and consider the set
\[
X \coloneq \{\psi\in \mathrm{T}(A)\mid \psi_{|A_1} = \vphi_1\}.
\]
We will prove that $X$ is a closed face of $\mathrm{T}(A)$ affinely homeomorphic to the Poulsen simplex, which will finish the proof.

Firstly, if $\phi,\chi\in X$, then it is immediate that for every $\lambda\in [0,1]$ we have
\[
\big(\lambda \phi + (1-\lambda)\chi\big)|_{A_1} = \lambda \phi_{|A_1} + (1-\lambda)\chi_{|A_1} = \vphi_1,
\]
hence $X$ is convex. Similarly, if $\psi_n\in X$ is a sequence converging to some $\psi\in\mathrm{T}(A)$, then $\psi_{|A_1} = \lim_n \psi_{n|A_1} = \vphi_1$, hence $X$ is closed. 

Next, assume that $\psi\in X$ satisfies $\psi = \lambda \phi + (1-\lambda) \chi$, for some $\phi,\chi\in\text{T}(A)$ and $\lambda\in [0,1]$. Then
\[
\vphi_1 = \psi_{|A_1} = \lambda \phi_{|A_1} + (1-\lambda) \chi_{|A_1}.
\]
Since $\phi_{|A_1},\chi_{|A_1}\in\mathrm{T}(A_1)$ and $\vphi_1\in\partial_{\mathrm{e}}\mathrm{T}(A_1)$, we conclude that $\phi_{|A_1} = \chi_{|A_1} = \vphi_1$, i.e., $\phi,\chi\in X$. We conclude that $X$ is a closed face of $\mathrm{T}(A)$. 

Note that $X$ contains at least two points. Indeed, if $\varphi_2\in\text{T}(A_2)$ is any trace which is not $1$-dimensional, then the traces $\varphi_1\otimes\varphi_2$ and $\varphi_1*\varphi_2$ are distinct and belong to $X$.

Finally, we show that the extreme points of $X$ are dense in $X$. Since we already know that $X$ is a closed face, hence a simplex, 
and that $X$ has at least two points,
this will imply that $X$ is affinely homeomorphic to the Poulsen simplex. 

Fix $\psi\in X$, and let $\pi:A\to M$ be the associated tracial representation. Denote by $M_i = \pi(A_i)''$. Note that by construction, $\tau\circ \pi_{|A_1} = \psi_{|A_1} = \vphi_1$, hence $\pi_{|A_1}: A_1\to M_1$ is a tracial representation associated to $\vphi_1$. Since $\vphi_1\in\partial_{\mathrm{e}}\mathrm{T}(A_1)\cap\mathrm{T}_\infty(A_1)$, it follows that $M_1$ is a II$_1$ factor. If $M = M_1$, then $\psi$ is an extreme point and we are done. Otherwise, we necessarily have $M_2\neq \C 1$, and thus applying Corollary~\ref{one diffuse algebra}, we get a von Neumann algebra $\widetilde{M}$ containing $M$, and unitaries $u_n\in\widetilde{M}$, with $\norm{u_n - 1}_2 \to 0$ as $n\to \infty$, such that $M_1\vee u_n M_2 u_n^*$ is a II$_1$ factor for every $n\in\mathbb N$. Define $\pi_n:A\to \widetilde{M}$ by $\pi_n(a_1) = \pi(a_1)$, for $a_1\in A_1$, and $\pi_n(a_2)=u_n\pi(a_2)u_n^*$, for $a_2\in A_2$. Then $\psi_n\coloneq \tau\circ \pi_n$ converges to $\psi$, as $n\rightarrow\infty$. Furthermore, we have by construction that $\pi_n(A)''$ is a II$_1$ factor, and $\psi_{n|A_1} = \tau\circ \pi_{n|A_1} = \tau\circ \pi_{|A_1} = \psi_{|A_1} = \vphi_1$, hence $\psi_n\in \partial_{\mathrm{e}} X$, for every $n\in\N$. We conclude that the extreme points of $X$ are dense in $X$. This finishes the proof of the lemma.
\end{proof}

\begin{proof}[Proof of Theorem \ref{Poulsenface}]
Recall that $A=A_1*A_2$, where $A_i$ is a unital, separable $C^*$-algebra such that $\text{T}(A_i)$ is non-empty and does not consist of a single $1$-dimensional trace, for every $i\in\{1,2\}$. Assume that condition (1) of Theorem \ref{Poulsenface} fails, i.e., ($\star$) $\text{T}(A)$ does not admit a closed face which is affinely homeomorphic to the Poulsen simplex.
We will prove that condition (2) holds.

First,  combining ($\star$) and Lemma~\ref{lem:infinite} implies that $A_i$ does not have an infinite dimensional extreme trace, for every $i\in\{1,2\}$. In other words, all extreme traces of $A_1$ and $A_2$ are finite dimensional. Second,  combining $(\star)$ and Corollary \ref{cor:matrices} implies that at least one of $A_1$ or $A_2$, say $A_1$, has only $1$-dimensional extreme traces. 
But then $A_1$ must have at least two $1$-dimensional extreme traces, so we have a surjective unital $*$-homomorphism $A_1\rightarrow\mathbb C^2$.

We claim that $A_2$ also only has $1$-dimensional extreme traces.
Otherwise, $A_2$ admits a $k$-dimensional extreme trace, for some $k\geq 2$. This would give a surjective unital $*$-homomorphism $A_2\rightarrow\mathbb \mathbb M_k(\mathbb C)$, and thus a surjective unital $*$-homomorphism $A\rightarrow\mathbb C^2*\mathbb M_k(\mathbb C)$. Since $\text{T}(\mathbb C^2*\mathbb M_k(\mathbb C))$ has a closed face affinely homeomorphic to the Poulsen simplex by Lemma \ref{lem:C2}, this would contradict ($\star$).

Hence, both $A_1$ and $A_2$ each have at least two extreme traces, all of which are $1$-dimensional. If  $A_1$ or $A_2$ admits at least three extreme $1$-dimensional traces, we would get a surjective unital $*$-homomorphism $A\rightarrow\mathbb C^2*\mathbb C^3$. However, by Lemma \ref{lem:C2}, $\text{T}(\mathbb C^2*\mathbb C^3)$ has a face affinely homeomorphic to the Poulsen simplex, which would again contradict $(\star)$.

Altogether, we conclude that $A_i$ admits exactly two extreme traces, both of which are $1$-dimensional, for every $i\in\{1,2\}$. This implies that $A_{i,\mathrm{tr}}\cong \C^2$, for $i=1,2$. It then follows from Lemma~\ref{lem:freeprodtracial} that $\mathrm{T}(A)$ is affinely homeomorphic to $\mathrm{T}(\C^2\ast\C^2)$, i.e., condition (2) holds.

Finally, since by Remark \ref{trace space of C^2*C^2}, $\text{T}(\mathbb C^2*\mathbb C^2)$ does not admit a closed face which is affinely homeomorphic to the Poulsen simplex,  conditions (1) and (2) are mutually exclusive. This finishes the proof.
\end{proof}

\subsection{The diffuse hull} 
In what follows, for a compact convex set $C$ and a subset $X\subset C$, we denote by $\overline{\mathrm{conv}}(X)\subset C$ the closed convex hull of $X$. 

For a free product of non-trivial finite groups, excluding $\Z/2\Z\ast\Z/2\Z$, it was asked in \cite[Question~1.11]{OSV23} whether the closed convex hull of all the infinite dimensional characters is a closed face affinely homeomorphic to the Poulsen simplex. One can ask the same question in the broader context of $C^*$-algebras.

\begin{question}\label{q:Poulsenface}
    Let $A=A_1\ast A_2$ be the 
    free product of two unital, separable $C^*$-algebras $A_1$ and $A_2$ such that $\mathrm{T}(A_i)$ is non-empty and does not consist of a single $1$-dimensional trace, for $i=1,2$. Assume further that $\mathrm{T}(A)$ is not affinely homeomorphic to $\mathrm{T}(\C^2\ast\C^2)$. Is $\overline{\mathrm{conv}}(\partial_{\mathrm{e}} \mathrm{T}(A)\cap\mathrm{T}_\infty(A))$ a Poulsen simplex?
\end{question}

When $A$ satisfies the equivalent conditions of Theorem~\ref{Poulsen}, then $\overline{\mathrm{conv}}(\partial_{\mathrm{e}} \mathrm{T}(A)\cap\mathrm{T}_\infty(A)) = \mathrm{T}(A)$ is the Poulsen simplex by Theorem~\ref{Poulsen}, and hence the answer is yes. However, we answer this question in the negative for every such $A$ that does not satisfy the equivalent conditions of Theorem~\ref{Poulsen}. This in particular answers \cite[Question~1.11]{OSV23} in the negative. To this end, we introduce the following notation.

\begin{definition}
    Let $A$ be a $C^*$-algebra. We denote by $\partial_{(\mathrm{fd})}\mathrm{T}(A)\subset \partial_{\mathrm{e}}\mathrm{T}(A)$ the set of extreme finite dimensional traces of $A$ that are isolated in $\partial_{\mathrm{e}}\mathrm{T}(A)$. We denote by $\mathrm{T}_{\mathrm{diff}}(A)\subset \mathrm{T}(A)$ the set of traces whose GNS von Neumann algebra is diffuse, and call its closure $\overline{\mathrm{T}_{\mathrm{diff}}(A)}$ the \emph{diffuse hull} of $\mathrm{T}(A)$.
\end{definition}

Note that $\mathrm{T}_{\mathrm{diff}}(A)$ is a convex set,  hence the diffuse hull $\overline{\mathrm{T}_{\mathrm{diff}}(A)}$ is a compact and convex set. 
It is also easy to see that $\mathrm{T}_{\mathrm{diff}}(A)$ is a face of $\text{T}(A)$. 

Following the strategy from Section~\ref{sec:Poulsen}, and in particular the proofs of Lemma~\ref{2} and Lemma~\ref{lem:ifd}, one can easily establish the following lemma which describes the diffuse hull of the trace simplex of a free product. 
We leave the details of the proof to the interested reader.

\begin{lemma}\label{lem:diff}
    Let $A=A_1\ast A_2$ be the 
    free product of two unital, separable $C^*$-algebras $A_1$ and $A_2$ such that $\mathrm{T}(A_i)$ is non-empty and does not consist of a single $1$-dimensional trace, for $i=1,2$. Assume further that $\mathrm{T}(A)$ is not affinely homeomorphic to $\mathrm{T}(\C^2\ast\C^2)$. Then $$\overline{\mathrm{T}_{\mathrm{diff}}(A)} = \overline{\mathrm{conv}}(\partial_{\mathrm{e}} \mathrm{T}(A)\cap\mathrm{T}_\infty(A)) = \overline{\mathrm{conv}}(\partial_{\mathrm{e}}\mathrm{T}(A)\setminus \partial_{(\mathrm{fd})}\mathrm{T}(A)) =  \overline{\partial_{\mathrm{e}}\mathrm{T}(A)\cap\mathrm{T}_\infty(A)}.$$
    In particular, the extreme points of the diffuse hull of $\mathrm{T}(A)$ are  dense. 
\end{lemma}

To a certain extent, Lemma \ref{lem:diff} suggests a positive answer to  Question~\ref{q:Poulsenface}.
However, a closer look show that,  unless the diffuse hull is the whole simplex $\mathrm{T}(A)$ (which exactly occurs when the conditions of Theorem A hold), the diffuse hull is neither a face nor a simplex.

\begin{proposition}\label{noface}
Let $A=A_1\ast A_2$ be the 
free product of two unital, separable $C^*$-algebras $A_1$ and $A_2$ such that $\mathrm{T}(A_i)$ is non-empty and does not consist of a single $1$-dimensional trace, for $i=1,2$. Assume that $A_1$ admits an isolated extreme $1$-dimensional trace $\varphi_1$ and $A_2$ admits an isolated extreme $k$-dimensional trace $\varphi_2$, for some $k\in\mathbb N$. Assume further that $\mathrm{T}(A)$ is not affinely homeomorphic to $\mathrm{T}(\C^2\ast\C^2)$.

Then the diffuse hull of $\mathrm{T}(A)$ 
is neither a face of $\emph{T}(A)$  nor a Choquet simplex. 
\end{proposition}

The proof will use free products of traces. For $i=1,2$, let $\varphi_i\in\text{T}(A_i)$. Let  $(M_i,\tau_i)$ be a tracial von Neumann algebra and $\pi_i:A_i\rightarrow M_i$ be a tracial representation with $\varphi_i=\tau_i\circ\pi_i$. Let $M=M_1*M_2$ endowed with the free product trace $\tau=\tau_1*\tau_2$. Define $\pi:A\rightarrow M$ by letting $\pi_{|A_i}=\pi_i$, for $i=1,2$. We call $\varphi=\tau\circ\pi\in\text{T}(A)$ the free product of $\varphi_1$ and $\varphi_2$, and write $\varphi=\varphi_1*\varphi_2$.

\begin{proof} We will prove that $C=\overline{\mathrm{conv}}(\partial_{\mathrm{e}} \mathrm{T}(A)\cap\mathrm{T}_\infty(A))$ is neither a face of $\text{T}(A)$ nor a Choquet simplex. Since by Lemma \ref{lem:diff}, $C$ is equal to the diffuse hull of $\text{T}(A)$, this will prove the conclusion.

As in the proof of Theorem~\ref{freeprod}, let $\varphi\in\partial_{(\mathrm{fd})}\mathrm{T}(A)$ be such that $\varphi_{|A_i}=\varphi_i$, for $i=1,2$. We begin by noticing that  
$\vphi\notin C$. Indeed, assume that there exist diffuse traces $(\zeta_n)_n$  converging to $\vphi$. By Corollary~\ref{strongly_isolated}, we have $\mu_n(\{\vphi\})\to 1$, where $\mu_n$ is the unique probability measure on $\partial_{\mathrm{e}}\mathrm{T}(A)$ whose barycenter is $\zeta_n$. As $\vphi$ is finite dimensional, this contradicts the diffuseness of $\zeta_n$. 

We next define $\psi_i\in\text{T}(A_i)$, for $i=1,2$, by considering three cases:

\begin{enumerate}
\item  Assume that $k=1$ and $|\partial_{\mathrm{e}}\text{T}(A_1)|=|\partial_{\mathrm{e}}\text{T}(A_2)|=2$.
Since $\partial_{\mathrm{e}}\text{T}(A_1)$ and $\partial_{\mathrm{e}}\text{T}(A_2)$ cannot both consist of two $1$-dimensional traces, after possibly swapping indices, we may assume that there is $\psi_1\in \partial_{\mathrm{e}}\text{T}(A_1)$ of dimension at least $2$. Let $\psi_2\in\partial_{\mathrm{e}}\text{T}(A_2)\setminus\{\varphi_2\}$. 

\item  Assume that $k=1$, and $|\partial_{\mathrm{e}}\text{T}(A_1)|\geq 3$ or  $|\partial_{\mathrm{e}}\text{T}(A_2)|\geq 3$. After possibly swapping indices, we may assume that $|\partial_{\mathrm{e}}\text{T}(A_1)|\geq 3$. Let $\gamma,\delta$ be distinct elements of $\partial_{\mathrm{e}}\text{T}(A_1)\setminus\{\varphi_1\}$ and put $\psi_1=\frac{1}{2}(\gamma+\delta)\in\text{T}(A_1)$.  Let $\psi_2\in\partial_{\mathrm{e}}\text{T}(A_2)\setminus\{\varphi_2\}$.

\item Assume that $k\geq 2$. Let $\psi_1\in\partial_{\mathrm{e}}\text{T}(A_1)\setminus\{\varphi_1\}$ and put $\psi_2=\varphi_2$.
\end{enumerate}

We first show that $C$ is not a face of $\text{T}(A)$. 
Let $\psi\in\mathrm{T}(A)$ be a trace satisfying $\psi_{|A_i}=\psi_i$, for $i=1,2$. For instance, we can take $\psi$ to be the free product trace, $\psi_1*\psi_2$. Define $\chi=\frac{1}{2}(\varphi+\psi)\in\mathrm{T}(A)$. 

We claim that $\chi\in C$. Since  $\varphi\not\in C$, this claim will imply that $C$ is not a face of $\text{T}(A)$.
To prove the claim, let $\pi:A\rightarrow M$ be the tracial representation associated to $\chi$ and put $M_i=\pi(A_i)''$, for $i=1,2$.
Since $\chi_{|A_i}=\frac{1}{2}(\varphi_i+\psi_i)$, for $i=1,2$, it is easy to check that in any of the cases (1)-(3) we have $M_1,M_2\not=\mathbb C1$, $\dim(M_1)+\dim(M_2)\geq 5$ and $\text{e}(M_1)+\text{e}(M_2)\leq 1$. 
Thus, $\chi\in\mathrm{T}_{\mathrm{l}}(A)$, and combining Theorem \ref{perturbation} and Proposition \ref{prop:property A traces can be approxiamted} gives that indeed $\chi\in C$.

To prove the stronger statement that $C$ is not a Choquet simplex, we will need a more involved construction of traces on $A$. To this end, for $\alpha\in [\frac{1}{3},\frac{1}{2}]$ we define $\rho_\alpha\in\text{T}(A)$
by letting

\begin{itemize}
\item[(a)] $\rho_\alpha=\psi_1*(\alpha\varphi_2+(1-\alpha)\psi_2)$ in case (1).
\item[(b)] $\rho_\alpha=(\frac{1}{3}\varphi_1+\frac{2}{3}\psi_1)*(\alpha\varphi_2+(1-\alpha)\psi_2)=(\frac{1}{3}\varphi_1+\frac{1}{3}\gamma+\frac{1}{3}\delta)*(\alpha\varphi_2+(1-\alpha)\psi_2)$ in case (2).
\item[(c)] $\rho_\alpha=(\alpha\varphi_1+(1-\alpha)\psi_1)*\varphi_2$ in case (3).
\end{itemize}
We claim that \begin{equation}\label{extreme}\text{$\rho_\alpha\in\partial_{\text{e}}\text{T}(A)$, for every $\alpha\in [\frac{1}{3},\frac{1}{2}]$}.\end{equation}
We will derive this claim as a consequence of the following particular case of \cite[Theorem 4.1]{Ue11}: $(\star)$ if $N_1,N_2\not=\mathbb C$ are tracial von Neumann algebras with $\text{dim}(N_1)+\text{dim}(N_2)\geq 5$, then their free product $N_1*N_2$ is a II$_1$ factor if and only if for any projections $z_1\in\mathcal Z(N_1)$ and $z_2\in\mathcal Z(N_2)$  such that $N_1z_1\cong \mathbb M_{k_1}(\mathbb C)$ and $N_2z_2\cong\mathbb M_{k_2}(\mathbb C)$, for some $k_1,k_2\geq 1$, we have that $\frac{\tau(z_1)}{k_1^2}+\frac{\tau(z_2)}{k_2^2}\leq 1$.

Let $\alpha\in [\frac{1}{3},\frac{1}{2}]$ and denote by $N$ the GNS von Neumann algebra of $\rho_\alpha$. In case (1), we have that $N=N_1*N_2$, where $N_1\not=\mathbb C$ is a tracial factor and $N_2$ is a tracial von Neumann algebra whose minimal central projections have traces $\alpha,1-\alpha$. Since $\frac{1}{4}+\alpha\leq 1$ and $\frac{1}{4}+(1-\alpha)\leq 1$, $(\star)$ implies that $N$ is a II$_1$ factor.  Similarly,  in case (3), using that $\varphi_2\in\partial_{\text{e}}\text{T}(A_2)$ is $k$-dimensional for some $k\geq 2$, it follows that $N$ is a II$_1$ factor.
Finally, in case (2), $N=N_1*N_2$, where $N_1$ and $N_2$ are tracial von Neumann algebra whose minimal central projections have traces equal to $\frac{1}{3}$ and $\alpha,1-\alpha$, respectively. Since $\frac{1}{3}+\alpha\leq 1$ and $\frac{1}{3}+(1-\alpha)\leq 1$, $(\star)$ again implies that $N$ is a II$_1$ factor. Altogether, this proves that $N$ is a II$_1$ factor in all cases (1)-(3), and thus that $\rho_\alpha\in\partial_{\text{e}}\text{T}(A)$.

Next, for $\alpha\in [\frac{1}{3},\frac{1}{2}]$, we put $\beta_\alpha=\frac{\frac{1}{2}-\alpha}{1-\alpha}\in [0,\frac{1}{4}]$. We claim that
\begin{equation}\label{convexcomb}
\text{$\beta_\alpha \varphi+(1-\beta_\alpha)\rho_\alpha\in C$, for every $\alpha\in [\frac{1}{3},\frac{1}{2}]$.}
\end{equation}
To prove this claim we note that $\beta_\alpha+(1-\beta_\alpha)\alpha=\frac{1}{2}$ and therefore
\begin{itemize}
\item $(\beta_\alpha\varphi+(1-\beta_\alpha)\rho_\alpha)_{|A_1}=\beta_\alpha\varphi_1+(1-\beta_\alpha)\psi_1$  and \\ $(\beta_\alpha\varphi+(1-\beta_\alpha)\rho_\alpha)_{|A_2}=\frac{1}{2}\varphi_2+\frac{1}{2}\psi_2$ in case (1).
\item $(\beta_\alpha\varphi+(1-\beta_\alpha)\rho_\alpha)_{|A_1}=\frac{2\beta_\alpha+1}{3}\varphi_1+\frac{1-\beta_\alpha}{3}\gamma+\frac{1-\beta_\alpha}{3}\delta$ and \\ $(\beta_\alpha\varphi+(1-\beta_\alpha)\rho_\alpha)_{|A_2}=\frac{1}{2}\varphi_2+\frac{1}{2}\psi_2$  in case (2).
\item $(\beta_\alpha\varphi+(1-\beta_\alpha)\rho_\alpha)_{|A_1}=\frac{1}{2}\varphi_1+\frac{1}{2}\psi_1$  and $(\beta_\alpha\varphi+(1-\beta_\alpha)\rho_\alpha)_{|A_2}=\varphi_2$ in case (3).
\end{itemize}

Since $\psi_1\in\partial_{\text{e}}\text{T}(A_1)$ has dimension at least $2$ in case (1), $\frac{2\beta_\alpha+1}{3}\leq\frac{1}{2}$ and 
$\varphi_2\in\partial_{\text{e}}\text{T}(A_2)$ has dimension at least $2$ in case (3), it follows that in each of the cases (1)-(3) we have that $\beta_\alpha \varphi+(1-\beta_\alpha)\rho_\alpha\in \text{T}_{\text{l}}(A)$. Combining Theorem \ref{perturbation} and Proposition \ref{prop:property A traces can be approxiamted} gives that $\beta_\alpha \varphi+(1-\beta_\alpha)\rho_\alpha\in  \overline{\partial_{\mathrm{e}}\mathrm{T}(A)\cap\mathrm{T}_\infty(A)}\subset C$.

To show that $C$ is not a Choquet simplex, assume by contradiction that $C$ is a  Choquet simplex. Choose $\alpha\not=\alpha'\in [\frac{1}{3},\frac{1}{2})$. By \eqref{extreme}, $\varphi,\rho_\alpha,\rho_{\alpha'}\in\partial_{\text{e}}\text{T}(A)$ are distinct extreme traces. Thus,
\[  
    \Delta  = \mathrm{conv}(\varphi,\rho_\alpha,\rho_{\alpha'})
\]
is a closed face of $\mathrm{T}(A)$ which is isomorphic to the $2$-dimensional simplex. 
Since $C$ is assumed to be a simplex, the convex set $\Delta\cap C$ is also a simplex.

Put $I=\text{conv}(\varphi,\rho_\alpha)$. 
By \eqref{extreme}, we get that $\rho_\alpha\in C$, hence $\rho_\alpha\in I\cap C$.
Since $\alpha<\frac{1}{2}$, we get that $\beta_\alpha\not=0$ and hence $\beta_\alpha\varphi+(1-\beta_\alpha)\rho_\alpha\not=\rho_\alpha$. 
Hence, \eqref{convexcomb} implies that $\beta_\alpha\varphi+(1-\beta_\alpha)\rho_\alpha\in (I\cap C)\setminus\{\rho_\alpha\}$.
Since $\varphi\not\in C$, we conclude that
$I\cap C$ is a non-degenerate interval which does not contain $\varphi$. Similarly, if $I'=\text{conv}(\varphi,\rho_{\alpha'})$, then $I'\cap C$  is a non-degenerate interval which does not contain $\varphi$. On the other hand, $ \Delta \cap C$ contains the quadrilateral prescribed by the endpoints of $ I\cap C$ and $ I'\cap C$. However, any proper subset of $\Delta$ which contains such a quadrilateral cannot be a simplex. This gives the desired contradiction and finishes the proof. 
\end{proof}

\end{document}